\documentclass[a4paper, 11pt]{amsart}
\usepackage{amsfonts, amsthm, amssymb, amsmath, stmaryrd}
\usepackage{mathrsfs,array}
\usepackage{eucal,fullpage,times,color,enumerate,accents}
\usepackage[all]{xy}
\usepackage{url}
\usepackage{enumitem}

\input xy
\xyoption{all}

\newcommand{\calX}{{\mathcal{X}}}

\newcommand{\calO}{{\mathcal{O}}}

\newcommand{\calU}{{\mathcal{U}}}

\newcommand{\calD}{{\mathcal{D}}}

\newcommand{\calF}{\mathcal{F}}

\newcommand{\calH}{\mathcal{H}}
\newcommand{\calC}{\mathcal{C}}
\newcommand{\calS}{\mathcal{S}}

\newcommand{\A}{\mathbf{A}}
\newcommand{\G}{\mathbf{G}}
\newcommand{\Z}{\mathbf{Z}}
\newcommand{\N}{\mathbf{N}}
\newcommand{\C}{\mathbf{C}}
\newcommand{\F}{\mathbf{F}}
\newcommand{\Q}{\mathbf{Q}}
\renewcommand{\P}{\mathbf{P}}

\newcommand{\Spec}{{\mathrm{Spec}}}
\newcommand{\Hens}{\mathrm{Hens}}
\newcommand{\Gal}{\mathrm{Gal}}
\newcommand{\Hom}{\mathrm{Hom}}
\newcommand{\Isom}{\mathrm{Isom}}
\newcommand{\Fun}{\mathrm{Fun}}
\newcommand{\Map}{\mathrm{Map}}

\newcommand{\Aut}{\mathrm{Aut}}

\newcommand{\Shv}{\mathrm{Shv}}
\newcommand{\Ind}{\mathrm{Ind}}
\newcommand{\Pro}{\mathrm{Pro}}
\newcommand{\GL}{\mathrm{GL}}
\newcommand{\PGL}{\mathrm{PGL}}
\newcommand{\Ext}{\mathrm{Ext}}
\newcommand{\Tor}{\mathrm{Tor}}

\newcommand{\Sch}{\mathrm{Sch}}

\newcommand{\Mod}{\mathrm{Mod}}
\newcommand{\Spc}{\mathrm{Spc}}

\newcommand{\Set}{\mathrm{Set}}
\newcommand{\Ab}{\mathrm{Ab}}

\newcommand{\ord}{\mathrm{ord}}

\newcommand{\R}{\mathrm{R}}

\newcommand{\Ch}{\mathrm{Ch}}

\newcommand{\et}{\mathrm{\acute{e}t}}
\newcommand{\proet}{ {\mathrm{pro}\et} }
\newcommand{\cons}{\mathrm{cons}}

\newcommand{\aff}{\mathrm{aff}}

\newcommand{\im}{\mathrm{im}}
\newcommand{\id}{\mathrm{id}}
\newcommand{\ev}{\mathrm{ev}}

\newcommand{\pr}{\mathrm{pr}}

\newcommand{\dR}{\mathrm{dR}}

\newcommand{\perf}{\mathrm{perf}}

\newcommand{\can}{\mathrm{can}}

\newcommand{\comp}{\mathrm{comp}}

\newcommand{\opp}{\mathrm{op}}

\renewcommand{\ker}{\mathrm{ker}}
\newcommand{\cok}{\mathrm{cok}}
\newcommand{\naive}{\mathrm{naive}}

\newcommand{\cart}{\mathrm{cart}}
\newcommand{\cont}{\mathrm{cont}}
\newcommand{\Loc}{\mathrm{Loc}}
\newcommand{\Cov}{\mathrm{Cov}}
\newcommand{\Rep}{\mathrm{Rep}}
\newcommand{\Cons}{\mathrm{Cons}}
\newcommand{\Stab}{\mathrm{Stab}}

\newcommand{\fram}{\mathfrak{m}}

\newcommand{\comment}[1]{}

\newcommand{\cosimp}[3]{\xymatrix@1{#1 \ar@<.4ex>[r] \ar@<-.4ex>[r] & {\ }#2 \ar@<0.8ex>[r] \ar[r] \ar@<-.8ex>[r] & {\ } #3 \ar@<1.2ex>[r] \ar@<.4ex>[r] \ar@<-.4ex>[r] \ar@<-1.2ex>[r] & \cdots }}

\DeclareMathOperator*{\colim}{colim}

\newcommand{\equalizer}[2]{\xymatrix@1{#1 \ar@<.4ex>[r] \ar@<-0.4ex>[r] & {\ } #2}}

\newcommand{\adjunction}[4]{\xymatrix@1{#1{\ } \ar@<-0.3ex>[r]_{ {\scriptstyle #2}} & {\ } #3 \ar@<-0.3ex>[l]_{ {\scriptstyle #4}}}}

\begin{document}
\bibliographystyle{alpha}
\newtheorem{theorem}{Theorem}[subsection]
\newtheorem*{theorem*}{Theorem}
\newtheorem*{condition*}{Condition}
\newtheorem*{definition*}{Definition}
\newtheorem*{corollary*}{Corollary}
\newtheorem{proposition}[theorem]{Proposition}
\newtheorem{lemma}[theorem]{Lemma}
\newtheorem{corollary}[theorem]{Corollary}
\newtheorem{claim}[theorem]{Claim}
\newtheorem{definition}[theorem]{Definition}

\theoremstyle{definition}
\newtheorem{question}[theorem]{Question}
\newtheorem{remark}[theorem]{Remark}
\newtheorem{guess}[theorem]{Guess}
\newtheorem{example}[theorem]{Example}
\newtheorem{condition}[theorem]{Condition}
\newtheorem{warning}[theorem]{Warning}
\newtheorem{notation}[theorem]{Notation}
\newtheorem{construction}[theorem]{Construction}

\title{The pro-\'etale topology for schemes}
\author{Bhargav Bhatt and Peter Scholze}

\begin{abstract}
We give a new definition of the derived category of constructible $\overline{\Q}_\ell$-sheaves on a scheme, which is as simple as the geometric intuition behind them. Moreover, we define a refined fundamental group of schemes, which is large enough to see all lisse $\overline{\Q}_\ell$-sheaves, even on non-normal schemes. To accomplish these tasks, we define and study the pro-\'etale topology, which is a Grothendieck topology on schemes that is closely related to the \'etale topology, and yet better suited for infinite constructions typically encountered in $\ell$-adic cohomology. An essential foundational result is that this site is locally contractible in a well-defined sense.
\end{abstract}
\dedicatory{To G\'erard Laumon, with respect and admiration}
\maketitle
\tableofcontents
\pagebreak

\section{Introduction}

\renewcommand{\thesubsection}{\arabic{section}}

Let $X$ be a variety over an algebraically closed field $k$. The \'etale cohomology groups $H^i(X_\et,\overline{\Q}_\ell)$, where $\ell$ is a prime different from the characteristic of $k$, are of fundamental importance in algebraic geometry. Unfortunately, the standard definition of these groups is somewhat indirect. Indeed, contrary to what the notation suggests, these groups are not obtained as the cohomology of a sheaf $\overline{\Q}_\ell$ on the \'etale site $X_\et$. The \'etale site gives the correct answer only with torsion coefficients, so the correct definition is
\[
H^i(X_\et,\overline{\Q}_\ell) := (\varprojlim_n H^i(X_\et,\Z/\ell^n\Z) )\otimes_{\Z_\ell} \overline{\Q}_\ell\ .
\]
In this simple situation, this technical point is often unproblematic\footnote{It becomes a problem as soon as one relaxes the assumptions on $k$, though. For example, even for $k = \Q$, this definition is not correct: there is no Hochschild-Serre spectral sequence linking these naively defined cohomology groups of $X$ with those of $X_{\overline{k}}$.  One must account for the higher derived functors of inverse limits to get a theory linked to the geometry of $X_{\overline{k}}$,  see \cite{JannsenContsCoh}.}. However, even here, it takes effort to construct a natural commutative differential graded $\overline{\Q}_\ell$-algebra giving rise to these cohomology groups. This so-called $\overline{\Q}_\ell$-homotopy type was constructed by Deligne in \cite{DeligneWeil2}, using certain subtle integral aspects of homotopy theory due independently to Miller \cite{MillerIntegraldeRham} and Grothendieck. 

For more sophisticated applications, however, it is important to work in a relative setup (i.e., study constructible sheaves), and keep track of the objects in the derived category, instead of merely the cohomology groups. In other words, one wants a well-behaved derived category $D^b_c(X,\overline{\Q}_\ell)$ of constructible $\overline{\Q}_\ell$-sheaves. Deligne, \cite{DeligneWeil2}, and in greater generality Ekedahl, \cite{Ekedahl}, showed that it is possible to define such a category along the lines of the definition of $H^i(X_\et,\overline{\Q}_\ell)$. Essentially, one replaces $H^i(X_\et,\Z/\ell^n\Z)$ with the derived category $D^b_c(X,\Z/\ell^n\Z)$ of constructible $\Z/\ell^n\Z$-sheaves, and then performs all operations on the level of categories:\footnote{In fact, Ekedahl only defines the derived category of constructible $\Z_\ell$-sheaves, not performing the final $\otimes_{\Z_\ell} \overline{\Q}_\ell$-step.}
\[
D^b_c(X,\overline{\Q}_\ell) := (\varprojlim_n D^b_c(X,\Z/\ell^n\Z))\otimes_{\Z_\ell} \overline{\Q}_\ell\ .
\]
Needless to say, this presentation is oversimplified, and veils substantial technical difficulties.

Nonetheless, in daily life, one pretends (without getting into much trouble) that $D^b_c(X,\overline{\Q}_\ell)$ is simply the full subcategory of some hypothetical derived category $D(X,\overline{\Q}_\ell)$ of all $\overline{\Q}_\ell$-sheaves spanned by those bounded complexes whose cohomology sheaves are locally constant along a stratification. Our goal in this paper to justify this intuition, by showing that the following definitions recover the classical notions. To state them, we need the pro-\'etale site $X_\proet$, which is introduced below. For any topological space $T$, one has a `constant' sheaf on $X_\proet$ associated with $T$; in particular, there is a sheaf of (abstract) rings $\overline{\Q}_\ell$ on $X_\proet$ associated with the topological ring $\overline{\Q}_\ell$.

\begin{definition} Let $X$ be a scheme whose underlying topological space is noetherian.
	\begin{enumerate}
	\item A sheaf $L$ of $\overline{\Q}_\ell$-modules on $X_\proet$ is {\em lisse} if it is locally free of finite rank.
	\item A sheaf $C$ of $\overline{\Q}_\ell$-modules on $X_\proet$ is {\em constructible} if there is a finite stratification $\{X_i\to X\}$ into locally closed subsets $X_i\subset X$ such that $C|_{X_i}$ is lisse.
	\item An object $K\in D(X_\proet,\overline{\Q}_\ell)$ is {\em constructible} if it is bounded, and all cohomology sheaves are constructible. Let $D^b_c(X,\overline{\Q}_\ell)\subset D(X_\proet,\overline{\Q}_\ell)$ be the corresponding full triangulated subcategory.
\end{enumerate}
\end{definition}

The formalism of the six functors is easily described in this setup. In particular, in the setup above, {\it with the naive interpretation of the right-hand side}, one has
\[
H^i(X_\et,\overline{\Q}_\ell) = H^i(X_\proet,\overline{\Q}_\ell)\ ;
\]
for general $X$, one recovers Jannsen's continuous \'etale cohomology, \cite{JannsenContsCoh}. Similarly, the complex $\R\Gamma(X_\proet,\overline{\Q}_\ell)$ is obtained by literally applying the derived functor $\R\Gamma(X_\proet,-)$ to a sheaf of $\Q$-algebras, and hence naturally has the structure of a commutative differential graded algebra by general nonsense (see \cite[\S 2]{OlssonnonabelianpHT} for example); this gives a direct construction of the $\overline{\Q}_\ell$-homotopy type in complete generality.

A version of the pro-\'etale site was defined in \cite{ScholzepHT} in the context of adic spaces. The definition given there was somewhat artificial, mostly because non-noetherian adic spaces are not in general well-behaved. This is not a concern in the world of schemes, so one can give a very simple and natural definition of $X_\proet$. Until further notice, $X$ is allowed to be an arbitrary scheme.

\begin{definition}\ 
	\begin{enumerate}
	\item A map $f: Y\to X$ of schemes is {\em weakly \'etale} if $f$ is flat and $\Delta_f: Y\to Y\times_X Y$ is flat.
	\item The {\em pro-\'etale site} $X_\proet$ is the site of weakly \'etale $X$-schemes, with covers given by fpqc covers.
\end{enumerate}
\end{definition}

Any map between weakly \'etale $X$-schemes is itself weakly \'etale, and the resulting topos has good categorical properties, like coherence (if $X$ is qcqs) and (hence) existence of enough points.  For this definition to be useful, however, we need to control the class of weakly \'etale morphisms. In this regard, we prove the following theorem.

\begin{theorem} Let $f: A\to B$ be a map of rings.
\begin{enumerate}
\item $f$ is \'etale if and only if $f$ is weakly \'etale and finitely presented.
\item If $f$ is ind-\'etale, i.e. $B$ is a filtered colimit of \'etale $A$-algebras, then $f$ is weakly \'etale.
\item If $f$ is weakly \'etale, then there exists a faithfully flat ind-\'etale $g: B\to C$ such that $g\circ f$ is ind-\'etale.
\end{enumerate}
\end{theorem}

In other words, for a ring $A$, the sites defined by weakly \'etale $A$-algebras and by ind-\'etale $A$-algebras are equivalent, which justifies the name pro-\'etale site for the site $X_\proet$ defined above. We prefer using weakly \'etale morphisms to define $X_\proet$ as the property of being weakly \'etale is clearly \'etale local on the source and target, while that of being ind-\'etale is not even Zariski local on the target.

One might worry that the pro-\'etale site is huge in an uncontrolled way (e.g., covers might be too large, introducing set-theoretic problems). However, this does not happen. To see this, we need a definition:

\begin{definition} 
	An affine scheme $U$ is {\em w-contractible} if any faithfully flat weakly \'etale map $V \to U$ admits a section. 
\end{definition}

A w-contractible object $U \in X_\proet$ is somewhat analogous to a point in the topos theoretic sense: the functor $\Gamma(U,-)$ is exact and commutes with all limits, rather than colimits. In fact, a geometric point of $X$ defines a w-contractible object in $X_\proet$ via the strict henselisation. However, there are many more w-contractible objects, which is the key to the control alluded to above:

\begin{theorem} 
	\label{thm:IntroWcontractibleCover}
	Any scheme $X$ admits a cover in $X_\proet$ by w-contractible affine schemes.
\end{theorem}

Despite the analogy between w-contractible objects and points, Theorem \ref{thm:IntroWcontractibleCover} has stronger consequences than the mere existence of points. For example, the inverse limit functor on systems
\[
\ldots\to F_n\to F_{n-1}\to \ldots \to F_1\to F_0
\]
of sheaves on $X_\proet$ is well-behaved, the derived category of abelian sheaves on $X_\proet$ is left-complete and compactly generated, unbounded cohomological descent holds in the derived category, and Postnikov towers converge in the hypercomplete $\infty$-topos associated with $X_\proet$. This shows that the pro-\'etale site is useful even when working with torsion coefficients, as the derived category of $X_\et$ is left-complete (and unbounded cohomological descent holds) only under finiteness assumptions on the cohomological dimension of $X$, cf.  \cite{LaszloOlsson}.

We note that one can `cut off' $X_\proet$ by only allowing weakly \'etale $X$-schemes $Y$ of cardinality $<\kappa$ for some uncountable strong limit cardinal $\kappa > |X|$,  and all results above, especially the existence of w-contractible covers, remain true. In particular, the resulting truncated site $X_\proet$ forms a set, rather than a proper class, so we can avoid universes in this paper.

Let us explain the local structure of a scheme in the pro-\'etale site. 

\begin{definition}\ 
	\begin{enumerate}
	\item A ring $A$ is {\em w-local} if the subset $(\Spec A)^c\subset \Spec A$ of closed points is closed, and any connected component of $\Spec A$ has a unique closed point.
\item A map $f: A\to B$ of w-local rings is {\em w-local} if $\Spec f: \Spec B\to \Spec A$ maps closed points to closed points.
\end{enumerate}
\end{definition}

The next result shows that every scheme is covered by w-local affines in the pro-Zariski topology, and hence in the pro-\'etale topology. In particular, as noetherian schemes have finitely many connected components,  this shows that non-noetherian schemes are unavoidable when studying $X_\proet$, even for $X$ noetherian.

\begin{theorem} The inclusion of the category of w-local rings with w-local maps in the category of all rings admits a left adjoint $A\mapsto A^Z$. The unit $A\to A^Z$ of the adjunction is faithfully flat and an ind-(Zariski localisation), so $\Spec A^Z \to \Spec A$ is a cover in $\Spec(A)_\proet$. Moreover,  the subset $(\Spec A^Z)^c\subset \Spec A^Z$ of closed points maps homeomorphically to $\Spec A$, equipped with its constructible topology.
\end{theorem}

In other words, $\Spec A^Z$ is roughly the disjoint union of the local rings of $A$. However, the union is not exactly disjoint; rather, the set of connected components $\pi_0(\Spec A^Z)$ is naturally a profinite set, which is $\Spec A$ with its constructible topology. Thus, the study of w-local rings splits into the study of its local rings at closed points, and the study of profinite sets. It turns out in practice that these two aspects interact little. In particular, this leads to the following characterization of w-contractible schemes.

\begin{theorem} An affine scheme $X=\Spec A$ is w-contractible if and only if $A$ is w-local, all local rings at closed points are strictly henselian, and $\pi_0(X)$ is extremally disconnected.
\end{theorem}

Recall that a profinite set $S$ is extremally disconnected if the closure of any open subset $U\subset S$ is still open. By a theorem of Gleason, $S$ is extremally disconnected if and only if $S$ is projective in the category of compact Hausdorff spaces, i.e., any surjective map $T\to S$ from a compact Hausdorff space $T$ admits a section. In particular, the Stone-Cech compactification of any discrete set is extremally disconnected, which proves the existence of enough such spaces. Using this construction, if $A$ is w-local, it is relatively easy to construct a faithfully flat ind-\'etale $A$-algebra $B$ satisfying the conditions of the theorem, which proves the existence of enough w-contractible schemes.

As a final topic, we study the fundamental group. In SGA1,  a profinite group $\pi_1^\et(X,x)$ is defined for any connected scheme $X$ with a geometric point $x$. It has the property that the category of lisse $\Z_\ell$-sheaves on $X$ is equivalent to the category of continuous representations of $\pi_1^\et(X,x)$ on finite free $\Z_\ell$-modules. However, the analogue for lisse $\Q_\ell$-sheaves fails (unless $X$ is geometrically unibranch) as $\Q_\ell$-local systems admit $\Z_\ell$-lattices only \'etale locally. For example, if $X$ is $\P^1$  with $0$ and $\infty$ identified (over an algebraically closed field), then $X$ admits a cover $f: Y\to X$ where $Y$ is an infinite chain of $\P^1$'s. One can descend the trivial $\Q_\ell$-local system on $Y$ to $X$ by identifying the fibres at $0$ and $\infty$ using any unit in $\Q_\ell$, e.g. $\ell\in \Q_\ell^\times$. However, representations of $\pi_1^\et(X,x) = \hat{\Z}$ with values in $\GL_1(\Q_\ell)$ will have image in $\GL_1(\Z_\ell)$ by compactness. This suggests that the 'true' $\pi_1$ of $X$ should be $\Z\subset \hat{\Z} = \pi_1^\et(X,x)$. In fact, in SGA3 X6, a prodiscrete group $\pi_1^{\mathrm{SGA3}}(X,x)$ is defined, which gives the desired answer in this example. Its defining property is that $\Hom(\pi_1^{\mathrm{SGA3}}(X,x),\Gamma)$ is in bijection with $\Gamma$-torsors trivialized at $x$, for any discrete group $\Gamma$. However, in general, $\pi_1^{\mathrm{SGA3}}(X,x)$ is still too small to detect all $\Q_\ell$-local systems through its finite dimensional continuous $\Q_\ell$-representations: the failure is visible already for $X$ a high-genus curve with two points identified (this example is due to Deligne, and recalled in  Example \ref{ex:delignenodalcurve}). 

We circumvent the issues raised above by working with a larger category of ``coverings'' than the ones used in constructing $\pi_1^\et(X,x)$ and $\pi_1^{\mathrm{SGA3}}(X,x)$. To recover groups from such categories, we study some general infinite Galois theory. The formalism leads to the following kind of groups.

\begin{definition} 
	A topological group $G$ is called a {\em Noohi group} if $G$ is complete, and admits a basis of open neighborhoods of $1$ given by open subgroups.
\end{definition}

The word ``complete'' above refers to the two-sided uniform structure on $G$ determined by its open subgroups. For example, locally profinite groups, such as $\GL_n(\Q_\ell)$, are Noohi groups. Somewhat more surprisingly, $\GL_n(\overline{\Q}_\ell)$ is also a Noohi group. The main result is:

\begin{theorem} Let $X$ be a connected scheme whose underlying topological space is locally noetherian. The following categories are equivalent.
\begin{enumerate}
\item The category $\Loc_X$ of sheaves on $X_\proet$ which are locally constant.
\item The category $\Cov_X$ of \'etale $X$-schemes $Y$ which satisfy the valuative criterion of properness.
\end{enumerate}
For any geometric point $x$ of $X$, the infinite Galois theory formalism applies to $\Loc_X$ equipped with the fibre functor at $x$, giving rise to a Noohi group $\pi_1^\proet(X,x)$. The pro-finite completion of $\pi_1^\proet(X,x)$ is $\pi_1^\et(X,x)$, and the pro-discrete completion of $\pi_1^\proet(X,x)$ is $\pi_1^{\mathrm{SGA3}}(X,x)$. Moreover, $\Q_\ell$-local systems on $X$ are equivalent to continuous representations of $\pi_1^\proet(X,x)$ on finite-dimensional $\Q_\ell$-vector spaces, and similarly for $\Q_\ell$ replaced by $\overline{\Q}_\ell$.
\end{theorem}

Informally, the difference between $\pi_1^\proet(X,x)$ and the classical fundamental groups stems from the existence of pro-\'etale locally constant sheaves that are not \'etale locally constant. This difference manifests itself mathematically in the lack of enough Galois objects, i.e., $\pi_1^\proet(X,x)$ does not have enough open normal subgroups (and thus is not prodiscrete). It is important to note that the construction of $\pi_1^\proet(X,x)$ is not completely formal. Indeed, as with $\pi_1^{\mathrm{SGA3}}(X,x)$, it is not clear a priori that $\pi_1^\proet(X,x)$ contains even a single non-identity element: a cofiltered limit of discrete groups along surjective transition maps can be the trivial group. Thus, one must directly construct elements to show $\pi_1^\proet(X,x)$ is big enough. This is done by choosing actual paths on $X$, thus reuniting the classical point of view from topology with the abstract approach of SGA1.

Finally, let us give a short summary of the different sections. In Section \ref{sec:LocStruc}, we study w-local rings and the like. In Section \ref{sec:Replete}, we study a general topos-theoretic notion (namely, repleteness) which implies left-completeness of the derived category etc. . We also include some discussions on complete sheaves, which are again well-behaved under the assumption of repleteness. In Section \ref{sec:ProEtale}, we introduce the pro-\'etale site, and study its basic properties. The relation with the \'etale site is studied in detail in Section \ref{sec:Etale}. In Section \ref{sec:Constructible}, we introduce constructible sheaves (recalling first the theory for torsion coefficients on the \'etale site), showing that for schemes whose underlying topological space is noetherian, one gets the very simple definition stated above. Finally, in Section \ref{sec:Pi1}, we define the pro-\'etale fundamental group.

{\bf Acknowledgments.} The vague idea that such a formalism should exist was in the air since the paper \cite{ScholzepHT}, and the second-named author received constant encouragement from Michael Rapoport, Luc Illusie and many others to work this out. Martin Olsson's question on the direct construction of the $\Q_\ell$-homotopy type led to the birth of this collaboration, which soon led to much finer results than initially expected. Ofer Gabber suggested that weakly \'etale morphisms could be related to ind-\'etale morphisms. Johan de Jong lectured on some parts of this paper in Stockholm, and provided numerous useful and enlightening comments. Conversations with Brian Conrad also clarified some arguments.

H\'el\`ene Esnault urged us to think about fundamental groups of non-normal schemes from the perspective of the pro-\'etale topology, which led to \S \ref{sec:Pi1}. Moreover, Pierre Deligne generously shared his notes on fundamental groups, which had an important influence on the material in \S \ref{sec:Pi1}, especially in relation to Noohi groups and abstract infinite Galois theory. Deligne's results were slightly weaker: in the language introduced in \S \ref{ss:infgaloistheory}, he first proves that any countably generated (in a suitable sense) infinite Galois category is automatically tame, and then specializes this result to schemes to obtain, using purely abstract arguments, a pro-(Noohi group) from a certain category of ``coverings'' that turns out to be equivalent to $\Cov_X$;  here the pro-structure is dual to the ind-structure describing this category of coverings as a filtered colimit of countably generated infinite Galois categories. After we realized that this pro-group is realized by its limit by using geometric paths, Gabber explained to us his different perspective on fundamental groups, which we explain in Remark \ref{rem:GabberPi1} below.

This work was done while Bhargav Bhatt was supported by NSF grants DMS 1340424 and DMS 1128155, and Peter Scholze was a Clay Research Fellow.

\newpage

\section{Local structure}\label{sec:LocStruc}

\renewcommand{\thesubsection}{\arabic{section}.\arabic{subsection}}

The goal of this section is to study some algebra relevant to the pro-\'etale topology. Specifically, we show: (a) weakly \'etale and pro-\'etale maps define the same Grothendieck topology on rings in \S \ref{subsec:weaklyetaleproetale}, and (b) this Grothendieck topology has enough ``weakly contractible'' objects in \S \ref{subsec:localcontractibility}.

\subsection{Spectral spaces}
\label{subsec:spectralcwlocal}

Let $\calS$ be the category of spectral spaces with spectral maps, and let $\calS_f \subset \calS$ be the full subcategory of finite spectral spaces (= finite $T_0$ spaces), so $\calS = \Pro(\calS_f)$, cf. \cite{HochsterSpectral}. Our main goal is to show that each $X \in \calS$ admits a pro-(open cover) $X^Z \to X$ such that $X^Z$ admits no further non-split open covers. This goal is eventually realized in Lemma \ref{lem:cwlocaladjoint}. Before constructing $X^Z$, however, we introduce and study the subcategory of $\calS$ where spaces of the form $X^Z$ live:

\begin{definition}
	\label{def:cwlocal}
A spectral space $X$ is {\em w-local} if it satisfies:
\begin{enumerate}
	\item All open covers split, i.e., for every open cover $\{U_i \hookrightarrow X\}$, the map $\sqcup_i U_i \to X$ has a section.
	\item The subspace $X^c \subset X$ of closed points is closed.
\end{enumerate}
A map $f:X \to Y$ of w-local spaces is w-local if $f$ is spectral and $f(X^c) \subset Y^c$. Let $i:\calS^{wl} \hookrightarrow \calS$ be the subcategory of w-local spaces with w-local maps. 
\end{definition}

The first condition in Definition \ref{def:cwlocal} is obviously necessary for the promised application. The second condition turns out to be particularly convenient for applications.

\begin{example}
Any profinite set is a w-local space. Any local scheme has a w-local topological space. The collection of w-local spaces is closed under finite disjoint unions.
\end{example}

The property of w-locality passes to closed subspaces:

\begin{lemma}
	\label{lem:closedincwlocal}
	If $X \in \calS^{wl}$, and $Z \subset X$ is closed, then $Z \in \calS^{wl}$.
\end{lemma}
\begin{proof}
Open covers of $Z$ split as any open cover of $Z$ extends to one of $X$ (by extending opens and adding $X - Z$). Moreover, it is clear that $Z^c = X^c \cap Z$, so the claim follows.
\end{proof}

Recall that the inclusion $\Pro(\Set_f) \subset \Pro(\calS_f) = \calS$ has a left-adjoint $X \mapsto \pi_0(X)$, i.e.,  the counit $X \to \pi_0(X)$ is the universal spectral map from $X$ to a profinite set. Given a cofiltered presentation $X = \lim_i X_i$ with $X_i \in \calS_f$, we have $\pi_0(X) = \lim_i \pi_0(X_i)$. We use this to give an intrinsic description of w-local spaces:

\begin{lemma}
	\label{lem:cwlocalcharacterize}
	A spectral space $X$ is w-local if and only if $X^c \subset X$ is closed, and every connected component of $X$ has a unique closed point. For such $X$, the composition $X^c \to X \to \pi_0(X)$ is a homeomorphism.
\end{lemma}
\begin{proof}
	The second part follows immediately from the first as $X^c$ is profinite when $X$ is w-local. For the first, assume that $X$ is w-local; it suffices to show that each connected component has a unique closed point. Then Lemma \ref{lem:closedincwlocal} shows that any connected component is also w-local, so we may assume $X$ is connected. If $X$ has two distinct closed points $x_1,x_2 \in X^c$, then the open cover $(X - \{x_1\}) \sqcup (X - \{x_2\}) \to X$ has no section, which contradicts w-locality. 
	
	Conversely, assume $X^c \subset X$ is closed, and that each connected component has a unique closed point. Then $X^c$ is profinite, and hence $X^c \to \pi_0(X)$ is a homeomorphism. Now fix a finite open cover $\{U_i \hookrightarrow X\}$ with $U_i$ quasicompact. We must show that $\pi:Y := \sqcup_i U_i \to X$ has a section. As $X^c$ is profinite, there is a map $s:X^c \to Y$ lifting the inclusion $X^c \hookrightarrow X$. Let $Z \subset \pi_0(Y)$ be the image of the composite $X^c \stackrel{s}{\to} Y \to \pi_0(Y)$. Then $Z$ is a closed subset of $\pi_0(Y)$, and the canonical maps $X^c \to Z \to \pi_0(X)$ are all homeomorphisms. In particular $Z \hookrightarrow \pi_0(Y)$ is a pro-(open immersion). Let $Y' := Y \times_{\pi_0(Y)} Z \hookrightarrow Y$ be the inverse image. Then $Y'$ is a spectral space with $\pi_0(Y') = Z$. The map $Y' \to Y$ is pro-(open immersion), so the map $\phi:Y' \to X$ is pro-open. One checks from the construction $\phi$ induces a homeomorphism $\pi_0(Y') \to \pi_0(X)$. Moreover, the fibres of $Y' \to \pi_0(Y')$ identify with the fibres of $Y \to \pi_0(Y)$. As the image of $\pi_0(Y') \to \pi_0(Y)$ only contains connected components of $Y$ that contain a point lifting a closed point of $X$, it follows that the fibres of $Y' \to \pi_0(Y')$ map homeomorphically onto the fibres of $X \to \pi_0(X)$. Thus $\phi$ is a continuous pro-open bijection of spectral spaces. Any such map is a homeomorphism by a compactness argument. Indeed, if $U \subset Y'$ is a quasicompact open, then $\phi(U)$ is pro-(quasi-compact open), so $\phi(U) = \cap_i V_i$, where the intersection is indexed by all quasi-compact opens containing $\phi(U)$. Pulling back to $Y'$ shows $U = \cap_i \phi^{-1}(V_i)$. As $Y' - U$ is compact in the constructible topology and each $\phi^{-1}(V_i)$ is constructible, it follows that $U = \phi^{-1}(V_i)$ for some $i$, and hence $\phi(U) = V_i$.
\end{proof}

\begin{remark}
	\label{rmk:specializationmap}
Lemma \ref{lem:cwlocalcharacterize} shows that each w-local space $X$ comes equipped with a canonical ``specialization'' map $s:X \to X^c$, defined as the composition $X \to \pi_0(X) \simeq X^c$. Concretely, any $x \in X$ admits a unique closed specialization $s(x) \in X^c \subset X$; in fact, the connected component spanned by $x$ has $s(x)$ as its unique closed point. Any map in $\calS^{wl}$ preserves specializations and closed points, and is thus compatible with the specialization maps.
\end{remark}

\begin{definition}
Given a closed subspace $Z \subset X$ of a spectral space $X$, we say $X$ is {\em local along $Z$} if $X^c \subset Z$, or equivalently, if every $x \in X$ specializes to a point of $Z$. The (pro-open) subspace of $X$ comprising all points that specialize to a point of $Z$ is called the {\em localization of $X$ along $Z$}.
\end{definition}

\begin{lemma}
	\label{lem:cwlocalspreadout}
A spectral space $X$ that is local along a w-local closed subspace $Z \subset X$ with $\pi_0(Z)\cong \pi_0(X)$ is also w-local.
\end{lemma}
\begin{proof}
It suffices to show that $X^c \subset X$ is closed, and that the composition $X^c \to X \to \pi_0(X)$ is a homeomorphism. Since $X^c = Z^c$, the first claim is clear. The second follows from the w-locality of $Z$: one has $X^c = Z^c$ as before, and $\pi_0(X) = \pi_0(Z)$ by assumption. 
\end{proof}

We recall the structure of limits in $\calS$:

\begin{lemma}
$\calS$ admits all small limits, and the forgetful functor $\calS \to \Set$ preserves these limits.
\end{lemma}
\begin{proof}
Since $\calS = \Pro(\calS_f)$, it suffices to show that $\calS_f$ admits fibre products. Given maps $X \to Z \gets Y$ in $\calS_f$, one simply checks that a fibre product $X \times_Z Y$ in $\calS_f$ is computed by the usual fibre product $X \times_Z Y$ in $\Set_f$ with the topology induced from the product topology on $X \times Y$ under the inclusion $X \times_Z Y \subset X \times Y$. The second claim is then clear. Alternatively, observe that there is a factorization $\calS \stackrel{a}{\to} \Pro(\Set_{f}) \stackrel{b}{\to} \Set$, where $a(X)$ is $X$ with the constructible topology, and $b(Y) = Y$. Both functors $a$ and $b$ admit left adjoints $\alpha$ and $\beta$ respectively: $\beta$ is the Stone-Cech compactification functor, while $\alpha$ is the natural inclusion $\Pro(\Set_f) \subset \Pro(\calS_f) = \calS$. In particular, the forgetful functor $\calS \to \Set$ preserves limits.
\end{proof}

The category of w-local spaces also admits small limits:

\begin{lemma}
	\label{lem:cwlocallimits}
$\calS^{wl}$ admits all small limits, and the inclusion $i:\calS^{wl} \to \calS$ preserves these limits.
\end{lemma}
\begin{proof}
	We first check $\calS^{wl}$ admits fibre products. Given maps $X \to Z \gets Y$ in $\calS^{wl}$, the fibre product $X \times_Z Y$ in $\calS$ is local along the (profinite) closed subset $X^c \times_{Z^c} Y^c \subset X \times_Z Y$: a point $(x,y) \in X \times_Z Y$ specializes to the point $(s(x),s(y)) \in X^c \times_{Z^c} Y^c$, where $s$ is the specialization map from Remark \ref{rmk:specializationmap}. Then $X \times_Z Y \in \calS^{wl}$ by Lemma \ref{lem:cwlocalspreadout}.  Moreover, this also shows $(X \times_Z Y)^c = X^c \times_{Z^c} Y^c$, and that the projection maps $X \gets X \times_Z Y \to Y$  preserve closed points, which proves that $X \times_Z Y$ is a fibre product on $\calS^{wl}$. For cofiltered limits, fix a cofiltered diagram $\{X_i\}$ in $\calS^{wl}$. Let $X := \lim_i X_i$ be the limit (computed in $\calS$). We claim that $X \in \calS^{wl}$, and the maps $X \to X_i$ are w-local. As any open cover of $X$ can be refined by one pulled back from some $X_i$, one checks that all open covers of $X$ split. For the rest, it suffices to show $X^c = \lim_i X_i^c$; note that $\{X_i^c\}$ is a well-defined diagram as all transition maps $X_i \to X_j$ are w-local.  It is clear that $\lim_i X_i^c \subset X^c$. Conversely, choose $x \in X^c \subset X$ with image $x_i \in X_i$. Let $Y_i = \overline{\{x_i\}} \subset X_i$. Then  $\{Y_i\}$ forms a cofiltered diagram in $\calS^{wl}$ with $\lim_i Y_i \subset X$ by Lemma \ref{lem:closedincwlocal}. Moreover, one has $\lim_i Y_i = \overline{\{x\}} = \{x\} \subset X$ by the compatibility of closures and cofiltered limits. Now consider the cofiltered diagram $\{Y_i^c\}$. As each $Y_i^c \subset Y_i$ is a subset, we get $\lim_i Y_i^c \subset \lim_i Y_i = \{x\}$. Then either $x \in \lim_i Y_i^c$ or $\lim_i Y_i^c = \emptyset$; the latter possibility does not occur as a cofiltered limit of non-empty compact Hausdorff spaces is non-empty, so $x \in \lim_i Y_i^c \subset \lim_i X_i^c$.
\end{proof}

The adjoint functor theorem and Lemma \ref{lem:cwlocallimits} show that $i:\calS^{wl} \to \calS$ admits a left adjoint; this adjoint is characterized as the unique functor that preserves cofiltered limits and finite disjoint unions, and carries a connected finite $T_0$ space $X$ to $X \sqcup \{\ast\}$, where $\ast$ is declared to be a specialization of all points of $X$. This adjoint is not used in the sequel since it does not lift to the world of schemes. However, it turns out that $i:\calS^{wl} \hookrightarrow \calS$ also has a right adjoint which can be described via open covers, passes to the world of schemes, and will be quite useful:

\begin{lemma}
	\label{lem:cwlocaladjoint}
	The inclusion $i:\calS^{wl} \to \calS$ admits a right adjoint $X \mapsto X^Z$. The counit $X^Z \to X$ is a pro-(open cover) for all $X$, and the composite $(X^Z)^c \to X$ is a homeomorphism for the constructible topology on $X$.
\end{lemma}

\begin{proof}
	We first construct the functor $X \mapsto X^Z$ and the counit map $X^Z \to X$.  As the notions of w-local spaces and w-local maps are well-behaved under cofiltered limits by Lemma \ref{lem:cwlocallimits}, it suffices to construct, for each $X \in \calS_f$, a functorial open cover $X^Z \to X$ with $X^Z$ w-local such that: (a) the functor $X \mapsto X^Z$ carries maps to w-local maps, (b)  $(X^Z)^c \to X$ is a bijection, and (c) $(X^Z)^c \subset X^Z$ is discrete.

Let $X$ be a finite $T_0$ space. We define
\[
X^Z = \bigsqcup_{x\in X} X_x\ ,
\]
where $X_x\subset X$ is the subset of generalizations of $x$, which is an open subset of $X$. Then $X^Z \in \calS_f$. Moreover, each $X_x$ is w-local as the only open of $X_x$ containing $x$ is $X_x$ itself. Stability of w-locality under finite disjoint unions shows that $X^Z$ is w-local. If $f: X\to Y$ is a map of finite $T_0$ spaces, one gets an induced map
\[
f^Z: X^Z = \bigsqcup_{x\in X} X_x\to Y^Z = \bigsqcup_{y\in Y} Y_y\ ,
\]
by mapping $X_x$ into $Y_{f(x)}$. In particular, this sends the closed point $x\in X_x$ to the closed point $f(x)\in Y_{f(x)}$, so that this map is w-local. Moreover, there is a natural map $X^Z\to X$ for any $X$, by embedding each $X_x$ into $X$. Clearly, this is an open cover of $X$. The definition also shows $(X^Z)^c = X$ with the discrete topology (which is the also the constructible topology for finite $T_0$ spaces).

To show this defines an adjoint,  we must check: given $X \in \calS$, $Y \in \calS^{wl}$, and a spectral map $h:Y \to X$, there exists a unique w-local map $h':Y \to X^Z$ factoring $h$. We may assume $X \in \calS_f$ as before. As $Y^c \to Y$ is closed, the composite  $g:Y^c \hookrightarrow Y \to X$ is a spectral map from a profinite set to a finite $T_0$ space. One then checks that $g^{-1}(x)$ is clopen in $Y^c$ for all $x \in X$ (the preimage of any open of $X$ is a quasicompact open, and thus clopen, in the Hausdorff space $Y^c$; one deduces the claim by induction on $\#X$ by excising one closed point at a time). Picking an $x \in X$ with $g^{-1}(x) \neq \emptyset$ and replacing $Y$ with the clopen subset $s^{-1}(g^{-1}(x))$ where $s:Y \to \pi_0(Y) \simeq Y^c$ is the specialization map from Remark \ref{rmk:specializationmap},  we may assume that $h(Y^c) = \{x\} \subset X$; here we use Lemma \ref{lem:closedincwlocal} to ensure $Y$ remains w-local.  As each point of $Y$ specialises to a point of $Y^c$, the map $h$ factors through $X_x \subset X$, which gives the desired w-local lift $h':Y \to X_x \subset X^Z$; the w-locality requirement forces uniqueness of $h'$.
\end{proof}

\begin{remark} 	
	\label{rmk:cwlocalizeschemes}
	The space $X^Z$ can be alternatively described as: 
\[ X^Z = \lim_{ \{X_i \hookrightarrow X \}} \sqcup_i \widetilde{X_i}, \]
where the limit is indexed by the cofiltered category of constructible stratifications $\{X_i \hookrightarrow X\}$, and $\widetilde{X_i}$ denotes the set of all points of $X$ specializing to a point of $X_i$. One then has a corresponding description of closed subspaces
\[ (X^Z)^c = \lim_{ \{X_i \hookrightarrow X \}} \sqcup_i X_i \subset X^Z, \]
so it is clear that $(X^Z)^c \to X$ is a homeomorphism for the constructible topology on the target. This description and the cofinality of affine stratifications inside all constructible stratifications show that if $X$ is an affine scheme, then the maps $(X^Z)^c \stackrel{a}{\hookrightarrow} X^Z \stackrel{b}{\to} X$ lift to maps of affine schemes, with $a$ a closed immersion, and $b$ a pro-(open cover).
\end{remark}

\begin{definition}
	A map $f:W \to V$ of spectral spaces is a {\em Zariski localization} if $W = \sqcup_i U_i$ with $U_i \to V$ a quasicompact open immersion. A {\em pro-(Zariski localization)} is a cofiltered limit of such maps. 
\end{definition}

Both these notions are stable under base change. A key example is: 

\begin{lemma}
	\label{lem:profinitesetszariski}
	Any map  $f:S \to T$ of profinite sets is a pro-(Zariski localization). In fact, we can write $S = \lim_i S_i$ as a cofiltered limit of maps $S_i \to T$, each of which is the base change to $T$ of a map from a profinite set to a finite set.
\end{lemma}
\begin{proof}
	Choose a profinite presentation $T = \lim_i T_i$, and set $S_i = S \times_{T_i} T$.  Then $S_i \to T$ is the base change of $S \to T_i$, and $S \simeq \lim_i S_i$, which proves the claim.
\end{proof}

We use this notion to split a w-local map into a pro-(Zariski localization), and one that is entirely ``local:''

\begin{lemma}
	\label{lem:cwlocalfactorize}
	Any map $f:X \to Y$ in $\calS^{wl}$ admits a canonical factorization $X \to Z \to Y$ in $\calS^{wl}$ with $Z \to Y$ a pro-(Zariski localization) and $X \to Z$ inducing a homeomorphism $X^c \simeq Z^c$.
\end{lemma}
\begin{proof}
We have a diagram
\[ \xymatrix{ X^c \ar[r] \ar[d]^{f^c} & X \ar[r] \ar[d]^f & \pi_0(X) =: S \ar[d]^-{\pi_0(f)} \\
Y^c \ar[r] & Y \ar[r] & \pi_0(Y) =: T.} \]
Set $Z = Y \times_T S$. Then by Lemma \ref{lem:cwlocallimits}, $Z$ is w-local and $Z^c = Y^c \times_T S \simeq X^c$. Moreover, the map $S \to T$ is a pro-(Zariski localization), and hence so is $Z \to Y$.  The induced map $X \to Z$ sends $X^c$ to $Y^c \times_T S = Z^c$, and is thus w-local; as $X^c \to Z^c$ is a homeomorphism, this proves the claim.
\end{proof}

\subsection{Rings}
\label{subsec:cwlocalrings}

We now adapt the notions of \S \ref{subsec:spectralcwlocal}  to the world of rings via the Zariski topology, and also discuss variants for the \'etale topology:

\begin{definition} Fix a ring $A$.
	\begin{enumerate}
		\item $A$ is {\em w-local} if $\Spec(A)$ is w-local. 
		\item $A$ is {\em w-strictly local} if $A$ is w-local, and every faithfully flat \'etale map $A \to B$ has a section.
		\item A map $f:A \to B$ of w-local rings is w-local if $\Spec(f)$ is w-local.
		\item A map $f:A \to B$ is called a  {\em Zariski localization} if $B = \prod_{i=1}^n A[\frac{1}{f_i}]$ for some $f_1,\dots,f_n \in A$. An {\em ind-(Zariski localization)} is a filtered colimit of Zariski localizations.
		\item A map $f:A \to B$ is called {\em ind-\'etale} if it is a filtered colimit of \'etale $A$-algebras.
	\end{enumerate}
\end{definition}

\begin{example}
	For any ring $A$, there is an ind-(Zariski localization) $A \to A^Z$ such that $\Spec(A^Z) = \Spec(A)^Z$, see Lemma \ref{lem:cwlocalringsadjoint}. In particular, $A^Z$ is w-local. Any strictly henselian local ring $A$ is w-strictly local. Moreover, any cofiltered limit of w-strictly local rings along w-local maps is w-strictly local. 
\end{example}

Our goal in this section is to explain why every ring admits an ind-\'etale faithfully flat w-strictly local algebra. The construction of this extension, very roughly, mirrors the classical construction of the strict henselisations at a geometric point: first one Zariski localizes at the point, and then one passes up along all \'etale neighbourhoods of the point. The first step is accomplished using the functor $A \mapsto A^Z$; the next lemma describes the structure of the resulting ring.

\begin{lemma}
	\label{lem:universalabsolutelyflat}
	If $A$ is w-local, then the Jacobson radical $I_A$ cuts out $\Spec(A)^c \subset \Spec(A)$ with its reduced structure. The quotient $A/I_A$ is an absolutely flat ring.
\end{lemma}

Recall that a ring $B$ is called absolutely flat if $B$ is reduced with Krull dimension $0$ (or, equivalently, that $B$ is reduced with $\Spec(B)$ Hausdorff).

\begin{proof}
Let $J \subset A$ be the (radical) ideal cutting out $\Spec(A)^c \subset \Spec(A)$ with the reduced structure. Then $J \subset \fram$ for each $\fram \in \Spec(A)^c$, so $J \subset I_A$. Hence, $\Spec(A/I_A) \subset \Spec(A)^c$ is a closed subspace; we want the two spaces to coincide. If they are not equal, then there exists a maximal ideal $\fram$ such that $I_A \not\subset \fram$, which is impossible.
\end{proof}

The study of w-local spectral spaces has a direct bearing on w-local rings:

\begin{lemma}
	\label{lem:cwlocalringsadjoint}
The inclusion of the category w-local rings and maps inside all rings admits a left adjoint $A \mapsto A^Z$. The unit $A \to A^Z$ is a faithfully flat ind-(Zariski localization), and $\Spec(A)^Z = \Spec(A^Z)$ over  $\Spec(A)$.
\end{lemma}
\begin{proof}
	This follows from Remark \ref{rmk:cwlocalizeschemes}. In more details, let $X=\Spec A$, and define a ringed space $X^Z\to X$ by equipping $(\Spec A)^Z$ with the pullback of the structure sheaf from $X$. Then Remark \ref{rmk:cwlocalizeschemes} presents $X^Z$ as an inverse limit of affine schemes, so that $X^Z = \Spec(A^Z)$ is itself affine.
\end{proof}

\begin{example}
	For a ring $A$, the map $A \to A^Z/I_{A^Z}$ is the universal map from $A$ to an absolutely flat ring. Indeed, this follows by the universal property of $A^Z$, the w-locality of absolutely flat rings, and the observation that any w-local map $A^Z \to B$ with $B$ absolutely flat factors through a map $A^Z/I_{A^Z} \to B$.
\end{example}

\begin{lemma}
	\label{lem:cwlocalringsfactorize}
	Any w-local map $f:A \to B$ of w-local rings admits a canonical factorization $A \stackrel{a}{\to} C \stackrel{b}{\to} B$ with $C$ w-local, $a$ a w-local ind-(Zariski localization), and $b$ a w-local map inducing $\pi_0(\Spec(B)) \simeq \pi_0(\Spec(C))$.
\end{lemma}
\begin{proof}
	This follows from Lemma \ref{lem:cwlocalfactorize} and the observation that any map $S \to \pi_0(\Spec(A))$ of profinite sets is induced by an ind-(Zariski localization) $A \to C$ by applying $\pi_0(\Spec(-))$ thanks to Lemma \ref{lem:profinitesetszariski}.
\end{proof}

Due to the w-locality of $A^Z$ and Lemma \ref{lem:universalabsolutelyflat}, absolutely flat rings play an important role in this section. The next lemma explains the construction of w-strictly local ind-\'etale covers of absolutely flat rings.

\begin{lemma}
	\label{lem:absoluteflatcwstrictlylocal}
	For any absolutely flat ring $A$, there is an ind-\'etale faithfully flat map $A \to \overline{A}$ with $\overline{A}$ w-strictly local and absolutely flat. For a map $A \to B$ of absolutely flat rings, we can choose such maps $A \to \overline{A}$ and $B \to \overline{B}$ together with a map $\overline{A} \to \overline{B}$ of $A$-algebras.
\end{lemma}
\begin{proof}
The following fact is used without further comment below: any ind-\'etale algebra over an absolutely flat ring is also absolutely flat. Choose a set $I$ of isomorphism classes of faithfully flat \'etale $A$-algebras, and set $\overline{A} = \otimes_I A_i$, where the tensor product takes place over $A_i \in I$, i.e., $\overline{A} = \colim_{J \subset I} \otimes_{j \in J} A_j$, where the (filtered) colimit is indexed by the poset of finite subsets of $I$. Then one checks that $\overline{A}$ is absolutely flat, and that any faithfully flat \'etale $\overline{A}$-algebra has a section, so $\overline{A}$ is w-strictly local as $\Spec(\overline{A})$ is profinite. For the second part, simply set $\overline{B}$ to be a w-strictly local faithfully flat ind-\'etale algebra over $\overline{A} \otimes_A B$.
\end{proof}

To decouple topological problems from algebraic ones, we consistently use:

\begin{lemma}
	\label{lem:findprofinitecovers}
For any ring $A$ and a map $T \to \pi_0(\Spec(A))$ of profinite sets, there is an ind-(Zariski localization) $A \to B$ such that $\Spec(B) \to \Spec(A)$ gives rise to the given map $T \to \pi_0(\Spec(A))$ on applying $\pi_0$. Moreover, the association $T \mapsto \Spec(B)$ is a limit-preserving functor.
\end{lemma}

One may make the following more precise statement: for any affine scheme $X$, the functor $Y \mapsto \pi_0(Y)$ from affine $X$-schemes to profinite $\pi_0(X)$-sets has a fully faithful right adjoint $S \mapsto S \times_{\pi_0(X)} X$, the fibre product in the category of topological spaces ringed using the pullback of the structure sheaf on $X$. Moreover, the natural map $S \times_{\pi_0(X)} X \to X$ is a pro-(Zariski localisation) and pro-finite.

\begin{proof}
	Given $T$ as in the lemma, one may write $T = \lim T_i$ as a cofiltered limit of profinite $\pi_0(\Spec(A))$-sets $T_i$ with $T_i \to \pi_0(\Spec(A))$ being the base change of a map of finite sets, see Lemma \ref{lem:profinitesetszariski}. For each $T_i$, there is an obvious ring $B_i$ that satisfies the required properties. We then set $B := \colim B_i$, and observe that $\pi_0(\Spec(B)) = \lim \pi_0(\Spec(B_i)) = \lim T_i = T$ as a $\pi_0(\Spec(A))$-set.	
\end{proof}

One can characterize w-strictly local rings in terms of their topology and local algebra:

\begin{lemma}
	\label{lem:cwstrictlylocalcharacterize}
A w-local ring $A$ is w-strictly local if and only if all local rings of $A$ at closed points are strictly henselian.
\end{lemma}
\begin{proof}
	For the forward direction, fix a w-strictly local ring $A$ and choose a closed point $x \in \Spec(A)^c$. Any faithfully flat \'etale map $A_x \to B'$ is the localization at $x$ of a faithfully flat \'etale map $A[\frac{1}{f}] \to B$ for some $f$ invertible at $x$. As $x$ is a closed point, we may find $f_1,\ldots,f_n\in A$ vanishing at $x$ such that $C=B \times \prod_{i=1}^n A[f_i^{-1}]$ is a faithfully flat \'etale $A$-algebra. This implies that there is a section $C\to A$, and hence $C\otimes_A A_x\to A_x$. As $f_i$ vanishes at $x$, one has $C \otimes_A A_x = B_x \times A'$, where $A'$ has no point above $x$. The (algebra) section $B_x \times A' \to A_x$ then necessarily factors through the projection on the first factor, which gives us the desired section. For the converse direction, assume $A$ is a w-strictly local ring whose local rings at closed points are strictly henselian. Fix a faithfully flat \'etale $A$-algebra $B$. Then $A \to B$ has a section over each closed point of $\Spec(A)$ by the assumption on the local rings. Spreading out, which is possible by finite presentation constraints, there is a Zariski cover of $\Spec(A)$ over which $\Spec(B) \to \Spec(A)$ has a section; by w-locality of $\Spec(A)$, one finds the desired section $B \to A$.
\end{proof}

To pass from w-strictly local covers of absolutely flat rings to arbitrary rings, we use henselizations:

\begin{definition}
	\label{def:henselization}
	Given a map of rings $A \to B$, let $\Hens_{A}(-):\Ind(B_\et) \to \Ind(A_\et)$ be the functor right adjoint to the base change functor $\Ind(A_\et) \to \Ind(B_\et)$. Explicitly, for $B_0 \in \Ind(B_\et)$, we have $\Hens_A(B_0) = \colim A'$, where the colimit is indexed by diagrams $A \to A' \to B_0$ of $A$-algebras with $A \to A'$ \'etale.
\end{definition}

\begin{remark}
	The notation of Definition \ref{def:henselization} is {\em not} ambiguous, i.e., for any map $A \to B$ and $C \in \Ind(B_\et)$, the ring $\Hens_A(C)$ depends only on the $A$-algebra $C$, and not on $B$.  It follows that if $A \to A' \to C$ is a factorization with $A \to A'$ ind-\'etale, then $\Hens_A(C) \simeq \Hens_{A'}(C)$.
\end{remark}

Henselization is particularly well-behaved for quotient maps:

\begin{lemma}
	\label{lem:HensFullyFaithful}
For surjective maps $A \to A/I$, the functor $\Hens_A(-)$ is fully faithful, so $\Hens_A(-) \otimes_A A/I \simeq \id$ as functors on $\Ind( (A/I)_\et)$.
\end{lemma}
\begin{proof}
Fix some $B_0 \in \Ind( (A/I)_\et)$ and set $B = \Hens_A(B_0)$. By adjointness, it suffices to check $B/IB \simeq B_0$. As any \'etale $A/I$-algebra $C_0$ lifts to some \'etale $A$-algebra $C$, one immediately checks that $B \to B_0$ is surjective. Choose $f \in \ker(B \to B_0)$. Then $f$ lifts to some \'etale $A$-algebra $C$ along some map $C \to B$. If $f \in IC$, we are done. If not, $f$ gives an element of the kernel of $C/IC \to B_0$. Hence, there is some diagram $C/IC \to D_0 \to B_0$ in $\Ind( (A/I)_\et)$ with $C/IC \to D_0$ \'etale such that $f$ maps to $0$ in $D_0$. Choose an \'etale $C$-algebra $D$ lifting $D_0$, so $f \in ID$. The map $D \to D/ID = D_0 \to B_0$ of $A$-algebras then gives a factorization $C \to D \to B$, which shows that $f \in IB$.
\end{proof}

The \'etale analogue of Lemmas \ref{lem:closedincwlocal} and \ref{lem:cwlocalspreadout} is:

\begin{lemma}
	\label{lem:cwstrictlylocalspread}
	Let $A$ be a ring henselian along an ideal $I$. Then $A$ is w-strictly local if and only if $A/I$ is so.
\end{lemma}
\begin{proof}
First assume $A/I$ is w-strictly local. As $A$ is henselian along $I$, the space $\Spec(A)$ is local along $\Spec(A/I)$, so $A$ is w-local by Lemma \ref{lem:cwlocalspreadout}. Pick a faithfully flat \'etale $A$-algebra $B$. Then $A/I \to B/IB$ has a section. By the adjunction $\Hom_A(B,\Hens_A(A/I)) \simeq \Hom_A(B/IB,A/I)$ and the identification $\Hens_A(A/I) = A$, one finds the desired section $B \to A$. Conversely, assume $A$ is w-strictly local. Then $\Spec(A/I)^c = \Spec(A)^c$ by the henselian property, so $\Spec(A/I)^c \subset \Spec(A/I)$ is closed. Moreover, any faithfully flat \'etale $A/I$-algebra $B_0$ is the reduction modulo of $I$ of a faithfully flat \'etale $A$-algebra $B$, so the w-strict locality of $A$ immediately implies that for $A/I$.
\end{proof}

Henselizing along w-strictly local covers of absolutely flat rings gives w-strictly local covers in general:

\begin{corollary}
	Any ring $A$ admits an ind-\'etale faithfully flat map $A \to A'$ with $A'$ w-strictly local.
\end{corollary}
\begin{proof}
	Set $A' := \Hens_{A^Z}(\overline{A^Z/I_{A^Z}})$, where $\overline{A^Z/I_{A^Z}}$ is a w-strictly local ind-\'etale faithfully flat $A^Z/I_{A^Z}$-algebra; then $A'$ satisfies the required property by Lemma \ref{lem:cwstrictlylocalspread}.
\end{proof}

We end by noting that the property of w-strictly locality passes to quotients:

\begin{lemma}
Let $A$ be a ring with an ideal $I$. If $A$ is w-strictly local, so is $A/I$.
\end{lemma}
\begin{proof}
	The space $\Spec(A/I)$ is w-local by Lemma \ref{lem:closedincwlocal}. The local rings of $A/I$ at maximal ideals are quotients of those of $A$, and hence strictly henselian. The claim follows from Lemma \ref{lem:cwstrictlylocalcharacterize}.
\end{proof}

\subsection{Weakly \'etale versus pro-\'etale}
\label{subsec:weaklyetaleproetale}

In this section, we study the following notion:

\begin{definition}
	A morphism $A\to B$ of commutative rings is called {\em weakly \'etale} if both $A\to B$ and the multiplication morphism $B\otimes_A B\to B$ are flat.
\end{definition}

\begin{remark}
Weakly \'etale morphisms have been studied previously in the literature under the name of absolutely flat morphisms, see \cite{Olivier}. Here, we follow the terminology introduced in \cite[Definition 3.1.1]{GabberRamero}. 
\end{remark}

Our goal in this section is to show that weakly \'etale maps and ind-\'etale maps generate the same Grothendieck topology, see Theorem \ref{t:WeaklyVsIndEtale} below. We begin by recording basic properties of weakly \'etale maps.

\begin{proposition}\label{p:PropWeaklyEtale}
Fix maps $f:A \to B$, $g:B \to C$, and $h:A \to D$ of rings.
\begin{enumerate}
\item If $f$ is ind-\'etale, then $f$ is weakly \'etale.
\item If $f$  is weakly \'etale, then the cotangent complex $\mathbb{L}_{B/A}$ vanishes. In particular, $f$ is formally \'etale.
\item If $f$ is weakly \'etale and finitely presented, then $f$ is \'etale.
\item If $f$ and $g$ are weakly \'etale (resp. ind-\'etale), then $g\circ f$ is weakly \'etale (resp. ind-\'etale). If $g\circ f$ and $f$ are weakly \'etale (resp. ind-\'etale), then $g$ is weakly \'etale (resp. ind-\'etale).
\item If $h$ is faithfully flat, then $f$ is weakly \'etale if and only if $f\otimes_A D: D\to B\otimes_A D$ is weakly \'etale.
\end{enumerate}
\end{proposition}

\begin{proof}
These are well-known, so we mostly give references.	
\begin{enumerate}
\item As flatness and tensor products are preserved under filtered colimits, one reduces to the case of \'etale morphisms. Clearly, $f$ is flat in that case; moreover, $B\otimes_A B\to B$ is an open immersion on spectra, and in particular flat.
\item See \cite[Theorem 2.5.36]{GabberRamero} and \cite[Proposition 3.2.16]{GabberRamero}.
\item Since $f$ is weakly \'etale and finitely presented, it is formally \'etale and finitely presented by (2), hence \'etale.
\item The first part is clear. For the second part in the weakly \'etale case, see \cite[Lemma 3.1.2 (iv)]{GabberRamero}. For the ind-\'etale case, observe that the category of ind-\'etale algebras is equivalent to the ind-category of \'etale algebras by finite presentation constraints.
\item This is clear, as flatness can be checked after a faithfully flat base change. \qedhere
\end{enumerate}
\end{proof}

The analogue of (5) fails for ind-\'etale morphisms. Our main result in this section is:

\begin{theorem}
\label{t:WeaklyVsIndEtale}
	Let $f: A\to B$ be weakly \'etale. Then there exists a faithfully flat ind-\'etale morphism $g: B\to C$ such that $g\circ f: A\to C$ is ind-\'etale.
\end{theorem}

The local version of Theorem \ref{t:WeaklyVsIndEtale} follows from the following result of Olivier, \cite{Olivier}:

\begin{theorem}[Olivier]\label{t:Olivier} Let $A$ be a strictly henselian local ring, and let $B$ be a weakly \'etale local $A$-algebra. Then $f: A\to B$ is an isomorphism.
\end{theorem}

\begin{remark}
	\label{rmk:olivierdoesnotdirectlyapply}
	One might hope to use Theorem \ref{t:Olivier} for a direct proof of Theorem \ref{t:WeaklyVsIndEtale}: Assume that $f: A\to B$ is weakly \'etale. Let $C=\prod_{\overline{x}} A_{f^*\overline{x}}$, where $\overline{x}$ runs over a set of representatives for the geometric points of $\Spec(B)$, and $A_{f^*\overline{x}}$ denotes the strict henselization of $A$ at $f^*\overline{x}$. Then Theorem \ref{t:Olivier} gives maps $B \to B_{\overline{x}} \simeq A_{f^*\overline{x}}$ for each $\overline{x}$, which combine to give a map $B \to C$ inducing a section of $C \to B \otimes_A C$. However, although each $A_{\overline{x}}$ is ind-\'etale over $A$, $C$ is not even weakly \'etale over $A$, as infinite products do not preserve flatness. In order to make the argument work, one would have to replace the infinite product by a finite product; however, such a $C$ will not be faithfully flat. If one could make the sections $B\to A_{\overline{x}}$ factor over a finitely presented $A$-subalgebra of $A_{\overline{x}}$, one could also make the argument work. However, in the absence of any finiteness conditions, this is not possible. 
\end{remark}

Our proof of Theorem \ref{t:WeaklyVsIndEtale} circumvents the problem raised in Remark \ref{rmk:olivierdoesnotdirectlyapply} using the construction of w-strictly local extensions given in \S \ref{subsec:cwlocalrings} to eventually reduce to Olivier's result. We begin by recording the following relative version of the construction of such extensions:

\begin{lemma}
	\label{lem:cwstrictlylocalizemaps}
Let $f:A \to B$ be a map of rings. Then there exists a diagram
\[ \xymatrix{ A \ar[r] \ar[d]^-f & A' \ar[d]^-{f'} \\
		B \ar[r] & B' } \]
		with $A \to A'$ and $B \to B'$ faithfully flat and ind-\'etale, $A'$ and $B'$ w-strictly local, and $A' \to B'$ w-local.
\end{lemma}
\begin{proof}
	Choose compatible w-strictly local covers to get a diagram
	\[ \xymatrix{ A^Z/I_{A^Z} \ar[r] \ar[d] & \overline{A^Z/I_{A^Z}} =: A_0 \ar[d] \\
	B^Z/I_{B^Z} \ar[r] & \overline{B^Z/I_{B^Z}} =: B_0 }\]
	of absolutely flat rings with horizontal maps being faithfully flat and ind-\'etale, and $A_0$ and $B_0$ being w-strictly local. Henselizing then gives a diagram
\[ \xymatrix{ A \ar[r] \ar[d]^f & A^Z \ar[d]^-{f^Z} \ar[r] & \Hens_{A^Z}(A_0) =: A' \ar[d]^{f'} \\
B \ar[r] & B^Z \ar[r] & \Hens_{B^Z}(B_0) =: B' } \]
Then all horizontal maps are ind-\'etale faithfully flat. Moreover, both $A'$ and $B'$ are w-strictly local by Lemma \ref{lem:cwstrictlylocalspread}. The map $f'$ is w-local since $\Spec(A')^c = \Spec(A_0)$, and $\Spec(B')^c = \Spec(B_0)$, so the claim follows.
\end{proof}

We now explain how to prove an analogue of Olivier's theorem for w-strictly local rings:

\begin{lemma}
	\label{lem:cwstrictlylocalzariski}
Let $f:A \to B$ be a w-local weakly \'etale map of w-local rings with $A$ w-strictly local. Then $f$  is a ind-(Zariski localization).
\end{lemma}
\begin{proof}
	First consider the canonical factorization $A \to A' \to B$ provided by Lemma \ref{lem:cwlocalringsfactorize}.  As $A \to A'$ is w-local with $A'$ w-local, Lemma \ref{lem:cwstrictlylocalcharacterize} shows that $A'$ is w-strictly local. Replacing $A$ with $A'$, we may assume $f$ induces a homeomorphism $\Spec(B)^c \simeq \Spec(A)^c$. Then for each maximal ideal $\fram \subset A$, the ring $B/\fram B$ has a unique maximal ideal and is absolutely flat (as it is weakly \'etale over the field $A/\fram$). Then $B/\fram B$ must be a field, so $\fram B$ is a maximal ideal. The map $A_\fram \to B_{\fram B}$ is an isomorphism by Theorem \ref{t:Olivier} as $A_\fram$ is strictly henselian, so $A \simeq B$.
\end{proof}

The promised proof is:

\begin{proof}[Proof of Theorem \ref{t:WeaklyVsIndEtale}]
	Lemma \ref{lem:cwstrictlylocalizemaps} gives a diagram
	\[ \xymatrix{ A \ar[r] \ar[d]^-f & A' \ar[d]^-{f'} \\
		B \ar[r] & B' } \]
		with $f'$ a w-local map of w-strictly local rings, and both horizontal maps being ind-\'etale and faithfully flat. The map $f'$ is also weakly \'etale since all other maps in the square are so. Lemma \ref{lem:cwstrictlylocalzariski} shows that $f'$ is a ind-(Zariski localization). Setting $C = B'$ then proves the claim.
\end{proof}

\subsection{Local contractibility}
\label{subsec:localcontractibility}

In this section, we study the following notion:

\begin{definition} 
	A ring $A$ is {\em w-contractible} if every faithfully flat ind-\'etale map $A \to B$ has a section. 
\end{definition}

The name ``w-contractible'' is inspired by the connection with the pro-\'etale topology: if $A$ is w-contractible, then $\Spec(A)$ admits no non-split pro-\'etale covers, and is hence a ``weakly contractible'' object of the corresponding topos. Our goal is to prove that every ring admits a w-contractible ind-\'etale faithfully flat cover. We begin by observing that w-contractible rings are already w-local:

\begin{lemma} 
	A w-contractible ring $A$ is w-local (and thus w-strictly local).
\end{lemma}
\begin{proof}
The map $\pi:\Spec(A^Z) \to \Spec(A)$ has a section $s$ by the assumption on $A$. The section $s$ is a closed immersion since $\pi$ is separated, and $\Spec(A^Z) = \Spec(A)^Z$ is w-local, so we are done by Lemma \ref{lem:closedincwlocal}.
\end{proof}

The notion of w-contractibility is local along a henselian ideal:

\begin{lemma}
	\label{lem:cwcontractiblespread}
	Let $A$ be a ring henselian along an ideal $I$. Then $A$ is w-contractible if and only if $A/I$ is so.
\end{lemma}
\begin{proof}
	This is proven exactly like Lemma \ref{lem:cwstrictlylocalspread} using that $\Ind(A_\et) \to \Ind( (A/I)_\et)$ is essentially surjective, and preserves and reflects faithfully flat maps.
\end{proof}

The main difference between w-contractible and w-strictly local rings lies in the topology. To give meaning to this phrase, recall the following definition:

\begin{definition}
	A compact Hausdorff space is {\em extremally disconnected} if the closure of every open is open.
\end{definition}

One has the following result characterizing such spaces, see \cite{GleasonExtremallyDisconnected}:
\begin{theorem}[Gleason]
Extremally disconnected spaces are exactly the projective objects in the category of all compact Hausdorff spaces, i.e., those $X$ for which every continuous surjection $Y \to X$ splits.
\end{theorem}

It is fairly easy to prove the existence of ``enough'' extremally disconnected spaces:

\begin{example}
	\label{ex:stonecech}
For any set $X$, given the discrete topology, the Stone-Cech compactification $\beta(X)$ is extremally disconnected: the universal property shows that $\beta(X)$ is a projective object in the category of compact Hausdorff spaces. If $X$ itself comes from a compact Hausdorff space, then the counit map $\beta(X) \to X$ is a continuous surjection, which shows that all compact Hausdorff spaces can be covered by extremally disconnected spaces. In fact, the same argument shows that any extremally disconnected space is a retract of $\beta(X)$ for some set $X$. 
\end{example}

Extremally disconnected spaces tend to be quite large, as the next example shows:

\begin{example}
An elementary argument due to Gleason shows that any convergent sequence in an extremally disconnected space is eventually constant. It follows that standard profinite sets, such as $\Z_p$ (or the Cantor set) are {\em not} extremally disconnected. 
\end{example}

The relevance of extremally disconnected spaces for us is:

\begin{lemma}
	\label{lem:cwcontractiblecharacterize}
	A w-strictly local ring $A$ is w-contractible if and only if $\pi_0(\Spec(A))$ is extremally disconnected.
\end{lemma}
\begin{proof}
	As $\Spec(A)^c \to \Spec(A)$ gives a section of $\Spec(A) \to \pi_0(\Spec(A))$, if $A$ is w-contractible, then every continuous surjection $T \to \pi_0(\Spec(A))$ of profinite sets has a section, so $\pi_0(\Spec(A))$ is extremally disconnected. Conversely, assume $A$ is w-strictly local and $\pi_0(\Spec(A))$ is extremally disconnected. By Lemma \ref{lem:cwcontractiblespread}, we may assume $A = A/I_A$. Thus, we must show: if $A$ is an absolutely flat ring whose local rings are separably closed fields, and $\Spec(A)$ is extremally disconnected, then $A$ is w-contractible. Pick an ind-\'etale faithfully flat $A$-algebra $B$. Then $A \to B$ induces an isomorphism on local rings. Lemma \ref{lem:cwlocalringsfactorize} gives a factorization $A \to C \to B$ with $A \to C$ a ind-(Zariski localization) induced by a map of profinite sets $T \to \Spec(A)$, and $B \to C$ a w-local map inducing an isomorphism on spectra.  Then $C \simeq B$ as the local rings of $C$ and $B$ coincide with those of $A$. As $\Spec(A)$ is extremally disconnected, the map $T \to \Spec(A)$ of profinite sets has a section $s$. The closed subscheme $\Spec(C') \subset \Spec(C)$ realizing $s(\Spec(A)) \subset T$ maps isomorphically to $\Spec(A)$, which gives the desired section.
\end{proof}

We now show the promised covers exist:

\begin{lemma}
	\label{lem:cwcontractiblecover}
For any ring $A$, there is an ind-\'etale faithfully flat $A$-algebra $A'$ with $A'$ w-contractible.
\end{lemma}
\begin{proof}
	Choose an ind-\'etale faithfully flat $A^Z/I_{A^Z}$-algebra $A_0$ with $A_0$ w-strictly local and $\Spec(A_0)$ an extremally disconnected profinite set; this is possible by Example \ref{ex:stonecech}, Lemma \ref{lem:absoluteflatcwstrictlylocal}, and Lemma \ref{lem:findprofinitecovers}. Let $A' = \Hens_{A^Z}(A_0)$. Then $A'$ is w-contractible by Lemma \ref{lem:cwcontractiblespread} and Lemma \ref{lem:cwcontractiblecharacterize}, and the map $A \to A'$ is faithfully flat and ind-\'etale since both $A \to A^Z$ and $A^Z \to A'$ are so individually.
\end{proof}

\begin{lemma}
	\label{lem:finiteovercontractible}
Let $A$ be a w-contractible ring, and let $f:A \to B$ be a finite ring map of finite presentation. Then $B$ is w-contractible.
\end{lemma}
\begin{proof}
We can write $A = \colim_i A_i$ as a filtered colimit of finite type $\Z$-algebras such that $A \to B$ is the base change of a finite ring map $A_0 \to B_0$ of some index $0$, assumed to be initial; set $B_i = B_0 \otimes_{A_0} A_i$, so $B = \colim_i B_i$. Then $\Spec(A) = \lim_i \Spec(A_i)$ and $\Spec(B) = \lim_i \Spec(B_i)$ as affine schemes and as spectral spaces, so $\pi_0(\Spec(B)) = \pi_0(\Spec(B_0)) \times_{\pi_0(\Spec(A_0))} \pi_0(\Spec(A))$. As $\pi_0(\Spec(A_0))$ and $\pi_0(\Spec(B_0))$ are both finite sets, it follows that $\pi_0(\Spec(B))$ is extremally disconnected as $\pi_0(\Spec(A))$ is such. Moreover, the local rings of $B$ are strictly henselian as they are finite over those of $A$. It remains to check $\Spec(B)$ is w-local. By finiteness, the subspace $\Spec(B)^c \subset \Spec(B)$ is exactly the inverse image of $\Spec(A)^c \subset \Spec(A)$, and hence closed. Now pick a connected component $Z \subset \Spec(B)$. The image of $Z$ in $\Spec(A)$ lies in some connected component $W \subset \Spec(A)$. The structure of $A$ shows that $W = \Spec(A_x)$ for some closed point $x \in \Spec(A)^c$, so $W$ is a strictly henselian local scheme. Then $Z \to W$ is a finite map of schemes with $Z$ connected, so $Z$ is also a strictly henselian local scheme, and hence must have a unique closed point, which proves w-locality of $\Spec(B)$.
\end{proof}

\begin{remark} The finite presentation assumption is necessary. Indeed, there are extremally disconnected spaces $X$ with a closed subset $Z\subset X$ such that $Z$ is not extremally disconnected. As an example, let $X$ be the Stone-Cech compactification of $\N$, and let $Z = X \setminus \N$. As any element of $\N$ is an open and closed point of $X$, $Z\subset X$ is closed. Consider the following open subset $\tilde{U}$ of $X$:
\[
\tilde{U} = \bigcup_{n\geq 1} \{x\in X\mid x\not\equiv 0\mod 2^n\}\ .
\]
Here, we use that the map $\N\to \Z/n\Z$ extends to a unique continuous map $X\to \Z/n\Z$. Let $U = \tilde{U}\cap Z$, which is an open subset of $Z$. We claim that the closure $\overline{U}$ of $U$ in $Z$ is not open. If not, then $Z$ admits a disconnection with one of the terms being $\overline{U}$. It is not hard to see that any disconnection of $Z$ extends to a disconnection of $X$, and all of these are given by $\overline{M}\sqcup (X\setminus \overline{M})$ for some subset $M\subset \N$. It follows that $\overline{U} = \overline{M}\cap Z$ for some subset $M\subset \N$. Thus, $U\subset \overline{M}$, which implies that for all $n\geq 0$, almost all integers not divisible by $2^n$ are in $M$. In particular, there is a subset $A\subset M$ such that $A=\{a_0,a_1,\ldots\}$ with $2^i | a_i$. Take any point $x\in \overline{A}\setminus \N\subset Z$. Thus, $x\in \overline{M}\cap Z = \overline{U}$. On the other hand, $x$ lies in the open subset $V=\overline{A}\cap Z\subset Z$, and $V\cap U=\emptyset$: Indeed, for any $n\geq 0$,
\[
\overline{A}\cap \{x\in X\mid x\not\equiv 0\mod 2^n\}\subset \{a_0,\ldots,a_{n-1}\}\subset \N\ .
\]
This contradicts $x\in \overline{U}$, finally showing that $\overline{U}$ is not open.
\end{remark}

\newpage

\section{On replete topoi}\label{sec:Replete}

A topos is the category of sheaves on a site, up to equivalence, as in \cite{SGA4Tome1}. We will study in \S \ref{ss:reptopoi} a general property of topoi that implies good behaviour for the $\lim$ and $\R\lim$ functors, as well as unbounded cohomological descent, as discussed in \S \ref{ss:reptopoicompleteness}. A special subclass of such topoi with even better completeness properties is isolated in \S \ref{ss:lwctopoi}; this class is large enough for all applications later in the paper. In \S \ref{subsec:derivedcompfadic} and \S \ref{subsec:derivedcompnoeth}, with a view towards studying complexes of $\ell$-adic sheaves on the pro-\'etale site, we study derived completions of rings and modules in a replete topos; the repleteness ensures no interference from higher derived limits while performing completions, so the resulting theory is as good as in the punctual case.

\subsection{Definition and first consequences} 
\label{ss:reptopoi}

The key definition is:
 
\begin{definition}
	A topos $\calX$ is {\em replete} if surjections in $\calX$ are closed under sequential limits, i.e., if $F:\N^\opp \to \calX$ is a diagram with $F_{n+1} \to F_n$ surjective for all $n$, then $\lim F \to F_n$ is surjective for each $n$.
\end{definition}

Before giving examples, we mention two recogition mechanisms for replete topoi:

\begin{lemma}
	\label{lem:repletebc}
	If $\calX$ is a replete topos and $X \in \calX$, then $\calX_{/X}$ is replete.
\end{lemma}
\begin{proof}
	This follows from the fact that the forgetful functor $\calX_{/X} \to \calX$ commutes with connected limits and preserves surjections.
\end{proof}

\begin{lemma}
	\label{lem:repletelocal}
	A topos $\calX$ is replete if and only if there exists a surjection $X \to 1$ and $\calX_{/X}$ is replete. 
\end{lemma}
\begin{proof}
This follows from two facts: (a) limits commute with limits, and (b) a map $F \to G$ in $\calX$ is a surjection if and only if it is so after base changing to $X$.
\end{proof}

\begin{example}
	\label{ex:setsreplete}
The topos of sets is replete, and hence so is the topos of presheaves on a small category. As a special case, the classifying topos of a finite group $G$ (which is simply the category of presheaves on $B(G)$) is replete.
\end{example}

\begin{example}
\label{ex:fieldreplete}
Let $k$ be a field with a fixed separable closure $\overline{k}$. Then $\calX = \Shv(\Spec(k)_\et)$ is replete if and only if $\overline{k}$ is a finite extension of $k$.\footnote{Recall that this happens only if $k$ is algebraically closed or real closed; in the latter case, $k(\sqrt{-1})$ is an algebraic closure of $k$.} One direction is clear: if $\overline{k}/k$ is finite, then $\Spec(\overline{k})$ covers the final object of $\calX$ and $\calX_{/\Spec(\overline{k})} \simeq \Set$, so $\calX$ is replete by Lemma \ref{lem:repletelocal}. Conversely, assume that $\calX$ is replete with $\overline{k}/k$ infinite. Then there is a tower $k = k_0 \hookrightarrow k_1 \hookrightarrow k_2 \hookrightarrow \dots$ of strictly increasing finite separable extensions of $k$. The associated diagram $\dots\to \Spec(k_2) \to \Spec(k_1) \to \Spec(k_0)$ of surjections has an  empty limit in $\calX$, contradicting repleteness.
\end{example}

\begin{remark}
	Replacing $\N^\opp$ with an arbitrary small cofiltered category in the definition of replete topoi leads to an empty theory:  there are cofiltered diagrams of sets with surjective transition maps and empty limits. For example, consider the poset $I$ of finite subsets of an uncountable set $T$ ordered by inclusion, and $F:I^\opp \to \Set$ defined by 
	\[ F(S) = \{f \in \Hom(S,\Z) \mid f \ \mathrm{injective} \}. \]  
Then $F$ is a cofiltered diagram of sets with surjective transition maps, and $\lim F = \emptyset$.
\end{remark}

Example \ref{ex:fieldreplete} shows more generally that the Zariski (or \'etale, Nisnevich, smooth, fppf)  topoi of most schemes fail repleteness due to ``finite presentation'' constraints. Nevertheless, there is an interesting geometric source of examples:

\begin{example}
\label{ex:fpqcreplete}
The topos $\calX$ of fpqc sheaves on the category of schemes\footnote{To avoid set-theoretic problems, one may work with countably generated affine schemes over a fixed affine base scheme.} is replete.  Given a  diagram $\dots \to F_{n+1} \to F_n \to \dots \to F_1 \to F_0$ of fpqc sheaves with $F_n \to F_{n-1}$ surjective, we want $\lim F_n \to F_0$ to be surjective. For any affine $\Spec(A)$ and a section $s_0 \in F_0(\Spec(A))$, there is a faithfully flat map $A \to B_1$ such that $s_0$ lifts to an $s_1 \in F_1(\Spec(B_1))$. Inductively, for each $n \geq 0$, there exist faithfully flat maps $A \to B_{n}$ compatible in $n$ and sections $s_n \in F_n(\Spec(B_n))$ such that $s_n$ lifts $s_{n-1}$.  Then $B = \colim_n B_n$ is a faithfully flat $A$-algebra with $s_0 \in F_0(\Spec(A))$ lifting to an $s \in \lim F_n(\Spec(B))$, which proves repleteness as $\Spec(B) \to \Spec(A)$ is an fpqc cover.
\end{example}

The next lemma records a closure property enjoyed by surjections in a replete topos.

\begin{lemma}
\label{lem:limiteffepi}
Let $\calX$ be a replete topos, and let $F \to G$ be a map in $\Fun(\N^\opp,\calX)$. Assume that the induced maps $F_i \to G_i$ and $F_{i+1} \to F_i \times_{G_i} G_{i+1}$  are surjective for each $i$. Then $\lim F \to \lim G$ is surjective.
\end{lemma}
\begin{proof}
	Fix an $X \in \calX$ and a map $s:X \to \lim G$ determined by a compatible sequence $\{s_n:X \to G_n\}$ of maps. By induction, one can show that there exists a tower of surjections $\dots \to X_n \to X_{n-1} \to \dots \to X_1 \to X_0 \to X$ and maps $t_n:X_n \to F_n$ compatible in $n$ such that $t_n$ lifts $s_n$. In fact, one may take $X_0 = X \times_{G_0} F_0$,  and 
	\[ X_{n+1} = X_n \times_{F_n \times_{G_n} G_{n+1}} F_{n+1}.\]
The map $X' := \lim_i X_i \to X$ is surjective by repleteness of $\calX$. Moreover, the compatibility of the $t_n$'s gives a map $t:X' \to \lim F$ lifting $s$, which proves the claim.
\end{proof}

We now see some of the benefits of working in a replete topos. First, products behave well:

\begin{proposition}
	\label{prop:prodexact}
	Countable products are exact in a replete topos.
\end{proposition}
\begin{proof}
Given surjective maps $f_n:F_n \to G_n$ in $\calX$ for each $n \in \N$, we want $f:\prod_n F_n \to \prod_n G_n$ to be surjective. This follows from Lemma \ref{lem:limiteffepi} as $f = \lim \prod_{i < n} f_i$; the condition from the lemma is trivial to check in this case.
\end{proof}

In a similar vein, inverse limits behave like in sets:

\begin{proposition} 
	\label{prop:limrlim}
	If $\calX$ is a replete topos and $F:\N^\opp \to \Ab(\calX)$ is a diagram with $F_{n+1} \to F_n$ surjective for all $n$, then $\lim F_n \simeq \R\lim F_n$.
\end{proposition}
\begin{proof}
	By Proposition \ref{prop:prodexact}, the product $\prod_n F_n \in \calX$ computes the derived product in $D(\calX)$. This gives an exact triangle
\[  \R\lim F_n \to \prod_n F_n \stackrel{t-\id}{\to} \prod_n F_n, \]
where $t:F_{n+1} \to F_n$ is the transition map. It thus suffices to show that $s := t-\id$ is surjective. Set $G_n = \prod_{i \leq n} F_n$, $H_n = G_{n+1}$, and let $s_n:H_n \to G_n$ be the map induced by $t - \id$. The surjectivity of $t$ shows that $s_n$ is surjective. Moreover, the surjectivity of $t$ also shows that $H_{n+1} \to G_{n+1} \times_{G_n} H_{n}$ is surjective, where the fibre product is computed using $s_n:H_n \to G_n$ and the projection $G_{n+1} \to G_n$. In fact, the fibre product is $H_{n} \times F_{n+1}$ and $H_{n+1} \to H_{n} \times F_{n+1}$ is $(\pr,t-\id)$. By Lemma \ref{lem:limiteffepi}, it follows that $s = \lim s_n$ is also surjective.
\end{proof}

\begin{proposition}
	\label{prop:cdrlim}
If $\calX$ is a replete topos, then the functor of $\N^\opp$-indexed limits has cohomological dimension $1$.
\end{proposition}
\begin{proof}
	For a diagram $F:\N^\opp \to \Ab(\calX)$, we want $\R\lim F_n \in D^{[0,1]}(\calX)$. By definition, there is an exact triangle
	\[ \R\lim F_n \to \prod_n F_n \to \prod_n F_n\]
	with the last map being the difference of the identity and transition maps, and the products being derived. By Proposition \ref{prop:prodexact}, we can work with naive products instead, whence the claim is clear by long exact sequences.
\end{proof}

\begin{question}
	\label{question:hypercompletereplete}
	Do Postnikov towers converge in the hypercomplete $\infty$-topos of sheaves of spaces (as in \cite[\S 6.5]{LurieHTT}) on a replete topos?
\end{question}

\subsection{Locally weakly contractible topoi}
\label{ss:lwctopoi}

We briefly study an exceptionally well-behaved subclass of replete topoi:

\begin{definition}
\label{rmk:locallyweaklycontractible}
An object $F$ of a topos $\calX$ is called {\em weakly contractible} if every surjection $G \to F$ has a section. We say that $\calX$ is {\em locally weakly contractible} if it has enough weakly contractible coherent objects, i.e., each $X \in \calX$ admits a surjection $\cup_i Y_i \to X$ with $Y_i$ a coherent weakly contractible object.
\end{definition}

The pro-\'etale topology will give rise to such topoi. A more elementary example is:

\begin{example} 
The topos $\calX = \Set$ is locally weakly contractible: the singleton set $S$ is weakly contractible coherent, and every set is covered by a disjoint union of copies of $S$. 
\end{example}

The main completeness and finiteness properties of such topoi are:

\begin{proposition}
Let $\calX$ be a locally weakly contractible topos. Then
\begin{enumerate}
	\item $\calX$ is replete.
	\item The derived category $D(\calX) = D(\calX,\Z)$ is compactly generated.
	\item Postnikov towers converge in the associated hypercomplete $\infty$-topos. (Cf. \cite{LurieHTT}.)
\end{enumerate}
\end{proposition}
\begin{proof}
	For (1), note that a map $F \to G$ in $\calX$ is surjective if and only if $F(Y) \to G(Y)$ is so for each weakly contractible $Y$; the repleteness condition is then immediately deduced. For (2), given $j:Y \to 1_{\calX}$ in $\calX$ with $Y$ weakly contractible coherent, one checks that $\Hom(j_! \Z,-) = H^0(Y,-)$ commutes with arbitrary direct sums in $D(\calX)$, so $j_! \Z$ is compact; as $Y$ varies, this gives a generating set of $D(\calX)$ by assumption on $\calX$, proving the claim. For (3), first note that the functor $F \mapsto F(Y)$ is exact on sheaves of spaces whenever $Y$ is weakly contractible. Hence, given such an $F$ and point $\ast \in F(Y)$ with $Y$ weakly contractible, one has $\pi_i(F(Y),\ast) = \pi_i(F,\ast)(Y)$. This shows that $F \simeq \lim_n \tau_{\leq n} F$ on $\calX$, which proves hypercompleteness. (Cf. \cite[Proposition 7.2.1.10]{LurieHTT}.)
\end{proof}

\subsection{Derived categories, Postnikov towers, and cohomological descent}
\label{ss:reptopoicompleteness}

We first recall the following definition:

\begin{definition}
	Given a topos $\calX$, we define the {\em left-completion} $\widehat{D}(\calX)$ of $D(\calX)$ as the full subcategory of $D(\calX^\N)$ spanned by projective systems $\{K_n\}$ satisfying:
	\begin{enumerate}
		\item $K_n \in D^{\geq -n}(\calX)$.
		\item The map $\tau^{\geq -n} K_{n+1} \to K_n$ induced by the transition map $K_{n+1} \to K_n$ and (1) is an equivalence.
	\end{enumerate}
	We say that $D(\calX)$ is {\em left-complete} if the map $\tau:D(\calX) \to \widehat{D}(\calX)$ defined by $K \mapsto \{\tau^{\geq -n} K\}$ is an equivalence.
\end{definition}
	
Left-completeness is extremely useful in accessing an unbounded derived category as Postnikov towers converge:

\begin{lemma}
	\label{lem:leftcompletionadjunction}
	The functor $\R\lim:\widehat{D}(\calX) \hookrightarrow D(\calX^\N) \to D(\calX)$ provides a right adjoint to $\tau$. In particular, if $D(\calX)$ is left-complete, then $K \simeq \R\lim \tau^{\geq -n} K$ for any $K \in D(\calX)$.
\end{lemma}
\begin{proof}
Fix $K \in D(\calX)$ and $\{L_n\} \in \widehat{D}(\calX)$. Then we claim that
\[\begin{aligned}
\R\Hom_{D(\calX)}(K,\R\lim L_n) \simeq \R\lim \R\Hom_{D(\calX)}(K,L_n) &\simeq \R\lim \R\Hom_{D(\calX)}(\tau^{\geq -n} K,L_n)\\
& \simeq \R\Hom_{\widehat{D}(\calX)}(\tau(K),\{L_n\}).
\end{aligned}\]
This clearly suffices to prove the lemma. Moreover, the first two equalities are formal. For the last one, recall that if $F,G \in \Ab(\calX^\N)$, then there is an exact sequence
\[ 1 \to \Hom(F,G) \to \prod_n \Hom(F_n,G_n) \to \prod_n \Hom(F_{n+1},G_n), \]
where the first map is the obvious one, while the second map is the difference of the two maps $F_{n+1} \to F_n \to G_n$ and $F_{n+1} \to G_{n+1} \to G_n$. One can check that if $F,G \in \Ch(\calX^\N)$, and $G$ is chosen to be K-injective, then the above sequence gives an exact triangle
\[ \R\Hom(F,G) \to \prod_n \R\Hom(F_n,G_n) \to \prod_n \R\Hom(F_{n+1},G_n).\]
In the special case where $F,G \in \widehat{D}(\calX)$, one has $\R\Hom(F_{n+1},G_n) = \R\Hom(F_n,G_n)$ by adjointness of truncations, which gives the desired equality.
\end{proof}

Classically studied topoi have left-complete derived categories only under (local) finite cohomological dimension constraints; see Proposition \ref{prop:leftcompletecrit} for a criterion, and Example \ref{ex:notleftcompleteag} for a typical example of the failure of left-completeness for the simplest infinite-dimensional objects. The situation for replete topoi is much better:

\begin{proposition}
	\label{prop:repletepostnikov}
If $\calX$ is a replete topos, then $D(\calX)$ is left-complete.
\end{proposition}
\begin{proof}
We repeatedly use the following fact: limits and colimits in the abelian category $\Ch(\Ab(\calX))$ are computed termwise. First, we show that $\tau:D(\calX) \to \widehat{D}(\calX)$ is fully faithful. By the adjunction from Lemma \ref{lem:leftcompletionadjunction}, it suffices to show that $K \simeq \R\lim \tau^{\geq -n} K$ for any $K \in D(\calX)$. Choose a complex $I \in \Ch(\Ab(\calX))$ lifting $K \in D(\calX)$. Then $\prod_n \tau^{\geq -n} I \in \Ch(\Ab(\calX))$ lifts the derived product $\prod_n \tau^{\geq -n} K \in D(\calX)$ by Proposition \ref{prop:prodexact}. Since $I \simeq \lim \tau^{\geq -n} I \in \Ch(\Ab(\calX))$, it suffices as in Proposition \ref{prop:limrlim} to show that 
\[ \prod_n \tau^{\geq -n} I \stackrel{t-\id}{\to} \prod_n \tau^{\geq -n} I\]
is surjective in $\Ch(\Ab(\calX))$, where we write $t$ for the transition maps. Since surjectivity in $\Ch(\Ab(\calX))$ can be checked termwise, this follows from the proof of Proposition \ref{prop:limrlim} as $\tau^{\geq -n} I \stackrel{t-\id}{\to} \tau^{\geq -(n-1)} I$ is termwise surjective.  

For essential surjectivity of $\tau$, it suffices to show: given $\{K_n\} \in \widehat{D}(\calX)$, one has $K_n \simeq \tau^{\geq -n} \R\lim K_n$. Choose a $K$-injective complex $\{I_n\} \in \Ch(\Ab(\calX^\N))$ representing $\{K_n\}$. Then $\prod_n I_n \in \Ch(\Ab(\calX))$ lifts $\prod_n K_n$ (the derived product). Moreover, by $K$-injectivity, the transition maps $I_{n+1} \to I_n$ are (termwise) surjective. Hence, the map 
\[ \prod_n I_n \stackrel{t-\id}{\to} \prod_n I_n \]
in $\Ch(\Ab(\calX))$ is surjective by the argument in the proof of Proposition \ref{prop:limrlim}, and its kernel complex $K$ computes $\R\lim K_n$. We must show that $H^i(K) \simeq H^i(K_i)$ for each $i \in \N$.  Calculating cohomology and using the assumption $\{K_n\} \in \widehat{D}(\calX) \subset D(\calX^\N)$ shows that
\[ H^i(\prod_n I_n) = \prod_{n} H^i(I_n) = \prod_{n \geq i} H^i(I_n) = \prod_{n \geq i} H^i(K_i)\]
for each $i \in \N$; here we crucially use Proposition \ref{prop:prodexact} to distribute $H^i$ over $\prod$. The map $H^i(t-\id)$ is then easily seen to be split surjective with kernel $\lim H^i(K_n) \simeq \lim H^i(K_i) \simeq H^i(K_i)$, which proves the claim.
\end{proof}

If repleteness is dropped, it is easy to give examples where $D(\calX)$ is not left-complete.

\begin{example}
	\label{ex:notleftcomplete}
	Let $G = \prod_{n \geq 1} \Z_p$, and let $\calX$ be the topos associated to the category $B(G)$ of finite $G$-sets (topologized in the usual way). We will show that $D(\calX)$ is not left-complete. More precisely, we will show that $K \to \widehat{K} := \R\lim \tau^{\geq -n} K$ does not have a section for $K = \oplus_{n \geq 1} \Z/p^n[n] \in D(\calX)$; here $\Z/p^n$ is given the trivial $G$-action.

For each open subgroup $H \subset G$, we write $X_H \in B(G)$ for the $G$-set $G/H$ given the left $G$-action, and let $I^\opp \subset B(G)$ be the (cofiltered) full subcategory spanned by the $X_H$'s. The functor $p^*(\calF) =  \colim_{I} \calF(X_H)$ commutes with finite limits and all small colimits, and hence comes from a point $p:\ast \to \calX$. Deriving gives $p^* L = \colim_I  \R\Gamma(X_H,L)$ for any $L \in D(\calX)$, and so $H^0(p^* L) = \colim_I H^0(X_H,L)$.  In particular, if $L_1 \to L_2$ has a section, so does
	\[ \colim_I H^0(X_H,L_1) \to \colim_I H^0(X_H,L_2).\]
	If $\pi:\calX \to \Set$ denotes the constant map, then $K = \pi^* K'$ where $K' = \oplus_{n \geq 1} \Z/p^n[n] \in D(\Ab)$, so
	\[ \colim_I H^0(X_H,K) = H^0(p^* K) = H^0(p^* \pi^* K') = H^0(K') = 0.\]
	Since $\tau^{\geq -n} K \simeq \oplus_{ i \leq n} \Z/p^i[i] \simeq \prod_{i \leq n} \Z/p^i[i]$,  commuting limits shows that $\widehat{K} \simeq \prod_{n \geq 1} \Z/p^n[n]$ (where the product is derived), and so $\R\Gamma(X_H,\widehat{K}) \simeq \prod_{n \geq 1} \R\Gamma(X_H,\Z/p^n[n])$. In particular, it suffices to show that 
	\[ H^0(p^* \widehat{K}) = \colim_I \prod_{n \geq 1} H^n(X_H,\Z/p^n) \]
	is not $0$. Let $\alpha_n \in H^n(X_G,\Z/p^n) = H^n(\calX,\Z/p^n)$ be the pullback of a generator of $H^n(B(\prod_{i=1}^n \Z_p),\Z/p^n) \simeq \otimes_{i=1}^n H^1(B(\Z_p),\Z/p^n)$ under the projection $f_n:G \to \prod_{i=1}^n \Z_p$. Then $\alpha_n$ has exact order $p^n$ as $f_n$ has a section, so $\alpha := (\alpha_n) \in \prod_{n \geq  1} H^n(\calX,\Z/p^n)$ has infinite order. Its image $\alpha'$ in $H^0(p^* \widehat{K})$ is $0$ if and only if there exists an open normal subgroup $H \subset G$ such that $\alpha$ restricts to $0$ in $\prod_n H^n(X_H,\Z/p^n)$. Since $X_H \to X_G$ is a finite cover of degree $[G:H]$, a transfer argument then implies that $\alpha$ is annihilated by $[G:H]$, which is impossible, whence $\alpha' \neq 0$. 
\end{example}

\begin{remark}
	\label{ex:notleftcompleteag}
	The argument of Example \ref{ex:notleftcomplete} is fairly robust: it also applies to the \'etale topos of $X = \Spec(k)$ with $k$ a field provided there exist $M_n \in \Ab(X_\et)$ for infinitely many $n \geq 1$ such that $H^n(\calX,M_n)$ admits a class $\alpha_n$ with $\lim \ord(\alpha_n) = \infty$. In particular, this shows that $D(\Spec(k)_\et)$ is not left-complete for $k = \C(x_1,x_2,x_3,\dots)$.
\end{remark}

Thanks to left-completeness, cohomological descent in a replete topos is particularly straightforward:

\begin{proposition}
	\label{prop:cohdescentreplete}
	Let $f:X_\bullet \to X$ be a hypercover in a replete topos $\calX$. Then
	\begin{enumerate}
		\item The adjunction $\id \to f_* f^*$ is an equivalence on $D(X)$.
		\item The adjunction $f_! f^* \to \id$ is an equivalence on $D(X)$.
		\item $f^*$ induces an equivalence $D(X) \simeq D_\cart(X_\bullet)$.
	\end{enumerate}
\end{proposition}

Here we write $D(Y) = D(\Ab(\calX_{/Y}))$ for any $Y \in \calX$. Then $D(X_\bullet)$ is the derived category of the simplicial topos defined by $X_\bullet$, and $D_\cart(X_\bullet)$ is the full subcategory spanned by complexes $K$ which are {\em Cartesian}, i.e., for any map $s:[n] \to [m]$ in $\Delta$, the transition maps $s^*(K|_{X_n}) \to K|_{X_m}$ are equivalences. The usual pushforward then gives $f_*:D(X_\bullet) \to D(X)$ right adjoint to the pullback $f^*:D(X) \to D(X_\bullet)$ given informally via $(f^* K)|_{X_n} = K|_{X_n}$. By the adjoint functor theorem, there is a left adjoint $f_!:D(X_\bullet) \to D(X)$ as well. When restricted to $D_\cart(X_\bullet)$, one may describe $f_!$ informally as follows. For each Cartesian $K$ and any map $s:[n] \to [m]$ in $\Delta$, the equivalence $s^*(K|_{X_n}) \simeq K|_{X_m}$ has an adjoint map $K|_{X_m} \to s_!(K|_{X_n})$. Applying $!$-pushforward along each $X_n \to X$ then defines a simplicial object in $D(X)$ whose homotopy-colimit computes $f_! K$.

\begin{proof} We freely use that homotopy-limits and homotopy-colimits in $D(X_\bullet)$ are computed ``termwise.'' Moreover, for any map $g:Y \to X$ in $\calX$, the pullback $g^*$ is exact and commutes with such limits and colimits (as it has a left adjoint $g_!$ and a right adjoint $g_*$). Hence $f^*:D(X) \to D(X_\bullet)$ also commutes with such limits and colimits.
	\begin{enumerate}
		\item For any $K \in \Ab(X)$, one has $K \simeq f_* f^* K$ by the hypercover condition. Passing to filtered colimits shows the same for $K \in D^+(X)$. For general $K \in D(X)$, we have $K \simeq \R\lim \tau^{\geq -n} K$ by repleteness.  By exactness of $f^*$ and repleteness of each $X_n$,  one has $f^* K \simeq \R\lim f^* \tau^{\geq -n} K$. Pushing forward then proves the claim.
		\item This follows formally from (1) by adjunction.
		\item The functor $f^*:D(X) \to D_\cart(X_\bullet)$ is fully faithful by (1) and adjunction. Hence, it suffices to show that any $K \in D_\cart(X_\bullet)$ comes from $D(X)$. The claim is well-known for $K \in D^+_\cart(X_\bullet)$ (without assuming repleteness). For general $K$, by repleteness, we have $K \simeq \R\lim \tau^{\geq -n} K$. Since the condition of being Cartesian on a complex is a condition on cohomology sheaves, the truncations $\tau^{\geq -n} K$ are Cartesian, and hence come from $D(X)$. The claim follows as $D(X) \subset D(X_\bullet)$ is closed under homotopy-limits. \qedhere
		\end{enumerate}
\end{proof}

We end by recording a finite dimensionality criterion for left-completeness:

\begin{proposition}
	\label{prop:leftcompletecrit}
	Let $\calX$ be a topos, and fix $K \in D(\calX)$. 
	\begin{enumerate}
		\item Given $U \in \calX$ with $\Gamma(U,-)$ exact, one has $\R\Gamma(U,K) \simeq \R\lim \R\Gamma(U,\tau^{\geq -n} K)$.
		\item If there exists $d \in \N$ such that $\calH^i(K)$ has cohomological dimension $\leq d$ locally on $\calX$ for all $i$, then $D(\calX)$ is left-complete.
	\end{enumerate}
\end{proposition}
\begin{proof}
	For (1), by exactness, $\R\Gamma(U,K)$ is computed by $I(U)$ where $I \in \Ch(\calX)$ is {\em any} chain complex representing $K$. Now $D(\Ab)$ is left-complete, so $I(U) \simeq \R\lim \tau^{\geq -n} I(U)$. As $\Gamma(U,-)$ is exact, it commutes with truncations, so the claim follows. (2) follows from \cite[Tag 0719]{StacksProject}.
\end{proof}

\subsection{Derived completions of f-adic rings in a replete topos}
\label{subsec:derivedcompfadic}

In this section, we fix a replete topos $\calX$, and a ring $R \in \calX$ with an ideal $I \subset R$ that is locally finitely generated, i.e., there exists a cover $\{U_i \to 1_{\calX}\}$ such that $I|_{U_i}$ is generated by finitely many sections of $I(U_i)$. Given $U \in \calX$, $x \in R(U)$ and $K \in D(\calX_{/U},R)$,  we write $T(K,x) := \R\lim(\dots \stackrel{x}{\to}  K \stackrel{x}{\to} K \stackrel{x}{\to} K) \in D(\calX_{/U},R)$.

\begin{definition}
	We say that $M\ \in \Mod_R$ is {\em classically $I$-complete} if $M \simeq \lim M/I^nM$; write $\Mod_{R,\comp} \subset \Mod_R$ for the full subcategory of such $M$. We say that $K \in D(\calX,R)$ is {\em derived $I$-complete} if for each $U \in \calX$ and $x \in I(U)$, we have $T(K|_U,x) = 0$; write $D_{\comp}(\calX,R) \subset D(\calX,R)$ for the full subcategory of such $K$.
\end{definition}

It is easy to see that $D_{\comp}(\calX,R)$ is a triangulated subcategory of $D(\calX,R)$. Moreover, for any $U \in \calX$, the restriction $D(\calX,R) \to D(\calX_{/U},R)$ commutes with homotopy-limits, and likewise for $R$-modules. Hence, both the above notions of completeness localise on $\calX$. Our goal is to compare these completeness conditions for modules, and relate completeness of a complex to that of its cohomology groups. The main result for modules is:

\begin{proposition}
	\label{prop:classcompdercompcrit}
	An $R$-module $M \in \Mod_R$ is classically $I$-complete if and only if it is $I$-adically separated and derived $I$-complete.
\end{proposition}

\begin{remark}
	The conditions of Proposition \ref{prop:classcompdercompcrit}  are not redundant: there exist derived $I$-complete $R$-modules $M$ which are not $I$-adically separated, and hence not classically complete. In fact, there exists a ring $R$ with principal ideals $I$ and $J$ such that $R$ is classically $I$-complete while the quotient $R/J$ is not $I$-adically separated; note that $R/J = \cok(R \to R)$ is derived $I$-complete by Lemma \ref{lem:dercompabcat}.
\end{remark}

The result for complexes is: 

\begin{proposition}
	\label{prop:complexcompcrit}
	An $R$-complex $K \in D(\calX,R)$ is derived $I$-complete if and only if each $H^i(K)$ is so.
\end{proposition}

\begin{remark}
	For $\calX = \Set$, one can find Proposition \ref{prop:complexcompcrit} in \cite{LurieDAGXII}.
\end{remark}

\begin{lemma}
	\label{lem:invertsum}
Given $x,y \in R(\calX)$, the sequence
\[ 0 \to R[\frac{1}{x+y}] \to R[\frac{1}{x\cdot(x+y)}] \oplus R[\frac{1}{y\cdot(x+y)}] \to R[\frac{1}{x\cdot y \cdot (x+y)}] \to 0 \]
is exact. 
\end{lemma}
\begin{proof}
	Using the Mayer-Vietoris sequence for $\Spec(R(U)[\frac{1}{x+y}])$ for each $U \in \calX$, one finds that the corresponding sequence of presheaves is exact, as $(x,y) = (1) \in R(U)[\frac{1}{x+y}]$; the claim follows by exactness of sheafification.
\end{proof}

The main relevant consequence is that $R[\frac{1}{x+y}] \in D(\calX,R)$ is represented by a finite complex whose terms are direct sums of filtered colimits of free $R[\frac{1}{x}]$-modules and $R[\frac{1}{y}]$-modules.

\begin{lemma}
	\label{lem:completedualcharac}
	Fix $K \in D(\calX,R)$ and $x \in R(\calX)$. Then $T(K,x) = 0$ if and only if $\underline{\R\Hom}_R(M,K) = 0$ for $M \in D(\calX,R[\frac{1}{x}])$.
\end{lemma}
\begin{proof}
	The backwards direction follows by setting $M = R[\frac{1}{x}]$ and using $R[\frac{1}{x}] = \colim\Big(R \stackrel{x}{\to} R \stackrel{x}{\to} R \to \dots \Big)$. For the forward direction, let $\calC \subset D(\calX,R[\frac{1}{x}])$ be the triangulated subcategory of all $M$ for which $\underline{\R\Hom}_R(M,K) = 0$. Then $\calC$ is closed under arbitrary direct sums, and $R[\frac{1}{x}] \in \calC$ by assumption. Since $T(K|_U,x) = T(K,x)|_U = 0$, one also has $j_!(R[\frac{1}{x}]|_U) \in \calC$ for any $j:U \to 1_{\calX}$. The claim now follows: for any ringed topos $(\calX,A)$, the smallest triangulated subcategory of $D(\calX,A)$ closed under arbitrary direct sums and containing $j_!(A|_U)$ for $j:U \to 1_{\calX}$ variable is $D(\calX,A)$ itself.
\end{proof}

\begin{lemma}
	\label{lem:xdivisible}
	Fix $K \in D(\calX,R)$ and $x \in I(\calX)$. Then $T(K,x)$ lies in the essential image of $D(\calX,R[\frac{1}{x}]) \to D(\calX,R)$.
\end{lemma}
\begin{proof}
We may represent $K$ by a $K$-injective complex of $R$-modules. Then $T(K,x) \simeq \underline{\R\Hom}_R(R[\frac{1}{x}],K) \simeq  \underline{\Hom}_R(R[\frac{1}{x}],K)$ is a complex of $R[\frac{1}{x}]$-modules, which proves the claim.
\end{proof}

\begin{lemma}
\label{lem:conscompletion}
The inclusion $D_{\comp}(\calX,R) \hookrightarrow D(\calX,R)$ admits a left adjoint $K \mapsto \widehat{K}$. The natural map $\widehat{K} \to \widehat{\widehat{K}}$ is an equivalence.
\end{lemma}
\begin{proof}
	The second part is a formal consequence of the first part as the inclusion $D_\comp(\calX,R) \subset D(\calX,R)$ is fully faithful. For the first part, we first assume $I$ is generated by global sections $x_1,\dots,x_r \in I(\calX)$. For $0 \leq i \leq r$, define functors $F_i:D(\calX,R) \to D(\calX,R)$ with maps $F_i \to F_{i+1}$ as follows: set $F_0 = \id$, and
	\[\begin{aligned}
	F_{i+1}(K) := \cok\Big(T(F_i(K),x_{i+1}) \to F_i(K)\Big) &\simeq \R\lim\Big(F_i(K) \stackrel{x_{i+1}^n}{\to} F_i(K)\Big) \\
	&\simeq \R\lim \Big(F_i(K) \otimes_{\Z[x_{i+1}]}^L \Z[x_{i+1}]/(x_{i+1}^n)\Big),
	\end{aligned}\]
	where the transition maps $\Big(F_i(K) \stackrel{x_{i+1}^{n+1}}{\to} F_i(K)\Big) \to \Big(F_i(K) \stackrel{x_{i+1}^n}{\to} F_i(K)\Big)$ are given by $x_{i+1}$ on the source, and the identity on the target. One then checks using induction and lemmas \ref{lem:completedualcharac} and \ref{lem:xdivisible} that $F_i(K)$ is derived $(x_1,\dots,x_i)$-complete, and that
	\[ \R\Hom(F_{i+1}(K),L) = \R\Hom(F_i(K),L) \]
	if $L$ is $(x_1,\dots,x_{i+1})$-complete. It follows that $K \to F_r(K)$ provides the desired left adjoint; we rewrite $\widehat{K} := F_r(K)$ and call it the completion of $K$. The construction shows that completion   commutes with restriction. In general, this argument shows that there is a hypercover $f:X^\bullet \to 1_{\calX}$ such that the inclusion $D_\comp(X^n,R) \to D(X^n,R)$ admits a left adjoint, also called completion. As completion commtues with restriction, the inclusion $D_{\cart,\comp}(X^\bullet,R) \subset D_\cart(X^\bullet,R)$ of derived $I$-complete cartesian complexes inside all cartesian complexes admits a left-adjoint $D_\cart(X^\bullet,R) \to D_{\cart,\comp}(X^\bullet,R)$. The cohomological descent equivalence $f^*:D(\calX,R) \to D(X^\bullet,R)$ restricts to an equivalence $D_\comp(\calX,R) \to D_{\cart,\comp}(X^\bullet,R)$, so the claim follows.
\end{proof}

Lemma \ref{lem:conscompletion} leads to a tensor structure on $D_\comp(\calX,R)$:

\begin{definition}
For $K,L \in D(\calX,R)$, we define the completed tensor product via $K \widehat{\otimes}_R L := \widehat{K \otimes^L_R L} \in D_\comp(\calX,R)$.
\end{definition}

The completed tensor product satisfies the expected adjointness:

\begin{lemma}
	\label{lem:comptensorprodreplete}
For $K \in D(\calX,R)$ and $L \in D_\comp(\calX,R)$, we have $\underline{\R\Hom}_R(K,L) \in D_\comp(\calX,R)$. Moreover, there is an adjunction
\[ \Hom(K', \underline{\R\Hom}_R(K,L)) \simeq \Hom(K' \widehat{\otimes}_R K,L) \]
for any $K' \in D_\comp(\calX,R)$.
\end{lemma}
\begin{proof}
	For any $x \in I(\calX)$, we have $T(\underline{\R\Hom}_R(K,L),x) \simeq \underline{\R\Hom}_R(K,T(L,x)) \simeq 0$. Repeating this argument for a slice topos $\calX_{/U}$ then proves the first part. The second part is a formal consequence of the adjunction between $\otimes$ and $\underline{\R\Hom}$ in $D(\calX,R)$, together with the completeness of $L$.
\end{proof}

\begin{lemma} 
	\label{lem:dercompcharac}
	Fix $K \in D(\calX,R)$. The following are equivalent
\begin{enumerate}
	\item For each $U \in \calX$ and $x \in I(U)$, the natural map $K \to \R\lim \big(K \stackrel{x^n}{\to} K\big)$ is an isomorphism.
	\item $K$ is derived $I$-complete.
	\item There exists a cover $\{U_i \to 1_{\calX}\}$ and generators $x_1,\dots,x_r \in I(U_i)$ such that $T(K|_{U_i},x_i) = 0$.
	\item There exists a cover $\{U_i \to 1_{\calX}\}$ and generators $x_1,\dots,x_r \in I(U_i)$ such that 
		\[ K|_{U_i} \simeq \R\lim(K|_{U_i} \otimes^L_{\Z[x_1,\dots,x_r]} \Z[x_1,\dots,x_r]/(x_1^n,\dots,x_r^n))\]
		via the natural map.
\end{enumerate}
\end{lemma}
\begin{proof}
The equivalence of (1) and (2) follows from the observation that the transition map 
\[ \Big(K \stackrel{x^{n+1}}{\to} K\Big) \to \Big(K \stackrel{x^n}{\to} K\Big)\]
is given by $x$ on the first factor, and the identity on the second factor. Also, (2) clearly implies (3). For the converse, fix a $U \in \calX$ and $x \in I(U)$. To show $T(K|_U,x) = 0$, we are free to replace $U$ with a cover. Hence, we may assume $x = \sum_i a_i x_i$ with $T(K|_U,x_i) = 0$. Lemma \ref{lem:completedualcharac} shows  $T(K|_U,a_ix_i) = 0$, and Lemma \ref{lem:invertsum} does the rest. Finally, since each $x_j$ acts nilpotently on $K|_{U_i} \otimes^L_{\Z[x_1,\dots,x_r]} \Z[x_1,\dots,x_r]/(x_1^n,\dots,x_r^n)$, it is clear that (4) implies (3). Conversely, assume (3) holds. Replacing $\calX$ with a suitable $U_i$, we may assume $I$ is generated by global sections $x_1,\dots,x_r \in I(\calX)$. Consider the sequence of functors $F_0,\dots,F_r:D(\calX,R) \to D(\calX,R)$ defined in the proof of Lemma \ref{lem:conscompletion}. As each $\Z[x_i]/(x_i^n)$ is a perfect $\Z[x_i]$-module, the functor $- \otimes^L_{\Z[x_i]} \Z[x_i]/(x_i^n)$ commutes with homotopy-limits. Hence, we can write
\[ K \simeq F_r(K) \simeq \R\lim(K \otimes^L_{\Z[x_1]} \Z[x_1]/(x_1^n) \otimes^L_{\Z[x_2]} \Z[x_2]/(x_2^n) \otimes \dots \otimes^L_{\Z[x_r]} \Z[x_r]/(x_r^n)), \]
which implies (4).
\end{proof}

\begin{lemma}
	\label{lem:classcompdercomp}
If $M \in \Mod_R$ is classically $I$-complete, then $M$ is derived $I$-complete.
\end{lemma}
\begin{proof}
Commuting limits shows that the collection of all derived $I$-complete objects $K \in D(\calX,R)$ is closed under homotopy-limits. Hence, writing $M = \lim M/I^nM \simeq \R\lim M/I^nM$ (where the second isomorphism uses repleteness), it suffices to show that $M$ is derived $I$-complete if $I^n M = 0$. For such $M$, any local section $x \in I(U)$ for some $U \in \calX$ acts nilpotently on $M|_U$, so $T(M|_U,x) = 0$.
\end{proof}

The cokernel of a map of classically $I$-complete $R$-modules need not be $I$-complete, and one can even show that $\Mod_{R,\comp}$ is not an abelian category in general. In contrast, derived $I$-complete modules behave much better:

\begin{lemma}
	\label{lem:dercompabcat}
	The collection of all derived $I$-complete $M \in \Mod_R$ is an abelian Serre subcategory of $\Mod_R$.
\end{lemma}
\begin{proof}
Fix a map $f:M \to N$ of derived $I$-complete $R$-modules. Then there is an exact triangle
\[ \ker(f)[1] \to \Big(M \to N\Big) \to \cok(f) \]
For any $x \in I(\calX)$, there is an exact triangle
\[ T(\ker(f)[1],x) \to 0 \to T(\cok(f),x) \]
where we use the assumption on $M$ and $N$ to get the middle term to be $0$. The right hand side lies in $D^{\geq 0}(\calX,R)$, while the left hand side lies in $D^{\leq 0}(\calX,R)$ as $\R\lim$ has cohomological dimension $\leq 1$ (as $\calX$ is replete). Chasing sequences shows that the left and right terms are also $0$. Repeating the argument for a slice topos $\calX_{/U}$ (and varying $x \in I(U)$) proves that $\ker(f)$ and $\cok(f)$ are derived $I$-complete. It is then immediate that $\im(f) = M/\ker(f)$ is also derived $I$-complete. Since closure of derived $I$-completeness under extensions is clear, the claim follows.
\end{proof}

\begin{proof}[Proof of Proposition \ref{prop:complexcompcrit}]
	Assume first that each $H^i(K)$ is derived $I$-complete. Then each finite truncation $\tau^{\leq n} \tau^{\geq m} K$ is derived $I$-complete. Hence, $\tau^{\leq m} K \simeq \R\lim \tau^{\geq -n} \tau^{\leq m} K$ is also derived $I$-complete for each $m$; here we use that $D(\calX)$ is left-complete since $\calX$ is replete. For any $x \in I(\calX)$, applying $T(-,x)$ to
	\[ \tau^{\leq m} K \to K \to \tau^{\geq m+1}(K).\]
shows that $T(K,x) \simeq T(\tau^{\geq m+1} K,x) \in D^{\geq m+1}(\calX,R)$. Since this is true for all $m$, one has $T(K,x) = 0$. Repeating the argument for $x \in I(U)$ for $U \in \calX$ then proves the claim.

	Conversely, assume that $K$ is derived $I$-complete. By shifting, it suffices to show that $H^0(K)$ is derived $I$-complete. Assume first that $K \in D^{\leq 0}(\calX,R)$. Then there is an exact triangle
	\[ \tau^{\leq -1} K \to K \to H^0(K).\]
	Fixing an $x \in I(\calX)$ and applying $T(-,x)$ gives
	\[ T(\tau^{\leq -1} K,x) \to T(K,x) \to T(H^0(K),x).\]
	The left term lives in $D^{\leq 0}(\calX,R)$, the middle term vanishes by assumption on $K$, and the right term lives in $D^{\geq 0}(\calX,R)$, so the claim follows by chasing sequences (and replacing $\calX$ with $\calX_{/U}$). Now applying the same argument to the triangle
	\[ \tau^{\leq 0} K \to K \to \tau^{\geq 1} K\]
	shows that each $\tau^{\leq 0} K$ and $\tau^{\geq 1} K$ are derived $I$-complete. Replacing $K$ by $\tau^{\leq 0} K$ then proves the claim.
\end{proof}

\begin{proof}[Proof of Proposition \ref{prop:classcompdercompcrit}]
	The forward direction follows from Lemma \ref{lem:classcompdercomp}. Conversely, assume $M$ is derived $I$-complete and $I$-adically separated. To show $M$ is classically $I$-complete, we may pass to slice topoi and assume that $I$ is generated by global sections $x_1,\dots,x_r \in I(\calX)$. Then derived $I$-completeness of $M$ gives
	\[ M \simeq \R\lim(M \otimes^L_{\Z[x_1,\dots,x_r]} \Z[x_1,\dots,x_r]/(x_i^n)).\]
	Calculating $H^0(M) \simeq M$ via the Milnor exact sequence (which exists by repleteness) gives
	\[ 1 \to \R^1 \lim H^{-1}(M \otimes^L_{\Z[x_1,\dots,x_r]} \Z[x_1\dots,x_r]/(x_i^n)) \to M \to \lim M/(x_1^n,\dots,x_r^n)M \to 1.\]
	By $I$-adic separatedness, the last map is injective, and hence an isomorphism.
\end{proof}

\subsection{Derived completions of noetherian rings in a replete topos}
\label{subsec:derivedcompnoeth}

In this section, we specialize the discussion of \S \ref{subsec:derivedcompfadic} to the noetherian constant case. More precisely, we fix a replete topos $\calX$, a noetherian ring $R$, and an ideal $\fram \subset R$. We also write $\fram \subset R$ for the corresponding constant sheaves on $\calX$. Our goal is to understand $\fram$-adic completeness for $R$-complexes on $\calX$.

\begin{proposition}
	\label{prop:derivedcompnoetherian}
Fix $K \in D(\calX,R)$. Then
\begin{enumerate}
	\item $K$ is derived $\fram$-complete if and only if $K \simeq \R\lim(K \otimes_R^L R/\fram^n)$ via the natural map. 
	\item $\R\lim(K \otimes_R^L R/\fram^n)$ is derived $\fram$-complete.
	\item The functor $K \mapsto \R\lim(K \otimes_R^L R/\fram^n)$ defines a left adjoint $D(\calX,R) \to D_\comp(\calX,R)$ to the inclusion.
\end{enumerate}
\end{proposition}
\begin{proof}
(2) is clear as each $K \otimes_R^L R/\fram^n$ is derived $\fram$-complete. For the rest, fix generators $f_1,..,f_r \subset \fram$. Set $P = \Z[x_1,\dots,x_r]$, and $J = (x_1,\dots,x_r) \subset P$. Consider the map $P \to R$ defined via $x_i \mapsto f_i$ (both in $\Set$ and $\calX$). By Lemma \ref{lem:dercompcharac}, $K$ is derived $\fram$-complete precisely when $K \simeq \R\lim(K \otimes_P^L P/J^n)$ via the natural map. For (1), it thus suffices to check that 
	\[ a:\{P/J^n \otimes_P^L R\} \to \{R/\fram^n\} \]
	is a strict pro-isomorphism. There is an evident identification
	\[ \{P/J^n \otimes_P^L R\} = \{P/J^n \otimes_P^L (P \otimes_\Z R) \otimes_{P \otimes_\Z R}^L R \},\]
where $P \otimes_\Z R$ is viewed as a $P$-algebra via the first factor. As $P/J^n$ and $P \otimes_\Z R$ are $\Tor$-independent over $P$, we reduce to checking that
\[ \{ R[x_1,\dots,x_r]/(x_1,\dots,x_r)^n  \otimes^L_{R[x_1,\dots,x_r]} R \} \to \{R/\fram^n\} \]
	is a strict pro-isomorphism. This follows from the Artin-Rees lemma. Finally, (3) follows from $a$ being a pro-isomorphism as the construction of Lemma \ref{lem:conscompletion} realises the $\fram$-adic completion of $K$ as $\R\lim(K \otimes_P^L P/J^n)$.
\end{proof}

Proposition \ref{prop:derivedcompnoetherian} gives a good description of the category $D_\comp(\calX,R)$ of derived $\fram$-complete complexes. Using this description, one can check that $R$ itself is {\em not} derived $\fram$-complete in $\calX$ in general. To rectify this, we study the $\fram$-adic completion $\widehat{R}$ of $R$ on $\calX$, and some related categories.

\begin{definition}
Define $\widehat{R} := \lim R/\fram^n \in \calX$.  In particular, $\widehat{R}$ is an $R$-algebra equipped with $R$-algebra maps $\widehat{R} \to R/\fram^n$. An object $K \in D(\calX, \widehat{R})$ is called {\em $\fram$-adically complete} if the natural map $K \to \R\lim(K \otimes_{\widehat{R}}^L R/\fram^n)$ is an equivalence. Let $i:D_{\comp}(\calX,\widehat{R}) \hookrightarrow D(\calX, \widehat{R})$ be the full subcategory of such complexes.
\end{definition}

Our immediate goal is to describe $\fram$-adically complete complexes in terms of their truncations. To this end, we introduce the following category of compatible systems:

\begin{definition}
	\label{def:compatiblesystem}
	Let $\calC = \Fun(\N^\opp,\calX)$ be the topos of $\N^\opp$-indexed projective systems $\{F_n\}$ in $\calX$. Let $R_\bullet = \{R/\fram^n\} \in \calC$ be the displayed sheaf of rings, and let $D_{\comp}(\calC, R_\bullet) \subset D(\calC,R_\bullet)$  be the full subcategory spanned by complexes $\{K_n\}$ such that the induced maps $K_n \otimes^L_{R/\fram^n} R/\fram^{n-1} \to K_{n-1}$ are equivalences for all $n$.
\end{definition}

\begin{lemma}
	\label{lem:proisomchangecoeff}
	For $\{K_n\} \in D^{-}(\calC,R_\bullet)$, one has an identification of pro-objects $\{K_n \otimes_R^L R/\fram\} \simeq \{K_n \otimes_{R/\fram^n}^L R/\fram\}$, and hence a limiting isomorphism $\R\lim(K_n \otimes_R^L R/\fram) \simeq \R\lim(K_n \otimes_{R/\fram^n}^L R/\fram)$. If $\fram$ is regular, this extends to unbounded complexes.
\end{lemma}
\begin{proof}
	Change of rings gives $\{K_n \otimes_R^L R/\fram\} \simeq \{K_n \otimes_{R/\fram^n}^L R/\fram^n \otimes_R^L R/\fram\}$. The Artin-Rees lemma shows that $\{R/\fram^n \otimes_R^L R/\fram\} \to \{R/\fram\}$ is a pro-isomorphism. Since $\{K_n\}$ is bounded above, the spectral sequence for $\Tor$ has only finitely many contributing terms to a given $E_\infty$-term, and hence
	\[ \{K_n \otimes_R^L R/\fram\} \simeq \{K_n \otimes_{R/\fram^n}^L R/\fram^n \otimes_R^L R/\fram\} \to \{K_n \otimes_{R/\fram^n}^L R/\fram\} \]
	is also a pro-isomorphism. Applying $\R\lim$ and using repleteness then gives the claim. Finally, if $\fram$ is generated by a regular sequence $(f_1,\dots,f_r)$, then $\{R/\fram^n\}$ is pro-isomorphic to $\{R/(f_1^n,\dots,f_r^n)\}$. Each quotient $R/(f_1^n,\dots,f_r^n)$ is $R$-perfect, and hence the $\Tor$-spectral sequence calculating $\calH^i(K \otimes_R^L R/(f_1^n,\dots,f_r^n))$ has only finitely many non-zero terms even when $K$ is unbounded, so the preceding argument applies.
\end{proof}

\begin{lemma}
	\label{lem:truncationofcomplete}
	For $\{K_n\} \in D^{-}_\comp(\calC,R_\bullet)$, the natural map gives $(\R\lim K_n) \otimes^L_R R/\fram^k \simeq K_k$ for $k \geq 0$. If $\fram$ is regular, this extends to unbounded complexes.
\end{lemma}
\begin{proof}
	By devissage and the completeness of $\{K_n\}$, we may assume $k=1$. By shifting, we can also assume $\{K_n\} \in D^{\leq 0}(\calC)$, i.e., $K_n \in D^{\leq 0}(\calX)$ for all $n$. Fix an integer $i \geq 0$, and an $R$-perfect complex $P_i$ with a map $P_i \to R/\fram$ whose cone lies in $D^{\leq -i}(R)$. Then there is a commutative diagram
	\[ \xymatrix{ (\R\lim K_n) \otimes_R^L P_i \ar[r]^-a \ar[d]^-b & \R\lim(K_n \otimes_R P_i) \ar[d]^-d \\
	(\R\lim K_n) \otimes_R^L R/\fram \ar[r]^-c & \R\lim(K_n \otimes_R R/\fram) \simeq K_1. } \]
The isomorphism on the bottom right is due to Lemma \ref{lem:proisomchangecoeff}. As $P_i$ is perfect, $a$ is an isomorphism. Moreover, $\cok(b) \in D^{\leq -i+1}(\calX)$ as $\R\lim K_n \in D^{\leq 1}(\calX)$ by repleteness. A similar argument also shows $\cok(d) \in D^{\leq -i+1}(\calX )$. Hence, $\cok(c) \in D^{\leq -i+1}(\calX)$. Then $c$ must be an isomorphism as this is true for all $i$.
\end{proof}

We can now show that the two notions of completeness coincide:

\begin{lemma}
	\label{lem:arcomplete}
For each $m$, the natural map induces $\widehat{R} \otimes^L_R R/\fram^m \simeq R/\fram^m$. In particular, $D_\comp(\calX,\widehat{R}) \simeq D_\comp(\calX,R)$.
\end{lemma}
\begin{proof}
	The first part follows from Lemma \ref{lem:truncationofcomplete}. The second part follows formally from this and Proposition \ref{prop:derivedcompnoetherian}.
\end{proof}

We now show that an $\fram$-adically complete complex is determined by its reductions modulo powers of $\fram$; this will be used later to compare complexes on the pro-\'etale site to Ekedahl's category of adic complexes.

\begin{lemma} 
	\label{lem:completecompatiblesystems}
	With notation as above, we have:
	\begin{enumerate}
		\item There is a map $\pi:(\calC, R_\bullet) \to (\calX, \widehat{R})$ of ringed topoi given by $\pi_*(\{F_n\}) = \lim F_n$ with $\pi^{-1} \widehat{R} \to R_\bullet$ the natural map. 
		\item Pullback under $\pi$ induces a fully faithful functor $\pi^*:D_\comp(\calX, \widehat{R}) \to D_\comp(\calC,R_\bullet)$.
		\item Pushforward under $\pi$ induces a fully faithful functor $\pi_*:D^{-}_\comp(\calC,R_\bullet) \to D^{-}_\comp(\calX, \widehat{R})$. 
		\item $\pi$ induces an equivalence $D^{-}_\comp(\calX,\widehat{R}) \simeq D^{-}_\comp(\calC,R_\bullet)$.
		\item If $\fram$ is regular, then (3) and (4) extend to the unbounded case.
	\end{enumerate}
\end{lemma}
\begin{proof}
	(1) is clear. The functor $\pi^*:D(\calX, \widehat{R}) \to D(\calC,R_\bullet)$ is given by $K \mapsto \{K \otimes_{\widehat{R}} R/\fram^n\}$, while $\pi_*:D(\calC,R_\bullet) \to D(\calX, \widehat{R})$ is given by $\pi_* (\{K_n\}) \simeq \R\lim K_n$. It is then clear that $\pi^*$ carries complete complexes to complete ones. Given $\{K_n\} \in D_\comp(\calC,R_\bullet)$, each $K_n \in D(\calX, R/\fram^n)$ is derived $\fram$-complete, and hence $\pi_*$ preserves completeness as well (since $\pi_* \{K_n\} := \R\lim K_n$ is $\fram$-adically complete). For (2), it then suffices to check that $K \simeq \R\lim(K \otimes_{\widehat{R}}^L R/\fram^n)$ for any $K \in D_\comp(\calX,\widehat{R})$, which is true by Proposition \ref{prop:derivedcompnoetherian}. Lemma \ref{lem:truncationofcomplete} and (2) immediately give (3), and hence (4). Finally, (5) follows by the same argument as (3) as all the ingredients in the proof of the latter extend to the unbounded setting if $\fram$ is regular.
\end{proof}

\newpage

\section{The pro-\'etale topology}\label{sec:ProEtale}

We define the pro-\'etale site of a scheme in \S \ref{ss:proetsite}, and study the associated topos in \S \ref{ss:proettopos}. In \S \ref{ss:proetpoint}, we use these ideas to construct a variant of Tate's continuous cohomology of profinite groups that behaves better in some functorial respects.

\subsection{The site}
\label{ss:proetsite}

\begin{definition} A map $f:Y \to X$ of schemes is called {\em weakly \'etale} if $f$ is flat and $\Delta_f: Y\to Y \times_X Y$ is flat. Write $X_\proet$ for the category of weakly \'etale $X$-schemes, which we give the structure of a site by declaring a cover to be one that is a cover in the fpqc topology, i.e. a family $\{Y_i \to Y\}$ of maps in $X_\proet$ is a covering family if any open affine in $Y$ is mapped onto by an open affine in $\sqcup_i Y_i$.
\end{definition}

\begin{remark}
	To avoid set-theoretic issues, it suffices for our purposes to define the site $X_\proet$ using weakly \'etale maps $Y \to X$ with $|Y| < \kappa$, where $\kappa$ is a fixed uncountable strong limit cardinal larger than $|X|$.\footnote{Recall that a cardinal $\kappa$ is a strong limit cardinal if for any $\gamma<\kappa$, $2^\gamma<\kappa$.} The choice of $\kappa$ is dictated by the desire to have $\Shv(X_\proet)$ be locally weakly contractible. Increasing $\kappa$ results in a different topos, but cohomology remains the same, as it can be calculated by a simplicial covering with w-contractible schemes.
\end{remark}

\begin{remark} We do not directly work with pro-\'etale morphisms of schemes to define $X_\proet$ as the property of being pro-\'etale is not geometric:  Example \ref{ex:proetalenotlocal} shows its failure to localise on the target. Nonetheless, we call $X_\proet$ the pro-\'etale site, as by Theorem \ref{t:WeaklyVsIndEtale} any weakly \'etale map $f:Y\to X$ is Zariski locally on $X$ and locally in $Y_\proet$ of the form $\Spec B\to \Spec A$ with $A\to B$ ind-\'etale. 
\end{remark}

Some elementary examples of weakly \'etale maps:

\begin{example} For a field $k$, a map $\Spec(R) \to \Spec(k)$ is weakly \'etale if and only if $k \to R$ is ind-\'etale. Indeed, $R$ embeds into some ind-\'etale $k$-algebra $S$; but one checks easily that as $k$ is a field, any subalgebra of an ind-\'etale $k$-algebra is again ind-\'etale.
\end{example}

\begin{example}	For a scheme $X$ and a geometric point $x$, the map $\Spec(\calO_{X,x}^{sh}) \to X$ from the strict henselization is weakly \'etale; similarly, the henselization and Zariski localizations are also weakly \'etale.
\end{example}

We begin by recording some basic generalities on pro-\'etale maps.

\begin{lemma} Compositions and base changes of weakly \'etale maps are weakly \'etale.
\end{lemma}
\begin{proof} Clear.
\end{proof}

\begin{lemma} Any map in $X_\proet$ is weakly \'etale.
\end{lemma}
\begin{proof} This follows from Proposition \ref{p:PropWeaklyEtale} (iv).
\end{proof}

The previous observations give good categorical properties for $X_\proet$:

\begin{lemma} The category $X_\proet$ has finite limits, while the full subcategory spanned by affine weakly \'etale maps $Y \to X$ has all small limits. All limits in question agree with those in $\Sch_{/X}$.
\end{lemma}

\begin{proof}	For the first part, it suffices to show that $X_\proet$ has a final object and arbitrary fibre products. Clearly $X$ is a final object. Moreover, if $Y_1 \to Y_2 \gets Y_3$ is a diagram in $X_\proet$, then both maps in the composition $Y_1 \times_{Y_2} Y_3 \to Y_i \to X$ are weakly \'etale for any $i \in \{1,2,3\}$ by the previous lemmas, proving the claim. For the second part, the same argument as above shows finite limits exist. Hence, it suffices to check that small cofiltered limits exist, but this is clear: the limit of a cofiltered diagram of affine weakly \'etale $X$-schemes is an affine $X$-scheme that is weakly \'etale over $X$ as flatness is preserved under filtered colimits of rings.
\end{proof}

We record an example of a typical ``new'' object in $X_\proet$:

\begin{example}\label{ex:tensorprofiniteset} The category $X_\proet$ is ``tensored over'' profinite sets, i.e., given a profinite set $S$ and $Y \in X_\proet$, one can define $Y \otimes S \in X_\proet$ as follows. Given $S = \lim_i S_i$ as a cofiltered limit of finite sets, we obtain constant $X$-schemes $\underline{S_i} \in X_\et \subset X_\proet$ with value $S_i$. Set $\underline{S} = \lim_i \underline{S_i}$, and $Y \otimes S := Y \times_X \underline{S}$. If $X$ is qcqs, then for any finitely presented $X$-scheme $U$, one has $\Hom_X(Y \otimes S,U) = \colim_i \Hom_X(Y \otimes {S_i},U) = \colim_i \prod_{S_i} \Hom_X(Y,U)$. The association $S \mapsto \underline{S}$ defines a limit preserving functor from profinite sets to $X_\proet$.
\end{example}

Using these objects, we can describe the pro-\'etale site of a field explicitly:

\begin{example} 
	\label{ex:proetalegeometricpoint}
	Fix a field $k$. If $\overline{k}$ is a separable closure, then the qcqs objects in $\Spec(\overline{k})_\proet$ identify with the category of profinite sets via the functor $Y \mapsto Y(\overline{k})$ with inverse $S \mapsto \underline{S}$ (in the notation of Example \ref{ex:tensorprofiniteset}). The map $\Spec(\overline{k}) \to \Spec(k)$ is a weakly \'etale $\underline{G}$-torsor, so the qcqs objects in $\Spec(k)_\proet$ identify with pro-objects in the category of finite discrete $G$-sets, i.e., with the category of profinite continuous $G$-sets. Under this identification, a family $\{S_i \to S\}$ of continuous $G$-equivariant map of such sets is a covering family if there exists a finite subset $J$ of the indices such that $\sqcup_{j \in J} S_j \to S$ is surjective. To see this, we may assume $k = \overline{k}$. Given such a family $\{S_i \to S\}$, the corresponding map $\sqcup_{j \in J} \underline{S_j} \to \underline{S}$ is a surjective weakly \'etale map of affines, so $\{\underline{S_i} \to \underline{S}\}$ is a covering family in $\Spec(\overline{k})_\proet$; the converse is clear. Evaluation on $\underline{S}$ is exact precisely when $S$ is extremally disconnected; note that this functor is not a topos-theoretic point as it does not commute with finite coproducts (though it does commute with filtered colimits and all limits). 
\end{example}

\begin{remark}
	\label{rmk:oldproet}
	The site $X_\proet$ introduced in this paper differs from the one in \cite{ScholzepHT}, temporarily denoted $X'_{\proet}$. More precisely, there is a natural map $\mu_X:\Shv(X_\proet) \to \Shv(X'_\proet)$ of topoi, but $\mu_X$ is {\em not} an equivalence: $\mu_{X,\ast}$ is fully faithful, but there are more objects in $\Shv(X'_\proet)$. This is evident from the definition, and can be seen directly in Example \ref{ex:proetalegeometricpoint} when $X = \Spec(k)$ with $k$ an algebraically closed field. Indeed, both the categories $X_\proet$ and $X'_\proet$ are identified with the category of profinite sets, but $X_\proet$ has more covers than $X'_\proet$: all objects of $X'_\proet$ are weakly contractible,  while the weakly contractible ones in $X_\proet$ are exactly the ones corresponding to extremally disconnected profinite sets.
\end{remark}

The following example (due to de Jong) shows that the property of being pro-\'etale is not Zariski local on the target, and hence explains why weakly \'etale maps give a more geometric notion:

\begin{example}
	\label{ex:proetalenotlocal}	
	Let $S'$ be an infinite set with an automorphism $T':S' \to S'$ which does not stabilize any finite subset; for example, $S' = \Z$, and $T'(n) = n+1$. Write $(S,0)$ for the one point compactification of $S'$ and $T:S \to S$ for the induced automorphism (which has a unique fixed point at $0$); note that $S$ is profinite, and the unique non-empty clopen subset of $S$ stable under $T$ is $S$ itself.  Let $X \subset \A^2_\C$ be the union of two irreducible smooth curves $X_1$ and $X_2$ meeting transversely at points $p$ and $q$; note that $X$ is connected. Glueing $\underline{S} \otimes X_1 \in X_{1,\proet}$ to $\underline{S} \otimes X_2 \in X_{2,\proet}$ using the identity at $p$ and $T$ at $q$ gives $Y \in X_\proet$. We claim that $Y$ is not pro-\'etale over $X$. Assume otherwise that $Y = \lim_i Y_i \to X$ with $f_i:Y_i \to X$ \'etale. Let $0:X \to Y$ be the zero section, defined using $0 \in S$. Then the assumption on $Y$ shows that $0(X) = \cap U_i$ with $U_i \subset Y$ a clopen subset (pulled back from a clopen of $Y_i$). Now any clopen subset $U \subset Y$ defines a clopen subset $U_p \subset S$ that is stable under $T$, so $U_p = S$ is the only possibility by choice of $S$ and $T$; this gives $\{0\} = 0(X)_p = \cap_i S = S$, which is absurd.
\end{example}

We end by giving examples of covers in $X_\proet$.

\begin{example} Given a scheme $X$ and closed geometric points $x_1,\dots,x_n$, the map 
\[ \Big(\sqcup_i \Spec(\calO_{X,x_i}^{sh})\Big) \sqcup \Big(X-\{x_1,\dots,x_n\}\Big) \to X\]
is a weakly \'etale cover.  However, one cannot add infinitely points. For example, the map
\[ \sqcup_p \Spec(\Z_{(p)}^{sh}) \to \Spec(\Z) \]
is {\em not} a weakly \'etale cover as the target is not covered by a quasicompact open in the source.
\end{example}

\subsection{The topos}
\label{ss:proettopos}

To effectively study $\Shv(X_\proet)$, we single out a special class of weakly \'etale morphisms to serve as generators:

\begin{definition}
	Fix a scheme $X$. An object $U \in X_\proet$ is called a {\em pro-\'etale affine} if we can write $U = \lim_i U_i$ for a small cofiltered diagram $i \mapsto U_i$ of affine schemes in $X_\et$;  the expression $U = \lim_i U_i$ is called a {\em presentation} for $U$, and we often implicitly assume that the indexing category has a final object $0$. The full subcategory of $X_\proet$ spanned by pro-\'etale affines is denoted $X^\aff_\proet$.
\end{definition}

We remark that each $U \in X^\aff_\proet$ is, in particular, an affine scheme pro-\'etale over $X$.

\begin{lemma}
	\label{lem:mapsproetaff}
	Any map in $X^\aff_\proet$ is pro-(affine \'etale).
\end{lemma}
\begin{proof}
Fix a map $h:U \to V$ in $X^\aff_\proet$, and presentations $U = \lim_i U_i$ and $V = \lim_j V_j$ as pro-\'etale affines. Then, after changing the presentation for $U$, we may assume that $X = V_0$ is an affine scheme $\Spec(A)$.  The claim now follows from the observation that a map between ind-\'etale $A$-algebras is also ind-\'etale.
\end{proof}

\begin{remark} 
	By Lemma \ref{lem:mapsproetaff}, the category $X^\aff_\proet$ admits limits indexed by a connected diagram, and these agree with those in $\Sch_{/X}$. However, this category does not have a final object (unless $X$ is affine) or non-empty finite products (unless $X$ has an affine diagonal).
\end{remark}

The reason to introduce pro-\'etale affines is:

\begin{lemma}\label{lem:baseproet} The site $X_\proet$ is subcanonical, and the topos $\Shv(X_\proet)$ is generated by $X^\aff_\proet$.
\end{lemma}

\begin{proof}	The first part comes from fpqc descent. The second assertion means that any $Y \in X_\proet$ admits a surjection $\sqcup_i U_i \to Y$ in $X_\proet$ with $U_i \in X^\aff_\proet$, which follows from Theorem \ref{t:WeaklyVsIndEtale}.
\end{proof}

We record some consequences of the above observations on pro-\'etale maps for the pro-\'etale site:

\begin{remark}
	\label{rmk:affineproet}
	Assume $X$ is an affine scheme. Then $X^\aff_\proet$ is simply the category of all affine schemes pro-\'etale over $X$; this category admits all small limits, and becomes a site with covers defined to be fpqc covers. Lemma \ref{lem:baseproet} then shows that $\Shv(X_\proet) \simeq \Shv(X^\aff_\proet)$.
\end{remark}

\begin{lemma}\label{lem:sheafcriterion}	A presheaf $F$ on $X_\proet$ is a sheaf if and only if:
\begin{enumerate}
	\item For any surjection $V \to U$ in $X^\aff_\proet$, the sequence $F(U) \to \equalizer{F(V)}{F(V \times_U V)}$ is exact.
	\item The presheaf $F$ is a Zariski sheaf.
\end{enumerate}
\end{lemma}
\begin{proof} The forward direction is clear. Conversely, assume $F$ is a presheaf satisfying (1) and (2), and fix a cover $Z\to Y$ in $X_\proet$. Using (1) and (2), one readily checks the sheaf axiom in the special case where $Y \in X^\aff_\proet$, and $Z = \sqcup_i W_i$ with $W_i \in X^\aff_\proet$. In the case of a general cover, Lemma \ref{lem:baseproet} shows that we can find a diagram
	\[ \xymatrix{ \sqcup_{j \in J} U_j \ar[r]^-a \ar[d]^-b & Z \ar[d]^-c \\
	\sqcup_{i \in I} V_i \ar[r]^-d & Y } \]
	where $d$ is a Zariski cover, $a$ and $b$ are covers in $X_\proet$, and $U_j, V_i \in X^\aff_\proet$ with $b$ determined by a map $h:J \to I$ of index sets together with maps $U_j \to V_{h(j)}$ in $X^\aff_\proet$. The previous reduction and (2) give the sheaf axiom for $b$ and $d$, and hence $d \circ b$ as well. It formally follows that $F(Y) \to F(Z)$ is injective, and hence that $F(Z) \to \prod_i F(U_i)$ is also injective by (2) as $a$ is a cover. A diagram chase then shows that the sheaf axiom for $c$ follows from that for $c \circ a$.
\end{proof}

\begin{lemma}\label{lem:proettoposbc} For any $Y \in X_\proet$, pullback induces an identification $\Shv(X_\proet)_{/Y} \simeq \Shv(Y_\proet)$.
\end{lemma}

\begin{proof}
A composition of weakly \'etale maps is weakly \'etale, and any map between weakly \'etale maps is weakly \'etale.
\end{proof}

The pro-\'etale topos is locally weakly contractible in the sense of Definition \ref{rmk:locallyweaklycontractible}.

\begin{proposition}
	\label{prop:proetlocallycontractible}
For any scheme $X$, the topos $\Shv(X_\proet)$ is locally weakly contractible.
\end{proposition}
\begin{proof}
This follows immediately from Lemma \ref{lem:cwcontractiblecover} since any affine $U \in X_\proet$ is coherent.
\end{proof}

\begin{remark} Proposition \ref{prop:proetlocallycontractible} gives a recipe for calculating the pro-\'etale homotopy type $|X|$ of a qcqs scheme $X$. Namely, if $f:X^\bullet \to X$ is a hypercover in $X_\proet$ with each $X^n$ being w-contractible, then $|X| = |\pi_0(X^\bullet)|$; any two such choices of $f$ are homotopic, and hence $|X|$ is well-defined in the category of simplicial profinite sets up to continuous homotopy.
\end{remark}

We give an example illustrating the behaviour of constant sheaves on the pro-\'etale site:

\begin{example}	Fix a connected affine scheme $X$, and a profinite set $S = \lim_i S_i$ with $S_i$ finite. By the formula in Example \ref{ex:tensorprofiniteset}, the constant sheaf $\underline{A} \in \Shv(X_\proet)$ associated to a set $A$ satisfies
\[ \underline{A}(X \otimes S) = \colim_i \big(A^{S_i}\big).\]
In particular, the functor $A \mapsto \underline{A}$ is not compatible with inverse limits.
\end{example}

The following example shows classical points do not detect non-triviality in $\Shv(X_\proet)$.

\begin{example} Fix an algebraically closed field $k$, and set $X = \Spec(k)$. Then $\Shv(X_\proet)$ identifies with the topos of sheaves on the category of profinite sets $S$ as explained in Example \ref{ex:proetalegeometricpoint}. Consider the presheaf $G$ (resp. $F$) which associates to such an $S$ the group of all locally constant (resp. all) functions $S \to \Lambda$ for some abelian group $\Lambda$. Then both $F$ and $G$ are sheaves: this is obvious for $G$, and follows from the compatibility of limits in profinite sets and sets for $F$. Moreover, $G \subset F$, and $Q := F/G \in \Ab(X_\proet)$ satisfies $Q(X) = 0$, but $Q(S) \neq 0$ for $S$ not discrete. 
\end{example}

In fact, more generally, one can define 'constant sheaves' associated with topological spaces. Indeed, let $X$ be any scheme, and let $T$ be some topological space.

\begin{lemma}\label{l:ExSheafTopSpace} The association mapping any $U\in X_\proet$ to $\Map_\cont(U,T)$ is a sheaf $\mathcal{F}_T$ on $X_\proet$. If $T$ is totally disconnected and $U$ is qcqs, then $\mathcal{F}_T(U) = \Map_\cont(\pi_0(U),T)$. In particular, if $T$ is discrete, then $\mathcal{F}_T$ is the constant sheaf associated with $T$.
\end{lemma}

\begin{proof} To show that $\mathcal{F}_T$ is a sheaf, one reduces to proving that if $f: A\to B$ is a faithfully flat ind-\'etale morphism of rings, then $M\subset \Spec A$ is open if and only if $(\Spec f)^{-1}(M)\subset \Spec B$ is open. Only the converse is nontrivial, so assume $(\Spec f)^{-1}(M)\subset \Spec B$ is open. First, we claim that $M$ is open in the constructible topology. Indeed, the map $\Spec f: \Spec B\to \Spec A$ is a continuous map of compact Hausdorff spaces when considering the constructible topologies. In particular, it is closed, so
\[
\Spec A\setminus M = (\Spec f)(\Spec B\setminus (\Spec f)^{-1}(M))
\]
is closed, and thus $M$ is open (in the constructible topology). To check that $M$ is actually open, it is enough to verify that $M$ is closed under generalizations. This is clear, as $\Spec f$ is generalizing, and $(\Spec f)^{-1}(M)$ is open (and thus closed under generalizations).

If $T$ is totally disconnected and $U$ is qcqs, then any continuous map $U\to T$ will necessarily factor through the projection $U\to \pi_0(U)$, so that $\mathcal{F}_T(U) = \Map_\cont(\pi_0(U),T)$.
\end{proof}

We relate sheaves on $X$ with sheaves on its space $\pi_0(X)$ of connected components. Recall that if $X$ is a qcqs scheme, then $\pi_0(X)$ is a profinite set. If $\pi_0(X)_\proet$ denotes the site of profinite $\pi_0(X)$-sets as in Example \ref{ex:proetalegeometricpoint}, then the construction of Lemma \ref{lem:findprofinitecovers} defines a limit-preserving functor $\pi^{-1}:\pi_0(X)_\proet \to X_\proet$ which respects coverings. Hence, one has an induced map $\pi:\Shv(X_\proet) \to \Shv(\pi_0(X)_\proet)$ of topoi. This map satisfies:

\begin{lemma}
	\label{lem:pullbackfrompi0}
Assume $X$ is qcqs, and let $\pi:\Shv(X_\proet) \to \Shv(\pi_0(X)_\proet)$ be as above. Then
\begin{enumerate}
	\item $\pi^* F(U) = F(\pi_0(U))$ for any qcqs $U \in X_\proet$ and $F \in \Shv(\pi_0(X)_\proet)$. 
	\item $\pi^*$ commutes with limits.
	\item $\pi^*$ is fully faithful, so $\pi_* \pi^* \simeq \id$.
	\item $\pi^*$ identifies $\Shv(\pi_0(X)_\proet)$ with the full subcategory of those $G \in \Shv(X_\proet)$ such that $G(U) = G(V)$ for any map $U \to V$ of qcqs objects in $X_\proet$ inducing an isomorphism on $\pi_0$.
\end{enumerate}
\end{lemma}
\begin{proof}
	All schemes appearing in this proof are assumed qcqs. (2) is automatic from (1). For (1), fix some $F \in \Shv(\pi_0(X)_\proet)$. As any continuous $\pi_0(X)$-map $U \to S$ with $U \in X_\proet$ and $S \in \pi_0(X)_\proet$ factors canonically through $\pi_0(U)$, the sheaf $\pi^* F$ is the sheafification of the presheaf $U \mapsto F(\pi_0(U))$ on $U \in X_\proet$. As $F$ is itself a sheaf on $\pi_0(X)_\proet$, it is enough to check: for a surjection $U \to V$ in $X_\proet$, the map $\pi_0(U) \to \pi_0(V)$ is the coequalizer of the two maps $\pi_0(U \times_V U) \to \pi_0(U)$ in the category of profinite sets (induced by the two projection maps $U \times_V U \to U$). For any profinite set $S$, one has $(S \otimes X)(U) = \Map_{\cont}(\pi_0(U),S)$ with notation as in Example \ref{ex:tensorprofiniteset}, so the claim follows from the representability of $S \otimes X$ and fpqc descent. For (3), it suffices to check that $\pi_* \pi^* F \simeq F$ for any $F \in \Shv(\pi_0(X)_\proet)$, which is immediate from Lemma \ref{lem:findprofinitecovers} and (2). For (4), by (2), it remains to check that any $G$ with the property of (4) satisfies $G \simeq \pi^* \pi_* G$. Given $U \in X_\proet$, we have a canonical factorization $U \to \pi^{-1}(\pi_0(U)) \to X$, where $\pi^{-1}(\pi_0(U)) \to X$ is a pro-(finite \'etale) map inducing $\pi_0(U) \to \pi_0(X)$ on connected components, while $U \to \pi^{-1}(\pi_0(U))$ is an isomorphism on $\pi_0$. Then $G(U) = G(\pi^{-1}(\pi_0(U)))$ by assumption on $G$, which proves $G = \pi^* \pi_* G$ by (2).
\end{proof}

\begin{remark}
	The conclusion of Lemma \ref{lem:pullbackfrompi0} fails for $\pi:\Shv(X_\et) \to \Shv(\pi_0(X)_\et)$. Indeed, if $X$ is connected, then $\Shv(\pi_0(X)_\et) = \Set$, and $\pi^*$ coincides with the ``constant sheaf'' functor, which is not always limit-preserving.
\end{remark}

\subsection{The case of a point}
\label{ss:proetpoint}

Fix a profinite group $G$. We indicate how the definition of the pro-\'etale site can be adapted to give a site $BG_\proet$ of profinite $G$-sets. In particular, each topological $G$-module $M$ defines a sheaf $\calF_M$ on $BG_\proet$, and the resulting functor from topological $G$-modules to abelian sheaves on $BG_\proet$ is an embedding with dense image (in the sense of colimits). We use this construction to study the cohomology theory $M \mapsto \R\Gamma(BG_\proet,\calF_M)$ on $G$-modules: this theory is equal to continuous cohomology in many cases of interest, and yet better behaved in some functorial respects. The definition is:

\begin{definition}
	Let $BG_\proet$ be the {\em pro-\'etale site of $G$}, defined as the site of profinite continuous $G$-sets with covers given by continuous surjections. 
\end{definition}

For $S \in BG_\proet$, we use $h_S \in \Shv(BG_\proet)$ to denote the associated sheaf. Let $G\textrm{-}\Spc$ be the category of topological spaces with a continuous $G$-action; recall that $G\textrm{-}\Spc$ admits limits and colimits, and the formation of these commutes with passage to the underlying spaces (and thus the underlying sets). Let $G\textrm{-}\Spc_{cg} \subset G\textrm{-}\Spc$ be the full subcategory of $X \in G\textrm{-}\Spc$ whose underlying space may be written as a quotient of a disjoint union of compact Hausdorff spaces; we call these spaces compactly generated. There is a tight connection between these categories and $\Shv(BG_\proet)$:

\begin{lemma} 
	\label{lem:GSpcBG}
	Let notation be as above.
	\begin{enumerate}
		\item The association $X \mapsto \Map_{\cont,G}(-,X)$ gives a functor $\calF_{(-)}:G\textrm{-}\Spc \to \Shv(BG_\proet)$.
		\item The functor $\calF_{(-)}$ is limit-preserving and faithful.
		\item $\calF_{(-)}$ admits left adjoint $L$.
		\item $\calF_{(-)}$ is fully faithful on $G\textrm{-}\Spc_{cg}$.
		\item The essential image of $G\textrm{-}\Spc_{cg}$ generates $\Shv(BG_\proet)$ under colimits.
	\end{enumerate}
\end{lemma}
\begin{proof}
	The argument of Lemma \ref{l:ExSheafTopSpace} shows that any continuous surjection of profinite sets is a quotient map, which gives the sheaf property required in (1). It is clear that the resulting functor $\calF_{(-)}$ is limit-preserving. For any $X \in G\textrm{-}\Spc$, one has $\calF_X(G) = X$ where $G \in BG_\proet$ is the group itself, viewed as a left $G$-set via translation; this immediately gives (2). The adjoint functor theorem gives the existence of $L$ as in (3), but one can also construct it explicitly: the functor $h_S \mapsto S$ extends to a unique colimit preserving functor $\Shv(BG_\proet) \to G\textrm{-}\Spc$ by the universal property of the presheaf category (as a free cocompletion of $BG_\proet$) and the fact that covers in $BG_\proet$ give quotient maps. In particular, if $F \in \Shv(BG_\proet)$, then $F = \colim_{I_F} h_S$, where $I_F$ is the category of pairs $(S,s)$ with $S \in BG_\proet$ and $s \in F(S)$,  which gives $L(F) = \colim_{I_F} S$.  For (4), it is enough to show that $L(\calF_X) \simeq X$ for any compactly generated $X$. By the previous construction, one has $L(\calF_X) = \colim_{I_{\calF_X}} S$, so we must check that there exists a set $I$ of spaces $S_i \in BG_\proet$ and $G$-maps $s_i:S_i \to X$ such that $\sqcup_i S_i \to X$ is a quotient map. Choose a set $I$ of compact Hausdorff spaces $T_i$ and a quotient map $\sqcup_i T_i \to X$. Then the map $\sqcup_i T_i \times G \to X$ induced by the $G$-action is also a quotient, so we reduce to the case where $X$ is a compact Hausdorff $G$-space. Now consider $Y := G \times \beta(X) \in BG_\proet$, where the $G$-action is defined via $g \cdot (h,\eta) = (gh,\eta)$. There is an induced continuous map $f:Y \to X$ via $G \times \beta(X) \to G \times X \to X$, where the last map is the action. One checks that $f$ is $G$-equivariant and surjective. As $Y$ is profinite, this proves (4). Lastly, (5) is formal as $\calF_S = h_S$ for $S \in BG_\proet$.
\end{proof}

Let $G\textrm{-}\Mod$ denote the category of continuous $G$-modules, i.e., topological abelian groups equipped with a continuous $G$-action, and let $G\textrm{-}\Mod_{cg} \subset G\textrm{-}\Mod$ be the full subcategory of topological $G$-modules whose underlying space is compactly generated. The functor $\calF_{(-)}$ restricts to a functor $\calF_{(-)}:G\textrm{-}\Mod \to \Ab(BG_\proet)$, and  Lemma \ref{lem:GSpcBG} (1) - (4)  apply formally to this functor as well. The main non-formal statement is:

\begin{proposition}
	\label{prop:Gmodgenerates} 
	With notation as above, one has:
	\begin{enumerate}
		\item The essential image of $\calF_{(-)}:G\textrm{-}\Mod_{cg} \to \Ab(BG_\proet)$ generates the target under colimits. 
		\item Every $N \in \Ab(BG_\proet)$ has a resolution whose terms come from $G\textrm{-}\Mod_{cg}$.
	\end{enumerate}
\end{proposition}

To prove Proposition \ref{prop:Gmodgenerates}, we review some topological group theory. For a topological space $X$, write $AX$ for the free topological abelian group on $X$, defined by the obvious universal property. One may show that $AX$ is abstractly isomorphic to the free abelian group on the set $X$, see \cite[Theorem 7.1.7]{TopGroups}. In particular, one has a {\em reduced length} associated to each $f \in AX$, defined as the sum of the absolute values of the coefficients. Let $A_{\leq N}X \subset AX$ be the subset of words of length $\leq N$; one checks that this is a closed subspace, see \cite[Theorem 7.1.13]{TopGroups}. Moreover:

\begin{theorem}[Graev]
	\label{thm:graevtopgrp}
	If $X$ is a compact topological space, then $AX = \colim A_{\leq N} X$ as spaces.
\end{theorem}
\begin{proof}
	See Theorem \cite[Theorem 7.4.1]{TopGroups}.
\end{proof}

We use this to prove.

\begin{lemma}
	\label{lem:FunPointsFreeAbTop}
Fix a compact Hausdorff space $S$, an extremally disconnected profinite set $T$, and a continuous map $f:T \to AS$. Then there exists a clopen decomposition $T = \sqcup_i T_i$ such that $f|_{T_i}$ is a $\Z$-linear combination of continuous maps $T_i \to S$.
\end{lemma}
\begin{proof}
	Lemma \ref{lem:compactcountabletower} and Theorem \ref{thm:graevtopgrp} imply that $f$ factors through some $A_{\leq N} S$. Now consider the profinite set $\widetilde{S} = S \sqcup \{0\} \sqcup S$ and the induced map $\phi:\widetilde{S}^N \to A_{\leq N}$ defined by viewing $\widetilde{S}$ as the subspace $\Big(1\cdot S\Big) \sqcup \{0\} \sqcup \Big(-1 \cdot S\Big) \subset AS$ and using the group law. This map is continuous and surjective, and all spaces in sight are compact Hausdorff. By extremal disconnectedness, there is a lift $T \to \widetilde{S}^N$; one checks that this implies the desired claim.
\end{proof}

We can now identify the free abelian sheaf $\Z_{h_S}$ for any $S \in BG_\proet$:

\begin{lemma}
	\label{lem:GModfree}
	If $S \in BG_\proet$, then $\Z_{h_S} \simeq \calF_{AS}$.
\end{lemma}
\begin{proof}
	One clearly has $\calF_S = h_S$, so there is a natural map $\psi:\Z_{h_S} \to \calF_{AS}$ of abelian sheaves induced by $\calF_S \to \calF_{AS}$. We will check $\psi(T)$ is an isomorphism for $T$ covering $BG_\proet$. Let $F:\ast_\proet \to BG_\proet$ be a left adjoint to the forgetful functor $BG_\proet \to \ast_\proet$. Then it is enough to check $\psi(F(T))$ is an isomorphism for $T$ extremally disconnected. Unwinding definitions, this is exactly Lemma \ref{lem:FunPointsFreeAbTop}.
\end{proof}

Proposition \ref{prop:Gmodgenerates} falls out quickly:

\begin{proof}[Proof of Proposition \ref{prop:Gmodgenerates}]
Theorem \ref{thm:graevtopgrp} shows that $AS$ is compactly generated for any $S \in BG_\proet$. Now Lemma \ref{lem:GModfree} gives (1) as the collection $\{\Z_{h_S}\}$ generates $\Ab(BG_\proet)$ under colimits. Finally, (2) is formal from (1).
\end{proof}

The next lemma was used above, and will be useful later.

\begin{lemma}
	\label{lem:compactcountabletower}
	Fix a countable tower $X_1 \subset X_2 \subset \dots \subset X_n \subset \dots$ of closed immersions of Hausdorff topological spaces, and let $X = \colim_i X_i$. Then $\Map_{\cont}(S,X) = \colim \Map_{\cont}(S,X_i)$.
\end{lemma}
\begin{proof}
We must show each $f:S \to X$ factors through some $X_i$. Towards contradiction, assume there exists a map $f:S \to X$ with  $f(S) \not\subset X_i$ for all $i$. After reindexing, we may assume that there exist $x_i \in S$ such that $f(x_i) \in X_i - X_{i-1}$. These points give a map $\pi:\beta \N \to S$ via $i \mapsto x_i$. After replacing $f$ with $f \circ \pi$, we may assume $S = \beta \N$; set $T = \{f(i) | i \in \N\}$. Now pick any $x \in X - T$. Then $x \in X_j$ for some $j$. For $i > j$, we may inductively construct open neighourhoods $x \in U_i \subset X_i$ such that $U_i \cap T = \emptyset$, and $U_{i+1} \cap X_i = U_i$; here we use that $X_i \cap T$ is finite. The union $U = \cup_i U_i \subset X$ is an open neighbourhood of $x \in X$ that misses $T$. Hence, $f^{-1}(U) \cap \N = \emptyset$, so $f^{-1}(U) = \emptyset$ by density of $\N \subset S$. Varying over all $x \in X - T$ then shows that $f(S) = T$. Now one checks that $T \subset X$ is discrete: any open neighbourhood $1 \in U_1 \subset X_1$ can be inductively extended to open neighbourhoods $x_1 \in U_i \subset X_i$ such that $U_{i+1} \cap X_i = U_i$ and $x_i \notin U_i$.  Then $T$ must be finite as $S$ is compact, which is a contradiction.
\end{proof}

We now study the cohomology theory $M \mapsto \R\Gamma(BG_\proet, \calF_M)$ on $G\textrm{-}\Mod$. There is a natural transformation connecting it to continuous cohomology: 

\begin{lemma}
	\label{lem:NaturalMapContCoh}
	For any $M \in G\textrm{-}\Mod$, there is a natural map $\Phi_M:\R\Gamma_{\cont}(G,M) \to \R\Gamma(BG_\proet,\calF_M)$.
\end{lemma}
\begin{proof}
By \cite[Proposition 3.7]{ScholzepHT}, one has $\R\Gamma_{\cont}(G,M) = \R\Gamma(BG'_\proet,\mu_* \calF_M )$, where $BG'_\proet$ is defined as in Remark \ref{rmk:oldproet}, and $\mu:\Shv(BG_\proet) \to \Shv(BG'_\proet)$ the natural map; one then defines $\phi_M$ via pullback as $\mu^* \mu_* \simeq \id$ on $D(BG_\proet)$ (simply because $BG_\proet$ is finer topology than $BG'_\proet$ on the same category).
\end{proof}

The map $\Phi_M$ is an isomorphism for a fairly large collection of modules:

\begin{lemma}\label{lem:ContCohProetCoh}
	Let $\calC \subset G\textrm{-}\Mod$ be the full subcategory of all $M\in G\textrm{-}\Mod$ for which $R^i\mu_* \mathcal{F}_M = 0$ for all $i>0$, where $\mu: \Shv(BG_\proet) \to \Shv(BG'_\proet)$ is the natural map.
	\begin{enumerate}
		\item For all $M\in \calC$, the map $\Phi_M:\R\Gamma_{\cont}(G,M) \to \R\Gamma(BG_\proet,\calF_M)$ is an isomorphism.
		\item If $M\in G\textrm{-}\Mod$ is discrete, then $M\in \calC$.
		\item If $M=\colim M_i$ is a sequential colimit of Hausdorff $M_i\in \calC$ along closed immersions, then $M\in \calC$.
		\item If $M=\lim M_i$ is a sequential limit of $M_i\in \calC$ along profinitely split $M_{i+1}\to M_i$, then $M\in \calC$.
		\item If $M=\lim M_i$ is a sequential limit of $M_i\in \calC$ along $\beta$-epimorphisms $M_{i+1}\to M_i$ with kernel $K_i=\ker(M_{i+1}\to M_i)\in \calC$, then $M\in \calC$.
	\end{enumerate}
\end{lemma}

Here a quotient map $M \to N$ of topological spaces is said to be {\em profinitely split} if it admits sections over any map $K \to N$ with $K$ profinite. It is said to be a {\em $\beta$-epimorphism} if for every map $g:K \to N$ with $K$ compact Hausdorff, there is a surjection $K' \to K$ with $K'$ compact Hausdorff, and a lift $K' \to M$; equivalently, for any map $\beta(X) \to N$ where $X$ is discrete, there is a lift $\beta(X) \to M$. This property is automatic if $M \to N$ is a quotient map, and the kernel is compact Hausdorff.

\begin{proof}
	Parts (1) and (2) are clear. For (3), note that $\calF_M = \colim \calF_{M_i}$ by Lemma \ref{lem:compactcountabletower}, so the result follows as $\R\mu_\ast$ commutes with filtered colimits. For parts (4) and (5), note that if $M_{i+1}\to M_i$ is a $\beta$-epimorphism, then $\calF_{M_{i+1}}\to \calF_{M_i}$ is surjective on $BG_\proet$. By repleteness, we get $\calF_M = \lim \calF_{M_i} = \R\lim \calF_{M_i}$. Applying $\R\mu_\ast$ and using repleteness of $BG'_\proet$, we have to show that $\R^1\lim (\mu_* \calF_{M_i}) = 0$. If all $M_{i+1}\to M_i$ are profinitely split, then all $\mu_* \calF_{M_{i+1}}\to \mu_* \calF_{M_i}$ are surjective, so the result follows from repleteness of $BG'_\proet$. If $K_i = \ker(M_{i+1}\to M_i)\in \calC$, then on applying $\R\mu_*$ to the sequence
	\[0\to \calF_{K_i}\to \calF_{M_{i+1}}\to \calF_{M_i}\to 0,\]
	we find that $\mu_* \calF_{M_{i+1}}\to \mu_* \calF_{M_i}$ is surjective, so again the result follows from repleteness of $BG'_\proet$.
\end{proof}

\begin{remark}
The category $\calC$ of Lemma \ref{lem:ContCohProetCoh} includes many standard Galois modules occurring in arithmetic geometry obtained by iterations of completions and localisations applied to discrete modules. For example, when $G = \Gal(\overline{\Q}_p/\Q_p)$, the $G$-module $B_\dR$ is such an object.
\end{remark}

We now indicate one respect in which $\R\Gamma(BG_\proet,\calF_{(-)})$ behaves better than continuous cohomology: one gets long exact sequences in cohomology with fewer constraints.

\begin{lemma} 
	\label{lem:ProetCohLES}
	Fix an algebraically short exact sequence $0 \to M' \stackrel{a}{\to} M \stackrel{b}{\to} M'' \to 0$ in $G\textrm{-}\Mod$. Assume $b$ is a $\beta$-epimorphism, and $a$ realises $M'$ as a subspace of $M$.  Then there is an induced long exact sequence on applying $H^*(BG_\proet, \calF_{(-)})$.
\end{lemma}

\begin{proof}
It is enough to show that 
\[ 0 \to \calF_{M'} \to \calF_M \to \calF_{M''} \to 0\]
is exact. Exactness on the right results from the assumption on $b$, exactness on the left is obvious from the injectivity of $M' \hookrightarrow M$, and exactness in the middle comes from the assumption on $a$.
\end{proof}

\begin{remark}
Considerations of the discrete topology show that {\em some} hypothesis must be imposed in Lemma \ref{lem:ProetCohLES}. The assumption used above is fairly weak: it is automatic if $M'$ is compact Hausdorff.  In contrast, in continuous cohomology, one demands existence of sections after base change to {\em all} profinite sets over $M''$.
\end{remark}

\newpage

\newpage

\section{Relations with the \'etale topology}
\label{sec:Etale}

Fix a scheme $X$. Since an \'etale map is also a weakly \'etale map, we obtain a morphism of topoi
\[ \nu:\Shv(X_\proet) \to \Shv(X_\et).\]
The main goal of this section is to describe its behaviour at the level of derived categories.  The pullback and pushforward along $\nu$, together with the resulting semiorthogonal decompositions of complexes on $X_\proet$, are discussed in \S \ref{ss:nupullback} and \S \ref{ss:nupushforward}. This is used to describe the left-completion of $D(X_\et)$ in terms of $D(X_\proet)$ in \S \ref{ss:leftcompetale}. Some elementary remarks on the functoriality of $\nu$ in $X$ are recorded in \S \ref{ss:nufunctoriality}. Finally, we describe Ekedahl's category of ``adic'' complexes \cite{Ekedahl} in terms of $D(X_\proet)$ in \S \ref{subsec:ekedahl}. We rigorously adhere to the derived convention: the functors $\nu^*$ and $\nu_*$, when applied to complexes, are understood to be derived.

\subsection{The pullback}
\label{ss:nupullback}

We begin with the pullback at the level of sheaves of sets:

\begin{lemma}\label{lem:sheafpullback} For $F \in \Shv(X_\et)$ and $U \in X^\aff_\proet$ with a presentation $U = \lim_i U_i$,  one has $\nu^*F(U) = \colim_i F(U_i)$.
\end{lemma}
\begin{proof} The problem is local on $X$, so we may assume that $X = \Spec(A)$ is affine. In that case, by Remark \ref{rmk:affineproet}, the site $X_\proet$ is equivalent to the site $S$ given by ind-\'etale $A$-algebras $B=\colim B_i$, with covers given by faithfully flat maps. The pullback $F'$ of $F$ to $S$ as a presheaf is given by $F'(B) = \colim F(B_i)$. It thus suffices to check that $F'$ is a sheaf; we will do this using Lemma \ref{lem:sheafcriterion}. First, note that $F'$ is a Zariski sheaf since any finite collection of quasicompact open subschemes of $\Spec B$ come via pullback from some $\Spec B_i$. It remains to show that $F'$ satisfies the sheaf axiom for every faithfully flat ind-\'etale map $B\to C$ of ind-\'etale $A$-algebras. If $B \to C$ is actually \'etale, then it arises via base change from some faithfully flat \'etale map $B_i \to C_i$, so the claim follows as $F$ is a sheaf. In general, write $C=\colim C_j$ as a filtered colimit of \'etale $B$-algebras $C_j$, necessarily faithfully flat. Then $F'(C) = \colim_j F'(C_j)$. The sheaf axiom for $B \to C$ now follows by taking filtered colimits.
\end{proof}

A first consequence of the above formula is that $\nu^*$ is fully faithful. In fact, we have:

\begin{lemma}\label{lem:nupullbackdesc}	The pullback $\nu^*:\Shv(X_\et) \to \Shv(X_\proet)$ is fully faithful. Its essential image consists exactly of those sheaves $F$ with $F(U) = \colim_i F(U_i)$ for any $U \in X^\aff_\proet$ with presentation $U = \lim_i U_i$.
\end{lemma}

\begin{proof}	Lemma \ref{lem:sheafpullback} shows that $F \simeq \nu_* \nu^* F$ for any $F \in \Shv(X_\et)$, which formally implies that $\nu^*$ is fully faithful. For the second part, fix some $G \in \Shv(X_\proet)$ satisfying the condition of the lemma. Then Lemma \ref{lem:sheafpullback} (together with Lemma \ref{lem:baseproet}) shows that $\nu^* \nu_* G \to G$ is an isomorphism, which proves the claim.
\end{proof}

\begin{definition} A sheaf $F\in \Shv(X_\proet)$ is called {\em classical} if it lies in the essential image of $\nu^\ast: \Shv(X_\et)\to \Shv(X_\proet)$.
\end{definition}

In particular, $F$ is classical if and only if $\nu^\ast \nu_\ast F\to F$ is an isomorphism. We need a simple lemma on recognizing classical sheaves.

\begin{lemma}\label{l:Classical} Let $F$ be a sheaf on $X_\proet$. Assume that for some pro-\'etale cover $\{Y_i\to X\}$, $F|_{Y_i}$ is classical. Then $F$ is classical.
\end{lemma}

\begin{proof} We may assume that $X=\Spec A$ is affine, that there is only one $Y=Y_i=\Spec B$, with $A\to B$ ind-\'etale, $B=\colim_i B_i$, with $A\to B_i$ \'etale. We need to check that for any ind-\'etale $A$-algebra $C=\colim_j C_j$, we have $F(C) = \colim_j F(C_j)$. Now consider the following diagram, expressing the sheaf property for $C\to B\otimes C$, resp. $C_j\to B\otimes C_j$.
\[\xymatrix{
F(C)\ar[r]\ar[d] & F(C\otimes B)\ar@<-2pt>[r]\ar@<2pt>[r]\ar[d] & F(C\otimes B\otimes B)\ar[d] \\
\colim F(C_j)\ar[r] & \colim_j F(C_j\otimes B) \ar@<-2pt>[r]\ar@<2pt>[r] & \colim_j F(C_j\otimes B\otimes B)
}\]
The second and third vertical arrows are isomorphisms as $F|_{\Spec B}$ is classical. Thus, the first vertical arrow is an isomorphism as well, as desired.
\end{proof}

As an example, let us show how this implies that the category of local systems does not change under passage from $X_\et$ to $X_\proet$.

\begin{corollary}\label{c:LocSysDiscrete} Let $R$ be a discrete ring. Let $\Loc_{X_\et}(R)$ be the category of $R$-modules $L_\et$ on $X_\et$ which are locally free of finite rank. Similarly, let $\Loc_{X_\proet}(R)$ be the category of $R$-modules $L_\proet$ on $X_\proet$ which are locally free of finite rank. Then $\nu^\ast$ defines an equivalence of categories $\Loc_{X_\et}(R)\cong \Loc_{X_\proet}(R)$.
\end{corollary}

In the following, we denote either category by $\Loc_X(R)$.

\begin{proof} If $L_\et\in \Loc_{X_\et}(R)$, then clearly $L_\proet = \nu^\ast L_\et\in \Loc_{X_\proet}(R)$; as $\nu^\ast$ is fully faithful, it remains to verify essential surjectivity. Thus, take $L_\proet\in \Loc_{X_\proet}(R)$. As $L_\proet$ is locally free of finite rank, it is in particular locally classical, thus classical by Lemma \ref{l:Classical}. Thus, $L_\proet = \nu^\ast L_\et$ for some sheaf $L_\et$ of $R$-modules on $X_\et$. Assume that $U\in X_\proet^\aff$ with presentation $U=\lim U_i$ is such that $L_\proet|_U\cong R^n|_U$. The isomorphism is given by $n$ elements of $(L_\proet)(U) = \colim_i L_\et(U_i)$. This shows that the isomorphism $L_\proet|_U\cong R^n|_U$ is already defined over some $U_i$, thus $L_\et\in \Loc_{X_\et}(R)$, as desired.
\end{proof}

Next, we pass to derived categories.

\begin{corollary}
	\label{cor:pushpull}
	For any $K \in D^+(X_\et)$, the adjunction map $K \to \nu_* \nu^* K$ is an equivalence. Moreover, if $U \in X^\aff_\proet$ with presentation $U = \lim_i U_i$, then $\R\Gamma(U,\nu^*K) = \colim_i \R\Gamma(U_i,K)$. 
\end{corollary}
\begin{proof}
	The first part follows from the second part by checking it on sections using Lemma \ref{lem:baseproet}, i.e., by applying $\R\Gamma(V,-)$ to the map $K \to \nu_* \nu^* K$ for each affine $V \in X_\et$. For the second part, the collection of all $K \in D^+(X_\et)$ for which the claim is true forms a triangulated category stable under filtered colimits. Hence, it suffices to prove the claim for $K \in \Ab(X_\et) \subset D^+(X_\et)$. For such $K$, since we already know the result on $H^0$ by Lemma \ref{lem:sheafpullback}, it suffices to prove: $H^p(U,\nu^* I) = 0$ for $I \in \Ab(X_\et)$ injective, $p > 0$, and $U \in X^\aff_\proet$. By \cite[Proposition V.4.3]{SGA4Tome2},  it suffices to prove that $\check{H}^p(U,\nu^* I) = 0$ for the same data. Choose a presentation $U = \lim_i U_i$ for some cofiltered category $I$. By Theorem \ref{t:WeaklyVsIndEtale}, a cofinal collection of covers of $U$ in $X_\proet$ is obtained by taking cofiltered limits of affine \'etale covers obtained via base change from some $U_i$.  Using Lemma \ref{lem:sheafpullback} again, we can write
	\[ \check{H}^p(U,F) = \colim H^p\Big(\cosimp{I(V)}{I(V \times_{U_i} V)}{I(V \times_{U_i} V \times_{U_i} V)}\Big)\]
	where the colimit is computed over (the opposite of) the category of pairs $(i,V)$ where $i \in I$, and $V \to U_i$ is an affine \'etale cover.  For a fixed $i$, the corresponding colimit has vanishing higher cohomology since $I|_{U_i}$ is injective in $\Ab(U_{i,\et})$, and hence has trivial higher Cech cohomology. The claim follows as filtered colimits are exact.
\end{proof}

Again, we will refer to objects in the essential image of $\nu^\ast$ as classical, and Lemma \ref{l:Classical} extends to bounded-below derived categories with the same proof.

\begin{remark}
	The argument used to prove Corollary \ref{cor:pushpull} also shows: if $U \in X^\aff_\proet$ is w-strictly local, then $H^p(U,\nu^* F) = 0$ for all $F \in \Ab(X_\et)$ and $p > 0$. Indeed, for such $U$, any affine \'etale cover $V \to U$ has a section, so the corresponding Cech nerve is homotopy-equivalent to $U$ as a simplicial scheme.
\end{remark}

\begin{remark}
	\label{rmk:unboundedpullback}
If $K \in D(X_\et)$ is an unbounded complex, then the formula in Corollary \ref{cor:pushpull} is not true. Instead, to describe $\nu^* K$, first observe that $\nu^* K \simeq \R\lim \nu^* \tau^{\geq -n} K$ as $\Shv(X_\proet)$ is replete and $\nu^*$ commutes with Postnikov truncations. Hence, $\R\Gamma(Y,\nu^* K) \simeq \R\lim \colim_i \R\Gamma(Y_i,\tau^{\geq -n} K)$ for any $Y \in X^\aff_\proet$ with a presentation  $Y = \lim Y_i$. Moreover, since $\nu_*$ commutes with arbitrary limits, we also see that $\nu_* \nu^* K \simeq \R\lim \tau^{\geq -n} K$. For an explicit example, we remark that Example \ref{ex:notleftcomplete} can be adapted to exhibit the failure of $\id \to \nu_* \nu^*$ being an equivalence.
\end{remark}

An abelian consequence is:

\begin{corollary}
\label{cor:absheafpullbackserresubcat}
The pullback $\nu^*:\Ab(X_\et) \to \Ab(X_\proet)$ induces an equivalence on $\Ext^i$ for all $i$. In particular, $\nu^*(\Ab(X_\et)) \subset \Ab(X_\proet)$ is a Serre subcategory.
\end{corollary}
\begin{proof}
	Let $\calC \subset \Ab(X_\et)$ be the full subcategory of sheaves $F$ for which $\Ext^i(F,-) \simeq \Ext^i(\nu^*(F),\nu^*(-))$ for all $i$. Then $\calC$ contains all direct sums of sheaves of the form  $\Z_U$ for $U \in X_\et$ by Corollary \ref{cor:pushpull}. Since any $F \in \Ab(X_\et)$ admits a surjection from such a direct sum, the claim follows by dimension shifting.
\end{proof}

\subsection{The pushforward}
\label{ss:nupushforward}

Our goal is to describe the pushforward $\nu_*:D(X_\proet) \to D(X_\et)$ and the resulting decomposition of $D(X_\proet)$. To do so, it is convenient to isolate the kernel of $\nu_*$:

\begin{definition}
A complex $L \in D(X_\proet)$ is {\em parasitic} if $\R\Gamma(\nu^{-1}U,L) = 0$ for any $U \in X_\et$. Write $D_p(X_\proet) \subset D(X_\proet)$ for the full subcategory of parasitic complexes, $D_p^+(X_\proet)$ for bounded below parasitics, etc.
\end{definition}

The key example is:

\begin{example}
	\label{ex:parasiticlimits}
	Let $\{F_n\} \in \Fun(\N^\opp,\Ab(X_\et))$ be a projective system of sheaves with surjective transition maps. Set $K = \R\lim F_n \in D(X_\et)$, and $K' = \R\lim \nu^*(F_n) \in D(X_\proet)$. Then $K' \simeq \lim \nu^*(F_n)$ as $X_\proet$ is replete. The natural map $\nu^* K \to K'$ has a parasitic cone since $\nu_* \nu^* K \simeq K = \R\lim F_n \simeq \R\lim \nu_* \nu^* F_n \simeq \nu_* K'$. For example, when $X = \Spec(\Q)$, the cone of the map $\nu^* (\R\lim \mu_n) \to \lim \mu_n$ is non-zero and parasitic.
\end{example}

The basic structural properties of $D_p(X_\proet)$ are:

\begin{lemma}
	\label{lem:parasitic}
	The following are true:
	\begin{enumerate}
	\item $D_p(X_\proet)$ is the kernel of $\nu_*:D(X_\proet) \to D(X_\et)$.
	\item $D_p(X_\proet)$ is a thick triangulated subcategory of $D(X_\proet)$. 
	\item The inclusion $i:D_p(X_\proet) \to D(X_\proet)$ has a left adjoint $L$.
	\item The adjunction $ L \circ i \to \id $ is an equivalence.
\end{enumerate}
\end{lemma}
\begin{proof} Sketches:
\begin{enumerate}
	\item This follows from the adjunction between $\nu^*$ and $\nu_*$ together with the fact that $D(X_\et)$ is generated under homotopy-colimits by sheaves of the form $\Z_U$ for $U \in X_\et$. 
\item Clear.
\item Consider the functor $M:D(X_\proet) \to D(X_\proet)$ defined via $M(K) = \cok(\nu^* \nu_* K \to K)$. There is a map $\id \to M$, and hence a tower $\id \to M \to M^2 \to M^3 \to \dots$ , where $M^n$ is the $n$-fold composition of $M$ with itself. We set $L:D(X_\proet) \to D(X_\proet)$ to be the (filtered) colimit of this tower, i.e., $L(K) = M^\infty(K) := \colim_n M^n(K)$.  We will show that $L(K)$ is parasitic for any $K$, and that the induced functor $L:D(X_\proet) \to D_p(X_\proet)$ is a left adjoint to $i$. Choose any $U \in X_\et$. As $U$ is qcqs, we have
	\[ \R\Gamma(\nu^{-1}U,L(K)) \simeq \R\Gamma(\nu^{-1}U,\colim_n M^n(K)) = \colim_n \R\Gamma(\nu^{-1}U,M^n(K)).\]
	Hence, to show that $L$ takes on parasitic values, it suffices to show that
	\[\R\Gamma(\nu^{-1}U,K) \to \R\Gamma(\nu^{-1}U,M(K)) \]
	is the $0$ map for any $K \in D(X_\proet)$.  Since $\nu$ is a map of a topoi, we have a factorisation
	\[\R\Gamma(\nu^{-1}U,K) \simeq  \R\Gamma(U,\nu_*K ) \stackrel{\nu^{-1}}{\to} \R\Gamma(\nu^{-1}U,\nu^* \nu_* K) \to \R\Gamma(\nu^{-1}U,K)\]
	of the identity map on $\R\Gamma(\nu^{-1}U,K)$. The composition $\R\Gamma(\nu^{-1}U,K) \to \R\Gamma(\nu^{-1}U,M(K))$ is then $0$ by definition of $M(K)$, which proves that $L(K)$ is parasitic. To show that the induced functor $L:D(X_\proet) \to D_p(X_\proet)$ is a left adjoint to the inclusion, note first that for any $K,P \in D(X_\proet)$ with $P$ parasitic, one has $\Hom(\nu^* \nu_* K,P) = \Hom(\nu_*K,\nu_*P) = 0$ by (1). The exact triangle defining $M(K)$ shows
	\[ \Hom(K,P) \simeq \Hom(M(K),P) \simeq \Hom(M^2(K),P) \simeq \dots \simeq  \Hom(M^n(K),P) \]
	for any $n \geq 0$. Taking limits then shows
	\[ \Hom(K,P) = \lim \Hom(M^n(K),P) = \Hom(\colim_n M^n(K),P) = \Hom(L(K),P),\]
	which is the desired adjointness.
\item This follows from (1) and the construction of $L$ given in (3): for any parasitic $P \in  D(X_\proet)$, one has $P \simeq M(P) \simeq M^n(P) \simeq \colim_n M^n(P) \simeq L(P)$ since $\nu_* P = 0$. \qedhere
\end{enumerate}
\end{proof}

\begin{remark}
In Lemma \ref{lem:parasitic}, it is important to work at the derived level: the full subcategory $\Ab_p(X_\proet)$ of all $F \in \Ab(X_\proet)$ with $F(\nu^{-1}U) = 0$ for any $U \in X_\et$ is {\em not} a Serre subcategory of $\Ab(X_\proet)$. For example, let $X = \Spec(\Q)$ and set $\widehat{\Z}_\ell(1) := \lim \mu_{\ell^n} \in \Ab(X_\proet)$. Then there is an exact sequence 
\[ 1 \to \widehat{\Z}_\ell(1)  \stackrel{\ell}{\to} \widehat{\Z}_\ell(1)  \to \mu_\ell \to 1\]
in $\Ab(X_\proet)$. One easily checks that $\widehat{\Z}_\ell(1) \in \Ab_p(X_\proet)$, while $\mu_\ell \not\in \Ab_p(X_\proet)$.
\end{remark}

\begin{remark}
	\label{rmk:parasiticlocalisationbounded}
	The localisation functor $L:D(X_\proet) \to D_p(X_\proet)$ from Lemma \ref{lem:parasitic} admits a particularly simple description when restricted to bounded below complexes: $L(K) \simeq \cok(\nu^* \nu_* K \to K)$ for any $K \in D^+(X_\proet)$. Indeed, by the proof of Lemma \ref{lem:parasitic} (3), it suffices to show that $M(K) \simeq M^2(K)$ for such a complex $K$; this follows from the formula $\nu^* \nu_* \nu^* \nu_* K \simeq \nu^* \nu_* K$, which comes from Corollary \ref{cor:pushpull}.
\end{remark}

We can now show that $D^+(X_\et)$ and $D^+_p(X_\proet)$ give a semiorthogonal decomposition for $D^+(X_\proet)$.

\begin{proposition}
	\label{prop:parasiticsemiorth}
	Consider the adjoints  $\adjunction{D^+(X_\proet)}{\nu_*}{D^+(X_\et)}{\nu^*}$  and $\adjunction{D^+_p(X_\proet)}{i}{D^+(X_\proet)}{L}$.
	\begin{enumerate}
	\item $\nu^*$ is fully faithful.
	\item The adjunction $\id \to \nu_* \nu^*$ is an equivalence.
	\item The essential image of $\nu^*$ is exactly those $K \in D^+(X_\proet)$ whose cohomology sheaves are in $\nu^*(\Ab(X_\et))$.
	\item The pushforward $\nu_*$ realises $D^+(X_\et)$ as the Verdier quotient of $D^+(X_\proet)$ by $D^+_p(X_\proet)$.
	\item The map $L$ realises $D_p^+(X_\proet)$ as the Verdier quotient of $D^+(X_\proet)$ by $\nu^*(D^+(X_\et))$.
\end{enumerate}
\end{proposition}
\begin{proof} Sketches:
	\begin{enumerate}
		\item This follows formally from Corollary \ref{cor:pushpull}.
		\item This follows from (1) by Yoneda.
		\item Let $\calC \subset D^+(X_\proet)$ be the full subcategory of complexes whose cohomology sheaves lie in $\nu^*(\Ab(X_\et))$; by Corollary \ref{cor:absheafpullbackserresubcat}, this is a triangulated subcategory of $D^+(X_\proet)$ closed under filtered colimits. Moreover, chasing triangles and truncations characterises $\calC$ as the smallest triangulated subcategory of $D^+(X_\proet)$ closed under filtered colimits that contains $\nu^*(\Ab(X_\et))$. Now $\nu^*(D^+(X_\et)) \subset \calC$ as $\nu^*$ is exact. Moreover, by (1) and left-adjointness of $\nu^*$, we see that $\nu^*(D^+(X_\et))$ is a triangulated subcategory of $D^+(X_\proet)$ closed under filtered colimits. Since $\nu^*(D^+(X_\et))$ clearly contains $\nu^*(\Ab(X_\et))$, the claim follows.
		\item By Lemma \ref{lem:verdquot}, we want $\nu_*$ to admit a fully faithful left adjoint; this is what (1) says.
		\item This follows from Lemma \ref{lem:parasitic} and Lemma \ref{lem:verdquot} provided $\nu^*(D^+(X_\et))$ is the kernel of $L$. By Remark \ref{rmk:parasiticlocalisationbounded}, the kernel of $L$ is exactly those $K$ with $\nu^* \nu_* K \simeq K$, so the claim follows using Corollary \ref{cor:pushpull}. \qedhere
	\end{enumerate}
\end{proof}

The following observation was used above:

\begin{lemma}
	\label{lem:verdquot}
Let $L:\calC_1 \to \calC_2$ be a triangulated functor between triangulated categories. If $L$ admits a fully faithful left or right adjoint $i$, then $L$ is a Verdier quotient of $\calC_1$ by $\ker(L)$.
\end{lemma}
\begin{proof}
By symmetry, we may assume $L$ is a left adjoint. Given a triangulated functor $F:\calC_1 \to \calD$ which carries $\ker(L)$ to $0$, we will show that the natural map $F \to F \circ i \circ L$ is an equivalence.  First, adjunction shows $L \circ i \simeq \id$  via the natural map as $i$ is fully faithful. Hence, for each $K \in \calC_1$, we get a triangle $K' \to K \to (i \circ L)(K)$ such that $L(K') = 0$. This shows that $F(K) \simeq (F \circ i \circ L)(K)$ for all such $F$, proving the claim.
\end{proof}

\begin{remark}
\label{rmk:unboundedsemiorthfinitecd}
If we assume that $X_\et$ is locally of finite cohomological dimension, then $D(X_\et)$ is left-complete. Since $D(X_\proet)$ is also left-complete, one can show that $\nu^*:D(X_\et) \to D(X_\proet)$ is fully faithful by reduction to the bounded below case.  In fact, every statement in Proposition \ref{prop:parasiticsemiorth} extends to the unbounded setting in this case. 
\end{remark}

At the unbounded level, the pullback $\nu^*:D(X_\et) \to D(X_\proet)$ is {\em not} fully faithful in general, as explained in Remark \ref{rmk:unboundedpullback}, so none of the arguments in Proposition \ref{prop:parasiticsemiorth} apply. Nevertheless, we can still prove a semiorthogonal decomposition as in Proposition \ref{prop:parasiticsemiorth} at the expense of replacing $D(X_\et)$ with the smallest triangulated subcategory $D' \subset D(X_\proet)$ that contains $\nu^*(D(X_\et))$ and is closed under filtered colimits:

\begin{proposition}
	\label{prop:parasiticleftorthogonal}
	Let $D' \subset D(X_\proet)$ be as above. Then
\begin{enumerate}
\item If $\nu^*$ is fully faithful, then $\nu^*$ induces an equivalence $D(X_\et) \simeq D'$.
\item Given $K \in D(X_\proet)$, one has $K \in D'$ if and only if $\Hom(K,K') = 0$ for any $K' \in D_p(X_\proet)$.
\item The inclusion $i:D' \hookrightarrow D$ admits a right adjoint $N:D(X_\proet) \to D'$ such that $N \circ i \simeq \id$.
\item The localisation $L$ realises $D_p(X_\proet)$ as the Verdier quotient of $D(X_\proet)$ by $D'$.
\item The map $N$ realises $D'$ as the Verdier quotient of $D(X_\proet)$ by $D_p(X_\proet)$.
\end{enumerate}
\end{proposition}
\begin{proof} Sketches:
\begin{enumerate}
\item If $\nu^*$ is fully faithful, then $K \simeq \nu_* \nu^* K \simeq \R\lim \tau^{\geq -n} K$ (where the last isomorphism is from Remark \ref{rmk:unboundedpullback}). The claim now follows by reduction to the bounded case, as in Remark \ref{rmk:unboundedsemiorthfinitecd}.
\item Since $\nu^*(D(X_\et))$ is left-orthogonal to $D_p(X_\proet)$, so is $D'$. For the converse direction, consider the functors $N_i:D(X_\proet) \to D(X_\proet)$ defined via $N_i(K) = \ker(K \to M^i(K))$ where $M(K) = \cok(\nu^* \nu_* K \to K)$ (as in the proof of Lemma \ref{lem:parasitic}). The tower $\id \to M \to M^2 \to M^3 \to \dots$ gives rise to a tower $N_1 \to N_2 \to N_3 \to \dots \to \id$ with $N_{i+1}$ being an extension of $\nu^* \nu_* M^i$ by $N_i$; set $N = \colim_i N_i$. The description in terms of extensions shows $N_i(K) \in D'$ for all $i$, and hence $N \in D'$ as $D'$ is closed under filtered colimits. Moreover, setting $L = \colim_i M^i$ gives an exact triangle  $N \to \id \to L$ of functors. As in Lemma \ref{lem:parasitic}, $L$ realises the parasitic localisation $D(X_\proet) \to D_p(X_\proet)$. Hence, if $\Hom(K,K') = 0$ for every parasitic $K'$, then $K \simeq N(K) \in D'$ by the previous triangle.
\item One checks that the functor $N:D(X_\proet) \to D'$ constructed in (2) does the job (using the exact triangle $N \to \id \to L$ and the fact that $\Hom(D',L(K)) = 0$ for all $K$ by (2)).
\item This follows from Lemma \ref{lem:verdquot} if we could show that $D'$ is the kernel of $L$. For this, one simply uses the exact triangle $N \to \id \to L$ as in (2).
\item This is proven exactly like (4). \qedhere
\end{enumerate}
\end{proof}

\subsection{Realising the left-completion of $D(X_\et)$ via the pro-\'etale site}
\label{ss:leftcompetale}

Our goal is to identify the left-completion $\widehat{D}(X_\et)$ with a certain subcategory of $D(X_\proet)$ using the analysis of the previous sections.  The starting point is the following observation:  by Proposition \ref{prop:repletepostnikov}, the pullback $\nu^*:D(X_\et) \to D(X_\proet)$ factors through $\tau:D(X_\et) \to \widehat{D}(X_\et)$. To go further, we isolate a subcategory of $D(X_\proet)$ that contains the image of $\nu^*$:

\begin{definition}
	Let $D_{cc}(X_\proet)$ be the full subcategory of $D(X_\proet)$ spanned by complexes whose cohomology sheaves lie in $\nu^*(\Ab(X_\et))$; we write $D^+_{cc}(X_\proet)$ for the bounded below objects, etc.
\end{definition}

Since $\nu^*:D(X_\et) \to D(X_\proet)$ is exact, it factors through $D_{cc}(X_\proet)$, and hence we get a functor $\widehat{D}(X_\et) \to D_{cc}(X_\proet)$. Our main observation is that this functor is an equivalence. More precisely:

\begin{proposition}
	\label{prop:dccleftcompletion}
	There is an adjunction $\adjunction{D_{cc}(X_\proet)}{\nu_{cc,*}}{D(X_\et)}{\nu_{cc}^*}$ induced by $\nu_*$ and $\nu^*$. This adjunction is isomorphic to the left-completion adjunction $\adjunction{\widehat{D}(X_\et)}{\R\lim}{D(X_\et)}{\tau}$. In particular, $D_{cc}(X_\proet) \simeq \widehat{D}(X_\et)$.
\end{proposition}
\begin{proof}
	The existence of the adjunction is formal from the following: (a)  $\nu^*$ carries $D(X_\et)$ to $D_{cc}(X_\proet)$, and (b) $D_{cc}(X_\proet) \hookrightarrow D(X_\proet)$ is fully faithful. Proposition \ref{prop:parasiticsemiorth} immediately implies that $\nu_{cc}^*$ induces an equivalence $D^+(X_\et) \simeq D^+_{cc}(X_\proet)$. To extend to the unbounded setting, observe that $K \in D_{cc}(X_\proet)$ if and only if $\tau^{\geq -n} K \in D_{cc}(X_\proet)$ by the left-completeness of $D(X_\proet)$ and the exactness of $\nu^*$. This lets us define functors 
	$\mu:\widehat{D}(X_\et) \to D_{cc}(X_\proet)$ and $\gamma:D_{cc}(X_\proet) \to \widehat{D}(X_\et)$ via $\mu(\{K_n\}) = \R\lim \nu^*(K_n)$ and $\gamma(K) = \{\nu_* \tau^{\geq -n} K\}$; one can check that $\mu$ and $\gamma$ realise the desired mutually inverse equivalences.
\end{proof}

Since $D'$ is the smallest subcategory of $D(X_\proet)$ that contains $\nu^* D(X_\et)$ and is closed under filtered colimits, one has $D' \subset D_{cc}(X_\proet)$. It is natural to ask how close this functor is to being an equivalence. One can show that if $D(X_\et)$ is left-complete, then $D(X_\et) \simeq D' \simeq D_{cc}(X_\proet)$; we expect that $D^\prime\simeq D_{cc}(X_\proet)$ fails without left-completeness, but do not have an example.

\subsection{Functoriality}
\label{ss:nufunctoriality}

We study the variation of $\nu:\Shv(X_\proet) \to \Shv(X_\et)$ with $X$.  For notational clarity, we often write $\nu_X$ instead of $\nu$.

\begin{lemma}
	\label{lem:proetfunc}
	A morphism $f:X \to Y$ of schemes induces a map $f_\proet:\Shv(X_\proet) \to \Shv(Y_\proet)$ of topoi with $f^*$ given by pullback on representable objects. The induced diagram
	\[ \xymatrix{ \Shv(X_\proet) \ar[d]^{f_\proet} \ar[r]^{\nu_X} & \Shv(X_\et) \ar[d]^{f_\et} \\
	\Shv(Y_\proet) \ar[r]^{\nu_Y} & \Shv(Y_\et) }\]
	commutes. In particular, for $F$ either in $\Shv(Y_\et)$ or $D(Y_\et)$, there is an isomorphism $f_{\proet}^* \circ \nu_Y^* (F) \simeq \nu_X^* \circ f_\et^* (F)$.
\end{lemma}
\begin{proof}
Since all maps in sight are induced by morphisms of sites, this follows simply by the definition of pullback.
\end{proof}

\begin{lemma}\label{lem:topinv}
Let $f:X \to Y$ be a universal homeomorphism of schemes, i.e., $f$ is universally bijective and integral. Then $f_*:\Shv(X_\proet) \to \Shv(Y_\proet)$ is an equivalence.
\end{lemma}
\begin{proof}
	The claim is local on $Y$, so we may $Y$ and $X$ are affine. By Theorem \ref{t:WeaklyVsIndEtale}, we can identify $\Shv(Y_\proet)$ with the topos of sheaves on the site opposite to the category of ind-\'etale $\calO(Y)$-algebras with covers generated by faithfully flat maps and Zariski covers, and likewise for $X$. Since $f^{-1}$ identifies $X_\et$ with $Y_\et$ while preserving affine objects (by integrality) and covers, the claim follows from the topological invariance of the usual \'etale site.
\end{proof}

\begin{lemma}
	\label{lem:funcpushforward}
	Fix a qcqs map $f:Y \to X$ of schemes and $F$ either in $\Shv(Y_\et)$ or $D^+(Y_\et)$. Then the natural map
	\[ \nu_Y^* \circ f_{\et,*} (F) \to f_{\proet,*} \circ \nu_X^* (F) \]
	is an equivalence.
\end{lemma}
\begin{proof}
We first handle $F \in \Shv(Y_\et)$. The claim is local on $X$, so we may assume $X$ is affine. First, consider the case where $Y$ is also affine. Choose some $U \in Y^\aff_\proet$ with presentation $U = \lim_i U_i$. Then Lemma \ref{lem:sheafpullback} shows
\[ \nu_Y^* \circ f_{\et,*}(F) (U) = \colim_i F(f^{-1} U_i). \]
As $f^{-1} U \in Y^\aff_\proet$ with presentation $f^{-1} U = \lim_i f^{-1} U_i$, one concludes by reapplying Lemma \ref{lem:sheafpullback}. For not necessarily affine but separated and quasicompact $Y$, the same argument shows that the claim is true for all $F \in \Shv(Y_\et)$ obtained as pushforwards from an affine open of $Y$. Since the collection of all $F$ satisfying the above conclusion is stable under finite limits, a Mayer-Vietoris argument shows that the claim is true for all $F \in \Shv(Y)$ with $Y$ quasicompact and separated. Repeating the argument (and using the separated case) then gives the claim for all qcqs $Y$. For $F \in D^+(X_\et)$, the same argument applies using Corollary \ref{cor:pushpull} instead of Lemma \ref{lem:sheafpullback} (with finite limits replaced by finite homotopy-limits).
\end{proof}

\begin{remark}
Lemma \ref{lem:funcpushforward} does {\em not} apply to unbounded complexes.  Any scheme $X'$ with $D(X'_\et)$ not left-complete (see Remark \ref{ex:notleftcompleteag}) gives a counterexample as follows. Choose $K \in D(X'_\et)$ for which $K \not\simeq \R\lim \tau^{\geq -n} K$. Then there is an $X \in X'_\et$ for which $\R\Gamma(X,K) \not\simeq \R\Gamma(X,\R\lim \tau^{\geq -n} K) \simeq \R\Gamma(X_\proet,\nu^* K)$  (here we use Remark \ref{rmk:unboundedpullback}). The map $X \to \Spec(\Z)$ with $F = K|_{X}$ gives the desired counterexample.
\end{remark}

\begin{remark}
One reason to prefer the pro-\'etale topology to the fpqc topology is that the analogue of Lemma \ref{lem:funcpushforward} fails for the latter:  \'etale pushforwards do not commute with arbitrary base change.
\end{remark}

Lemma \ref{lem:funcpushforward} and the repleteness of the pro-\'etale topology let us access pushforwards of unbounded complexes quite easily; as pointed out by Brian Conrad, a similar statement can also be shown for $D(X_\et)$ using Hartshorne's formalism of ``way-out'' functors.

\begin{lemma}
	\label{lem:cohdimetproet}
	Let $f:X \to Y$ be a map of qcqs schemes. Assume $f_*:\Mod(X_\et,F) \to \Mod(Y_\et,F)$ has cohomological dimension $\leq d$ for a ring $F$. Then $f_*:D(X_\proet,F) \to D(Y_\proet,F)$ carries $D^{\leq k}_{cc}(X_\proet,F)$ to $D^{\leq k+d+1}_{cc}(Y_\proet,F)$. 
\end{lemma}
\begin{proof}
	Fix $K \in D^{\leq k}_{cc}(X_\proet)$. Then $K \simeq \R\lim \tau^{\geq -n} K$ by repleteness, so $f_* K \simeq \R\lim f_* \tau^{\geq -n} K$. Lemma \ref{lem:funcpushforward} and the assumption on $f$ show $f_* \tau^{\geq -n} K \in D^{\leq k+d}_{cc}(Y_\proet)$. As $\R\lim$ has cohomological dimension $\leq 1$ by repleteness, half of the claim follows. It remains to check that $\calH^i (f_* K) \in \nu^* \Ab(Y_\et)$. For this, observe that, for fixed $i$, the projective system $\{ \calH^i(f_* \tau^{\geq -n} K) \}$ is essentially constant: for $n \gg 0$, the map $f_* \tau^{\geq -(n+1)} K \to f_* \tau^{\geq -n}  K$ induces an isomorphism on $\calH^i$ by assumption on $f$. By repleteness, this proves $\calH^i (f_* K) \simeq \calH^i (f_* \tau^{\geq -n} K)$ for $n \gg 0$, which is enough by Lemma \ref{lem:funcpushforward}.
\end{proof}

\subsection{Relation with Ekedahl's theory}
\label{subsec:ekedahl}

In this section, we fix a noetherian ring $R$ complete for the topology defined by an ideal $\fram \subset R$. For this data, we follow the notation of \S \ref{subsec:derivedcompfadic} with $\calX = \Shv(X_\proet)$. We use here the following (slight variations on) assumptions introduced by Ekedahl, \cite{Ekedahl}.

\begin{definition}\ 
\begin{enumerate}
\item[{\rm (A)}] There is an integer $N$ and a set of generators ${Y_i}$, $Y_i\in X_\et$, of $X_\et$, such that for all $R/\fram$-modules $M$ on $X_\et$, $H^n(Y_i,M) = 0$ for $n>N$.
\item[{\rm (B)}] The ideal $\fram$ is regular, and the $R/\fram$-module $\fram^n/\fram^{n+1}$ has finite flat dimension bounded independently of $n$.
\end{enumerate}
\end{definition}

Here, condition (A) agrees with Ekedahl's condition (A), but condition (B) may be slightly stronger than Ekedahl's condition (B). By Proposition \ref{prop:leftcompletecrit} (2), condition (A) ensures that $D(X_\et,R/\fram)$ is left-complete, as are all $D(X_\et,R/\fram^n)$. Ekedahl considers the following category.

\begin{definition} If condition (A) is fulfilled, let $\ast = -$, if condition (B) is fulfilled, let $\ast = +$, and if condition (A) and (B) are fulfilled, let $\ast$ be empty. Define $D^\ast_{Ek}(X,R)$ as the full subcategory of $D^\ast(X_\et^{\N^\opp},R_\bullet)$ spanned by projective systems $\{M_n\}$ whose transition maps $M_n\otimes_{R/\fram^n} R/\fram^{n-1}\to M_{n-1}$ are isomorphisms for all $n$.
\end{definition}

In the pro-\'etale world, limits behave better, so we can define the following analogue:

\begin{definition} Define $D_{Ek}(X_\proet,\widehat{R}) \subset D_\comp(X_\proet,\widehat{R})$ as the full subcategory of complexes $K$ satisfying $K \otimes_{\widehat{R}} R/\fram \in D_{cc}(X_\proet)$, i.e., $H^i(K \otimes_{\widehat{R}} R/\fram) \in \nu^* \Ab(X_\et)$ for all $i$. If $\ast\in \{+,-,b\}$, let $D^\ast_{Ek}(X_\proet,\widehat{R})\subset D_{Ek}(X_\proet,\widehat{R})$ be the full subcategory with corresponding boundedness assumptions.
\end{definition}

The main comparison is:

\begin{proposition} If condition (A) is fulfilled, let $\ast = -$, if condition (B) is fulfilled, let $\ast = +$, and if condition (A) and (B) are fulfilled, let $\ast$ be empty. There is a natural equivalence $D^\ast_{Ek}(X_\proet,\widehat{R}) \simeq D^\ast_{Ek}(X_\et,R)$.
\end{proposition}

\begin{proof} Assume first that condition (A) is satisfied. By Lemma \ref{lem:completecompatiblesystems} (iv), we have $D^-_{comp}(X_\proet,\widehat{R})\simeq D^-_{comp}(X_\proet^{\N^\opp},R_\bullet)$. The full subcategory $D^-_{Ek}(X_\proet,\widehat{R})$ consists of those $\{K_n\}\in D^-_{comp}(X_\proet^{\N^\opp},R_\bullet)$ for which $K_n\in D^-_{cc}(X_\proet,R/\fram^n)$ for all $n$, as follows easily by induction on $n$. Under condition (A), $D(X_\et,R/\fram^n)$ is left-complete, so $D^-(X_\et,R/\fram^n)\cong D^-_{cc}(X_\proet,R/\fram^n)$. This gives the result.

Now assume condition (B). Thus, there exists $N \in \N$ such that if $K \in D^{\geq k}_{Ek}(X_\proet,\widehat{R})$ for some $k$, then $K \otimes_{\widehat{R}} R/\fram^n  \in D^{\geq k-N}_{cc}(X_\proet)$ for all $n$.  Hence, by Lemma \ref{lem:completecompatiblesystems}, we may view $D^+_{Ek}(X_\proet,\widehat{R})$ as the full subcategory of $D^+_\comp(X_\proet^{\N^\opp},R_\bullet)$ spanned by those $\{K_n\}$ with $K_n \in  D^+_{cc}(X_\proet)$. Moreover, by Proposition \ref{prop:parasiticsemiorth}, $\nu^*$ induces an equivalence $D^+(X_\et) \simeq D^+_{cc}(X_\proet)$. The desired equivalence is then induced by $\{M_n\} \mapsto \{\nu^* M_n\}$ and $\{K_n\} \mapsto \{\nu_* K_n\}$.

If condition (A) and (B) are satisfied, simply combine the two arguments.
\end{proof}

\subsection{Relation with Jannsen's theory} 

Fix a scheme $X$.  In \cite[\S 3]{JannsenContsCoh}, one finds the following definition:

\begin{definition}
	The {\em continuous \'etale cohomology} $H^i_{\cont}(X_\et,\{F_n\})$ of $X$ with coefficients in a pro-system $\{F_n\}$ of abelian sheaves on $X_\et$ is the value of the $i$-th derived functor of the functor $\Ab(X_\et)^\N \to \Ab$ given by $\{F_n\} \mapsto  H^0(X_\et,\lim F_n)$.
\end{definition}

In general, the groups $H^i_{\cont}(X_\et,\{F_n\})$  and $H^i(X_\et, \lim F_n)$ are distinct, even for the projective system $\{\Z/\ell^n\}$; the difference is explained by the derivatives of the inverse limit functor. As inverse limits are well-behaved in the pro-\'etale world, this problem disappears, and we obtain:

\begin{proposition}
	Let $\{F_n\}$ is a pro-system of abelian sheaves on $X_\et$ with surjective transition maps. Then there is a canonical identification 
	\[ H^i_{\cont}(X_\et, \{F_n\}) \simeq H^i(X_\proet, \lim \nu^* F_n).\]
\end{proposition}
\begin{proof}
	Write $\R\Gamma_{\cont}(X_\et,\{F_n\}) := \R\Gamma(X_\et,\R\lim F_n)$, so $H^i(\R\Gamma_{\cont}(X_\et,\{F_n\})) \simeq H^i_{\cont}(X_\et,\{F_n\})$ as defined above by the Grothendieck spectral sequence for composition of derived functors. We then have
	\[ \R\Gamma_{\cont}(X_\et,\{F_n\}) \simeq \R\lim \R\Gamma(X_\et, F_n) \simeq \R\lim\R\Gamma(X_\proet, \nu^* F_n) \simeq \R\Gamma(X_\proet,\R\lim \nu^* F_n); \]
	here the first and last equality use the commutation of $\R\Gamma$ and $\R\lim$, while the second equality comes from the boundedness of $F_n \in D(X_\et)$. The assumption on $\{F_n\}$ ensures that $\R\lim F_n \simeq \lim F_n$ by the repleteness of $X_\proet$, which proves the claim.
\end{proof}

\newpage

\section{Constructible sheaves}\label{sec:Constructible}

This long section studies constructible sheaves, with the ultimate goal of giving a different perspective on the notion of a $\overline{\Q}_\ell$-sheaf. We begin by studying in \S \ref{ss:FunClosed} and \S \ref{ss:FunLocClosed} the basic functoriality of pushforward and pullback along locally closed immersions; the main novelty here is that pullback along a closed immersion is limit- and colimit-preserving, contrary to the classical story. Next, we recall the theory of constructible complexes in the \'etale topology in \S \ref{ss:ConsEt}. We alert the reader that our definition of constructibility is more natural from the derived perspective, but not the usual one: a constructible complex on a geometric point is the same thing as a {\em perfect} complex, see Remark \ref{rmk:consperfect}. In particular, the truncation operators $\tau_{\geq n}$, $\tau_{\leq n}$ do not in general preserve constructibility. As a globalisation of this remark, we detour in \S \ref{ss:ConsCompact} to prove that constructible complexes are the same as compact objects under a suitable finiteness constraint; this material is surely standard, but not easy to find in the literature. We then introduce constructible complexes in the pro-\'etale world in \S \ref{ss:ConsProet} with coefficients in a complete noetherian local ring $(R,\fram)$ as those $R$-complexes on $X_\proet$ which are complete (in the sense of \S \ref{subsec:derivedcompfadic}), and classically constructible modulo $\fram$. This definition is well-suited for comparison with the classical picture, but, as we explain in \S \ref{ss:ConsNoeth}, also coincides with the more intuitive definition on a noetherian scheme: a constructible complex is simply an $R$-complex that is locally constant and perfect along a stratification. This perspective leads in \S \ref{ss:LadicSheaves} to a direct construction of the category of constructible complexes over coefficient rings that do not satisfy the above constraints, like $\overline{\Z}_\ell$ and $\overline{\Q}_\ell$. Along the way, we establish that the formalism of the $6$ functors ``works'' in this setting in \S \ref{ss:6Fun}.

\subsection{Functoriality for closed immersions}
\label{ss:FunClosed}

Fix a qcqs scheme $X$, and a qcqs open $j:U \hookrightarrow X$ with closed complement $i:Z \to X$. We use the subscript ``$0$'' to indicate passage from $X$ to $Z$. First, we show ``henselizations'' can be realised as pro-\'etale maps. 

\begin{lemma}
	\label{lem:henselize}
	Assume $X$ is affine. Then $i^{-1}:X^\aff_\proet \to Z^\aff_\proet$ admits a fully faithful left adjoint $V \mapsto \widetilde{V}$. In particular, we have $i^{-1}(\widetilde{V}) \simeq V$.
\end{lemma}
\begin{proof} See Definition \ref{def:henselization} and Lemma \ref{lem:HensFullyFaithful}.
\end{proof}

Henselization defines a limit-preserving functor between sites:

\begin{lemma}
	\label{lem:henselisecovers}
	Assume $X$ is affine. Then the functor $V \mapsto \widetilde{V}$ from Lemma \ref{lem:henselize} preserves surjections.
\end{lemma}
\begin{proof}
Fix $V = \Spec(A_0)$ with $\widetilde{V} = \Spec(A)$ for a ring $A$ that is henselian along $I = \ker(A \to A_0)$.  It suffices to show that any \'etale map $W \to \widetilde{V}$ whose image contains $V \subset \widetilde{V}$ is surjective.  The complement of the image gives a closed subset of $\widetilde{V}$ that misses $V$, but such sets are empty as $I$ lies in the Jacobson radical of $A$ by assumption.
\end{proof}

Contrary to the \'etale topology, we can realise $i^*$ simply by evaluation in the pro-\'etale world:

\begin{lemma}
	\label{lem:closedimmpullback}
If $X$ is affine, then  $i^* F(V) = F(\widetilde{V})$ for any w-contractible $V \in Z^\aff_\proet$ and $F \in \Shv(X_\proet)$.
\end{lemma}
\begin{proof}
Clearly, $i^* F$ is the sheafification of $V \mapsto F(\widetilde{V})$ on $Z^\aff_\proet$. On w-contractible objects, sheafification is trivial, giving the result.
\end{proof}

\begin{remark}
	\label{rmk:AffineAnalog}
It follows from the affine analogue of proper base change, \cite{GabberAffineBC}, \cite{HuberAffineBC}, that for classical torsion sheaves $F$, $i^* F(V) = F(\widetilde{V})$ for \emph{all} $V\in Z^\aff_\proet$; in fact, the affine analogue of proper base change says precisely that
\[
R\Gamma(V,i^* F) = R\Gamma(\widetilde{V},F)\ .
\]
\end{remark}

As $i^*$ is realised by evaluation, it commutes with limits (which fails for $X_\et$, see Example \ref{ex:pullbacklimits}):

\begin{corollary}
	\label{cor:closedimmpullbackadjoints}
	The pullback $i^*:\Shv(X_\proet) \to \Shv(Z_\proet)$ commutes with all small limits and colimits.
\end{corollary}
\begin{proof}
	The claim about colimits is clear by adjunction. For limits, we must show that the natural map $i^* \lim_i F_i \to  \lim_i i^* F_i$ is an isomorphism for any small diagram $F:I \to \Shv(X_\proet)$. As this is a local statement, we may assume $X$ is affine. The claim now follows from Lemma \ref{lem:closedimmpullback} by evaluating either side on w-contractible objects in $Z^\aff_\proet$.
\end{proof}

The next example illustrates how $i^*$ fails to be limit-preserving on the \'etale site:

\begin{example}
	\label{ex:pullbacklimits}
Consider $X = \Spec(k[x])$ with $k$ an algebraically closed field, and set $i:Z \hookrightarrow X$ to be the closed immersion defined by $I = (x)$. Let $R = k[x]$, and set $S$ to be the strict henselisation of $R$ at $I$, so $S = \colim_i S_i$ where the colimit runs over all \'etale neighbourhoods $R \to S_i \to k$ of $Z \to X$. Now consider the projective system $\{\calO_X/I^n\}$ in $\Shv(X_\et)$. Then $i^*(\calO_X/I^n) = S/IS^n$, so $\lim i^*(\calO_X/I^n)$ is the $I$-adic completion of $S$. On the other hand, $i^*(\lim \calO_X/I^n) = \colim_i \lim S_i/I^n$ is the colimit of the $I$-adic completions of each $S_i$; one can check that this colimit is not $I$-adically complete.
\end{example}

\begin{remark}
	Corollary \ref{cor:closedimmpullbackadjoints} shows that $i^*$ has a right adjoint $i_*$ as well as a left-adjoint $i_\#$. The latter is described as the unique colimit-preserving functor sending $V \in Z^\aff_\proet$ to $\widetilde{V} \in X^\aff_\proet$. Note that $i_\#$ is not left-exact in general, so there is no easy formula computing $\R\Gamma(V,i^* F)$ in terms of $\R\Gamma(\widetilde{V},F)$ for $V \in Z_\proet$ (except in the torsion case, as in Remark \ref{rmk:AffineAnalog}).
\end{remark}

\begin{lemma}
	\label{lem:closedimmacyclic}
The pushforward $i_*:\Shv(Z_\proet) \to \Shv(X_\proet)$ is exact.
\end{lemma}
\begin{proof}
	Fix a surjection $F \to G$ in $\Shv(Z_\proet)$. We must show $i_* F \to i_* G$ is surjective. As the claim is local, we may work with affines. Fix $Y \in X^\aff_\proet$ and $g \in i_* G(Y) = G(Y_0)$. Then there is a cover $W \to Y_0$ in $Z_\proet$ and a section $f \in F(W)$ lifting $g$. The map $\widetilde{W} \sqcup Y|_U \to Y$ is then a cover by Lemma \ref{lem:henselize}; here we use that $U \subset X$ is quasicompact, so $Y|_U$ is also quasicompact. One has $i_* F(Y|_U) = F(\emptyset) = \ast$, and $i_* F(\widetilde{W}) = F(\widetilde{W}_0) = F(W)$, so $f$ gives a section in $i_* F(\widetilde{W} \sqcup Y|_U)$ lifting $g$.
\end{proof}

We can now show that $i_*$ and $j_!$ behave in the expected way.

\begin{lemma}
	\label{lem:closeddevissage}
For any pointed sheaf $F \in \Shv(X_\proet)$, the adjunction map $F \to i_* i^* F$ is surjective. 
\end{lemma}
\begin{proof}
Since the statement is local, we may assume $X$ is affine. Fix $V \in X^\aff_\proet$. Then $i_* i^* F(V) = i^*F(V_0) = F(\widetilde{V_0})$. Now observe that $\widetilde{V_0} \sqcup V|_U \to V$ is a pro-\'etale cover. Since $F(V|_U) \neq \emptyset$ (as $F$ is pointed), one easily checks that any section in $i_* i^* F(V)$ lifts to a section of $F$ over $\widetilde{V_0} \sqcup V|_U$, which proves surjectivity. 
\end{proof}

\begin{remark}
	Lemma \ref{lem:closeddevissage} needs $F$ to be pointed. For a counterexample without this hypothesis,  take: $X = U \sqcup Z$ a disjoint union of two non-empty schemes $U$ and $Z$, and $F = i_! \Z$, where $i:Z \to X$  is the clopen immersion with complement $j:U \to X$.
\end{remark}

\begin{lemma}
	\label{lem:opencloseddevissage}
	For any pointed sheaf $F \in \Shv(X_\proet)$, we have $j_!j^* F \simeq \ker(F \to i_* i^* F)$.
\end{lemma}
\begin{proof}
	We may assume $X$ is affine. For any $V \in X^\aff_\proet$, we first observe that the sheaf axiom for the cover $\widetilde{V_0} \sqcup V|_U \to V$ gives a fibre square of pointed sets
	\[ \xymatrix{ F(V) \ar[r] \ar[d] & F(V|_U) \ar[d] \\
		 	F(\widetilde{V_0}) \ar[r] & F(\widetilde{V_0}|_U). }\]
			In particular, $\ker(F(V) \to F(\widetilde{V_0})) \simeq \ker(F(V|_U) \to F(\widetilde{V_0}|_U))$. Now $i_* i^* F(V) = F(\widetilde{V_0})$, so we must show that $j_! j^* F(U) = \ker(F(V) \to F(\widetilde{V_0})) \simeq \ker(F(V|_U) \to F(\widetilde{V_0}|_U))$. By definition, $j_! j^* F$ is the sheaf associated to the presheaf $F'$ defined via: $F'(V) = F(V)$ if $V \to X$ factors through $U$, and $F'(V) = 0$ otherwise. The sheaf axiom for the cover $\widetilde{V_0} \sqcup V|_U \to V$ then shows that $j_! j^* F$ is also the sheaf associated to the presheaf $F''$ given by $F''(V) = \ker(F(V|_U) \to F(\widetilde{V_0}|_U))$, which proves the claim.
\end{proof}

\begin{lemma}
	\label{lem:opencloseddevissagederived}
One has the following identification of functors at the level of unbounded derived categories:
\begin{enumerate}
	\item $i^* i_* \simeq \id$ and $j^* j_! \simeq j^* j_* \simeq \id$.	
	\item $j^* i_* \simeq 0$, and $i^* j_! \simeq 0$.
\end{enumerate}
\end{lemma}
\begin{proof}
	By deriving Lemma \ref{lem:opencloseddevissage}, there is an exact triangle $j_! j^*  \to \id \to i_* i^*$ of endofunctors on $D(X_\proet)$. Then (2) follows from (1) by applying $i^*$ and $j^*$ to this triangle. The second part of (1) is a general fact about monomorphisms $U \hookrightarrow X$ in a topos. For $i^* i_* \simeq \id$, we use that both functors are exact to reduce to the claim at the level of abelian categories, where it follows from $\widetilde{V}_0 \simeq V$ for any $V \in Z^\aff_\proet$.
\end{proof}

\begin{lemma}
	\label{lem:shriekpushforwardcont}
The pushforward $j_!:D(U_\proet) \to D(X_\proet)$ commutes with homotopy-limits.
\end{lemma}
\begin{proof}
	By Lemma \ref{lem:opencloseddevissage}, for any $K \in D(U_\proet)$, we have the following exact triangle:
\[ j_!K \to j_*K \to i_* i^* j_* K.\]
Since $j_*$, $i^*$ and $i_*$ all commute with homotopy-limits, the same is true for $j_!$.
\end{proof}

\begin{remark}
	One can show a more precise result than Lemma \ref{lem:shriekpushforwardcont}. Namely, the pushforward  $j_!:D(U_\proet) \to D(X_\proet)$ admits a left-adjoint $j^\#:D(X_\proet) \to D(U_\proet)$ which is defined at the level of free abelian sheaves as follows: given $V \in X_\proet$, we have $j^\#(\Z_V) = \cok(\Z_{\widetilde{V_0}|_U} \to \Z_{V|_U}) \simeq \cok(\Z_{\widetilde{V_0}} \to \Z_V)$.
\end{remark}

We record some special cases of the proper base change theorem:

\begin{lemma}
	\label{lem:pbcopenclosed}	
Consider the diagram
\[ \xymatrix{ f^{-1} Z \ar[r]^-i \ar[d]^-f & Y \ar[d]^-f & f^{-1} U \ar[d]^-f \ar[l]^-j \\
		Z \ar[r]^-i & X & U \ar[l]^-j } \]
For any $K \in D(U_\proet)$ and $L \in D(Z_\proet)$, we have
\[ i_* f^* L \simeq f^* i_* L \quad \mathrm{and} \quad j_! f^* K \simeq f^* j_! K.\]
\end{lemma}
\begin{proof}
Note that $i^* f^* i_* L \simeq f^* i^* i_* L \simeq f^* L$. Hence, using the sequence $j_! j^* \to \id \to i_* i^*$ of functors, to prove the claim for $L$, it suffices to show $j^* f^* i_* L \simeq 0$; this is clear as $j^* f^* i_* \simeq f^* j^* i_* \simeq 0$, since $j^* i_* \simeq 0$. The second claim follows by an analogous argument using $i^* j_! \simeq 0$.
\end{proof}

We end by noting that $i_*$ also admits a right adjoint:

\begin{lemma}
	\label{lem:shriekpullbackclosed}
The functor $i_*:D(Z_\proet) \to D(X_\proet)$ admits a right adjoint $i^!:D(X_\proet) \to D(Z_\proet)$. For any $K \in D(X_\proet)$, there is an exact triangle
\[ i_* i^! K \to K \to j_* j^* K.\]
\end{lemma}
\begin{proof}
The functor $i_*:D(Z_\proet) \to D(X_\proet)$ commutes with arbitrary direct sums. As all triangulated categories in sight are compactly generated, one formally deduces the existence of $i^!$. For the exact triangle, write $L$ for the homotopy-kernel of $K \to j_* j^* K$. One has a natural map $\eta:i_* i^! K \to L$ since $\R\Hom(i_* i^! K,j_*j^* K) = 0$. We first show $\eta$ is an isomorphism through its functor of points. For this, note that for any $M \in D(Z_\proet)$, one has 
\[ \R\Hom(i_*M,i_*i^! K) = \R\Hom(M,i^! K) = \R\Hom(i_* M, K) = \R\Hom(i_* M,L),\]
where the first equality uses the full faithfulness of $i_*$, the second comes from the definition of $i^!$, and the last one uses $\R\Hom(i_*M,j_*j^*K) = 0$. This proves that $\eta$ is an isomorphism. One also has $L = i_* i^* L$ as $j^* L = 0$, so the claim follows by full faithfulness of $i_*$.
\end{proof}

Finite morphisms are acyclic under finite presentation constraints:

\begin{lemma}
	\label{lem:finmapacyclic}
If $f:X \to Y$ is finitely presented and finite, then $f_*:\Ab(X_\proet) \to \Ab(Y_\proet)$ is exact.
\end{lemma}
\begin{proof}
	This follows from Lemma \ref{lem:finiteovercontractible}.
\end{proof}

\subsection{Functoriality for locally closed immersions}
\label{ss:FunLocClosed}

We fix a qcqs scheme $X$, a locally closed constructible subset $k:W \hookrightarrow X$.  We write $D_W(X_\proet)$ for the full subcategory spanned by $K \in D(X_\proet)$ with $K|_{X-W} \simeq 0$; we refer to such objects as ``complexes supported on $W$.''

\begin{lemma} 
	\label{lem:complexeswithsupport}
	Fix $i:Z \hookrightarrow X$ a constructible closed immersion with complement $j:U \hookrightarrow X$. Then one has:
	\begin{enumerate}
		\item The functor $j_!$ establishes an equivalence $D(U_\proet) \simeq D_U(X_\proet)$ with inverse $j^*$.
		\item The functor $i_*$ establishes an equivalence $D(Z_\proet) \simeq D_Z(X_\proet)$ with inverse $i^*$.
		\item The functor $k^*$ establishes an equivalence $D_W(X_\proet) \simeq D(W_\proet)$.
	\end{enumerate}
\end{lemma}
\begin{proof}
	For (1), we know that $j^* j_! \simeq \id$, so $j_!$ is fully faithful. Also, an object $K \in D(X_\proet)$ is supported on $U$ if and only if $i^* K \simeq 0$ if and only if $j_! j^* K \simeq K$, which proves (1). The proof of (2) is analogous. For (3), fix a factorization $W \stackrel{f}{\to} \overline{W} \stackrel{g}{\to} X$ with $f$ an open immersion, and $g$ a constructible closed immersion. Then $g_*$ induces an equivalence $D(\overline{W}_\proet) \simeq D_{\overline{W}}(X_\proet)$ with inverse $g^*$ by (2), and hence restricts to an equivalence $D_W(\overline{W}_\proet) \simeq D_W(X_\proet)$. Similarly, $f_!$ induces an equivalence $D(W_\proet) \simeq D_W(\overline{W}_\proet)$ with inverse $f^*$ by (1). Hence, the composition $k_! := g_* \circ f_!$ induces an equivalence $D(W_\proet) \simeq D_W(X_\proet)$ with inverse $k^*$.
\end{proof}

\begin{definition}
	\label{defn:compactlysupportedpushforwardimmersion}
	The functor $k_!:D(W_\proet) \to D(X_\proet)$ is defined as the composition $D(W_\proet) \stackrel{a}{\to} D_W(X_\proet) \stackrel{b}{\to} D(X_\proet)$, where $a$ is the equivalence of Lemma \ref{lem:complexeswithsupport} (inverse to $k^*$), and $b$ is the defining inclusion.
\end{definition}

\begin{lemma} 
	\label{lem:shriekpushforwardlocallyclosed}
	One has:
	\begin{enumerate}
		\item The functor $k_!$ is fully faithful, preserves homotopy-limits, and has a left inverse given by $k^*$.
		\item For any map $f:Y \to X$ of qcqs schemes, one has $k_! \circ f^* \simeq f^* \circ k_!$ as functors $D(W_\proet) \to D(Y_\proet)$.
		\item For any $K \in D(W_\proet)$ and $L \in D(X_\proet)$, we have $k_! K \otimes L \simeq k_! (K \otimes i^* L)$. 
		\item One has $k_! \circ \nu^* \simeq \nu^* \circ k_!$ as functors $D(W_\et) \to D(X_\proet)$.
		\item The functor $k_!$ admits a right adjoint $k^!:D(X_\proet) \to D(W_\proet)$.
	\end{enumerate}
\end{lemma}
\begin{proof}
	(1) follows from the proof of Lemma \ref{lem:complexeswithsupport} as both $f_!$ and $g_*$ have the same properties. (2) follows by two applications of Lemma \ref{lem:pbcopenclosed}. For (3), it suffices to separately handle the cases where $k$ is an open immersion and $k$ is a closed immersion. The case of an open immersion (or, more generally, any weakly \'etale map $k:W \to X$) follows by general topos theory and adjunction. Hence, we may assume $k$ is a closed immersion with open complement $j:U \hookrightarrow X$, so $k_! \simeq k_*$. For any $K' \in D(X_\proet)$, we have the triangle
\[ j_! j^* K' \to K' \to k_* k^* K'.\]
Tensoring this triangle with $L$ and using the projection formula for $j$ shows $k_* k^* K' \otimes L \simeq k_*\big(k^* K' \otimes k^* L)$. Setting $K' = k_* K$ then proves the claim as $k^* k_* \simeq \id$. For (4), assume first that $k$ is an open immersion. Then $\nu_* \circ k^* \simeq k^* \circ \nu_*$ as functors $D(X_\proet) \to D(U_\et)$ (which is true for any $U \to X$ in $X_\et$). Passing to adjoints then proves $k_! \circ \nu^* \simeq \nu^* \circ k_!$.  Now assume $k$ is a constructible closed immersion with open complement $j:U \hookrightarrow X$. Then for any $K \in D(X_\et)$, there is a triangle
\[ j_! j^* K \to K \to i_* i^* K\]
in $D(X_\et)$. Applying $\nu^*$ and using the commutativity of $\nu^*$ with $j_!$, $j^*$ and $i^*$ then proves the claim. (5) follows by considering the case of open and constructible closed immersions separately, and using Lemma \ref{lem:shriekpullbackclosed}.
\end{proof}

All the results in this section, except the continuity of $k_!$, are also valid in the \'etale topology.

\subsection{Constructible complexes in the \'etale topology}
\label{ss:ConsEt}

The material of this section is standard, but we include it for completeness. We fix a qcqs scheme $X$, and a ring $F$. Given an $F$-complex $L \in D(F)$, we write $\underline{L}$ for the associated constant complex, i.e., its image under the pullback $D(F) \to D(X_\et,F)$.

\begin{definition}
	\label{def:constructibleetale}
	A complex $K \in D(X_\et,F)$ is called {\em constructible} if there exists a finite stratification $\{X_i \to X\}$ by constructible locally closed $X_i\subset X$ such that $K|_{X_i}$ is locally constant with perfect values on $X_\et$.
\end{definition}

\begin{remark}
	\label{rmk:consperfect}
	One classically replaces the perfectness hypothesis in Definition \ref{def:constructibleetale} with a weaker finiteness constraint. However, imposing perfectness is more natural from the derived point of view: under mild conditions on $X$, our definition picks out the compact objects of $D(X_\et,F)$ (see  Proposition  \ref{lem:constructiblecompact}), and is stable under the six operations. Moreover, the two approaches coincide when $F$ is a field.
\end{remark}

\begin{lemma}
	\label{lem:consetalefiltration}
Any $K \in D_\cons(X_\et,F)$ admits a finite filtration with graded pieces of the form $i_! L$ with $i:Y \hookrightarrow X$ ranging through a stratification of $X$, and $L \in D(Y_\et,F)$ locally constant with perfect values.
\end{lemma}
\begin{proof}
	Same as in the classical case, see \cite[Proposition IX.2.5]{SGA4Tome3}.
\end{proof}

\begin{lemma}
Each $K \in D_\cons(X_\et,F)$ has finite flat dimension.
\end{lemma}
\begin{proof}
By Lemma \ref{lem:consetalefiltration}, we may assume $K = i_! L$ for $i:Y \hookrightarrow X$ locally closed constructible, and $L \in D(Y_\et,F)$ locally constant with perfect values. By the projection formula, it suffices to show $L$ has finite flat dimension. As we are free to localize, we may assume $L = \underline{K'}$ with $K' \in D_\perf(F)$, whence the claim is clear.
\end{proof}

\begin{lemma}
	$D_\cons(X_\et,F) \subset D(X_\et,F)$ is closed under tensor products.
\end{lemma}
\begin{proof}
	Clear.
\end{proof}

\begin{lemma}
	\label{lem:globalsectionsconstant}
	Given $K \in D(R)$ and $s \in H^0(X_\et,\underline{K})$, there exists an \'etale cover $\{U_i \to X\}$ such that $s|_{U_i}$ comes from $s_i \in H^0(K)$.
\end{lemma}
\begin{proof}
	Fix a geometric point $x:\Spec(k) \to X$, and consider the cofiltered category $I$ of factorizations $\Spec(k) \to U \to X$ of $x$ with $U \to X$ \'etale. Then $K \simeq \colim \R\Gamma(U_\et,\underline{K})$ where the colimit is indexed by $I^\opp$: the exact functor $x^*(F) =  \colim_I F(U)$ gives a point $x:\Set \to X_\et$, and the composition $(\Set,F) \stackrel{x}{\to} (X_\et,F) \stackrel{\can}{\to} (\Set,F)$ is the identity. This gives a section $s_i \in H^0(K)$ by passage to the limit. As filtered colimits are exact, one checks that $s$ agrees with the pullback of $s_i$ over some neighbourhood $U \to X$ in $I$.  Performing this construction for each geometric point then gives the desired \'etale cover.
\end{proof}

\begin{lemma}
	\label{lem:constancycrit}
	If $K \in D^b(X_\et,F)$ has locally constant cohomology sheaves, then there is an \'etale cover $\{U_i \to X\}$ such that $K|_{U_i}$ is constant.
\end{lemma}
\begin{proof}
	We may assume all cohomology sheaves of $K$ are constant. If $K$ has only one non-zero cohomology sheaf, there is nothing to prove. Otherwise, choose the maximal $i$ such that $\calH^i(K) \neq 0$. Then $K \simeq \ker(\calH^i(K)[-i] \stackrel{s}{\to} \tau^{< i} K[1])$. By induction, both $\calH^i(K)$ and $\tau^{< i}K$ can be assumed to be constant. The claim now follows by Lemma \ref{lem:globalsectionsconstant} applied to $\underline{\R\Hom}(\calH^i(K)[-i], \tau^{< i} K[1])$ with global section $s$; here we use that the pullback $G:D(F) \to D(X_\et,F)$ preserves $\underline{\R\Hom}$ between $A,B \in D^b(F)$ since $G(\R\lim C_i) = \R\lim G(C_i)$ if $\{C_i \hookrightarrow C\}$ is the stupid filtration on $C \in D^+(R)$ (with $C = \R\Hom(A,B)$ calculated by a projective resolution of $A$).
\end{proof}

\begin{lemma}
A complex $K \in D(X_\et,F)$ is constructible if and only if for any finite stratification $\{Y_i \to X\}$, the restriction $K|_{Y_i}$ is constructible.
\end{lemma}
\begin{proof}
The forward direction is clear as constructible sheaves are closed under pullback. For the reverse, it suffices to observe $k_!$ preserves constructibility for $k:W \hookrightarrow X$ locally closed constructible as $k$ identifies constructible subsets of $W$ with those of $X$ contained in $W$.
\end{proof}

\begin{lemma}
	\label{lem:constriangulated}
$D_\cons(X_\et,F)$ is a triangulated idempotent complete subcategory of $D(X_\et,F)$. It can be characterized as the minimal such subcategory that contains all objects of the form $k_! L$ with $k:Y \hookrightarrow X$ locally closed constructible, and $L \in D(Y_\et,F)$ locally constant with perfect values.
\end{lemma}
\begin{proof}
	To show $D_\cons(X_\et,F)$ is closed under triangles, by refining stratifications, it suffices to check: if $K,L \in D(X_\et,F)$ are locally constant with perfect values, then the cone of any map $K \to L$ has the same property. Replacing $X$ by a cover, we may assume $K = \underline{K'}$ and $L = \underline{L'}$ with $K',L' \in D_\perf(R)$. The claim now follows from Lemma \ref{lem:globalsectionsconstant} applied to $\underline{\R\Hom}(K',L')$. The idempotent completeness is proven similarly. The last part follows from Lemma \ref{lem:consetalefiltration} and the observation that each $k_! L$ (as in the statement) is indeed constructible.
\end{proof}

\begin{lemma}
	\label{lem:conslocalet}
	Constructibility is local on $X_\et$, i.e., given $K \in D(X_\et,F)$, if there exists a cover $\{f_i:X_i \to X\}$ in $X_\et$ with $f_i^*K$ constructible, then $K$ is constructible.
\end{lemma}
\begin{proof}
We may assume $f:Y \to X$ is a surjective \'etale map, and $f^* K$ is constructible. First assume that $f$ is a finite \'etale cover. Passing to Galois closures (and a clopen cover of $X$ if necessary), we may assume $f$ is finite Galois with group $G$. By refining strata, we can assume $f^* K$ is locally constant along a $G$-invariant stratification of $Y$. Such a stratification is pulled back from $X$, so the claim is clear. In general, there is a stratification of $X$ over which $f$ is finite \'etale, so one simply applies the previous argument to the strata.
\end{proof}

\begin{lemma}
	\label{lem:shriekpushforwardcons}
If $j:U \to X$ is qcqs \'etale, then $j_!:D(U_\et,F) \to D(X_\et,F)$ preserves constructibility.
\end{lemma}
\begin{proof}
If $j$ is finite \'etale, then the claim follows by Lemma \ref{lem:conslocalet} as any finite \'etale cover of $X$ is, locally on $X_\et$, of the form $\sqcup_{i=1}^n X \to X$. In general, there is a stratification of $X$ over which this argument applies.
\end{proof}

\begin{lemma}
	\label{lem:nilpidealcons}
	If $K \in D(X_\et,F)$, and $I \subset F$ is a nilpotent ideal such that $K \otimes_F F/I \in D_\cons(X_\et,F/I)$, then $K \in D_\cons(X_\et,F)$.
\end{lemma}
\begin{proof}
	We may assume $I^2 = 0$. By devissage, we may assume $K_1 = K \otimes_F F/I$ is locally constant with perfect value $L_1 \in D_\perf(F/I)$. By passage to an \'etale cover, we may assume $K_1 \simeq \underline{L_1}$. After further coverings, Lemma \ref{lem:constancycrit} shows $K \simeq \underline{L}$ for some $L \in D(F)$.  Since $L \otimes_F F/I \simeq L_1$ is perfect, so is $L$ (by the characterization of perfect complexes as compact objects of $D(F)$ and the $5$ lemma).
\end{proof}

\begin{lemma}
	\label{lem:conslocalproet}
	Constructibility is local in the pro-\'etale topology on $X$, i.e., given $K \in D(X_\et,F)$, if there exists a cover $\{f_i:X_i \to X\}$ in $X_\proet$ with $f_i^*K$ constructible, then $K$ is constructible.
\end{lemma}
\begin{proof}
	We may assume $X$ is affine, and that there exists a pro-\'etale affine $f:Y = \lim_i Y_i \to X$ covering $X$ with $f^*K$ constructible. The stratification on $Y$ witnessing the constructibility of $f^*K$ is defined over some $Y_i$. Hence, after replacing $X$ by an \'etale cover, we may assume that there exists a stratification $\{X_i \hookrightarrow X\}$ such that $f^*K$ is constant with perfect values over $f^{-1}(X_i)$. Replacing $X$ by $X_i$, we may assume $f^*K \simeq f^* \underline{L}$ with $L \in D_\perf(F)$. Then the isomorphism $f^* \underline{L} \to f^* K$ is defined over some $Y_i$ (since $L$ is perfect), so $K|_{Y_i}$ is constant.
\end{proof}

\begin{lemma}
	\label{lem:constructiblealmostcompact}
If $K \in D_\cons(X_\et,F)$, then $\R\Hom(K,-)$ commutes with all direct sums with terms in $D^{\geq 0}(X_\et,F)$.
\end{lemma}
\begin{proof}
	Let $\calC_X \subset D^b(X_\et,F)$ denote the full (triangulated) subcategory spanned by those $M$ for which  $\underline{\R\Hom}(M,-)$ commutes with all direct sums in $D^{\geq 0}(X_\et,F)$. Then one checks:
	\begin{enumerate}
		\item For any $M \in D_\perf(F)$, one has $\underline{M} \in \calC_X$. 
		\item For any qcqs \'etale map $j:U \to X$, the functor $j_!$ carries $\calC_U$ to $\calC_X$.
		\item The property of lying in $\calC_X$ can be detected locally on $X_\et$.
		\item $M \in D(X_\et,F)$ lies in $\calC_X$ if and only if $\R\Hom(M|_U,-)$ commutes with direct sums in $D^{\geq 0}(U_\et,F)$ for each qcqs $U \in X_\et$.
	\end{enumerate}
By (4), it suffices to show that a constructible complex $K$ lies in $\calC_X$. By Lemma \ref{lem:consetalefiltration}, we may assume $K = k_! L$ with $k:Y \hookrightarrow X$ locally closed constructible, and $L \in D(Y_\et,F)$ locally constant with perfect values. Choose a qcqs open $j:U \hookrightarrow X$ with $i:Y \hookrightarrow U$ a constructible closed subset. Then $K = k_! L \simeq (j_! \circ i_*) L$. By (2),  it suffices to show that $i_* K \in \calC_{U}$, i.e., we reduce to the case where $k$ is a constructible closed immersion with open complement $h:V \hookrightarrow X$. The assumption on $K$ gives a qcqs \'etale cover $g:Y' \to Y$ with $g^* L \simeq \underline{M}$ for $M \in D_\perf(F)$.  By passing to a cover of $X$ refining $g$ over $Y$, using (3), we may assume that $L = \underline{M}$. Then the exact triangle
	\[ h_! \underline{M} \to \underline{M} \to K \]
	and (1) and (2) above show that $K \in \calC_X$, as wanted.
\end{proof}

	\begin{remark}
	It is crucial to impose the boundedness condition in Lemma \ref{lem:constructiblealmostcompact}: if the cohomological dimension of $X$ is unbounded, then $\R\Hom(\underline{F},-) \simeq \R\Gamma(X_\et,-)$ does not commute with arbitrary direct sums in $D(X_\et,F)$.
\end{remark}

\begin{lemma}
	\label{lem:internalhometproet}
	For $K \in D_\cons(X_\et,F)$ and $L \in D^+(X_\et,F)$, one has
	\[\nu^* \underline{\R\Hom}(K,L) \simeq \underline{\R\Hom}(\nu^* K, \nu^* L)\ .\]
\end{lemma}
\begin{proof}
	Fix $U = \lim_i U_i \in X^\aff_\proet$, and write $j:U \to X$ and $j_i:U_i \to X$ for the structure maps. By evaluating on pro-\'etale affines, it suffices to check $\R\Hom(j^* K, j^* L) \simeq \colim_i \R\Hom(j_i^* K, j_i^* L)$. By adjunction, this is equivalent to requiring $\R\Hom(K, j_* j^* L) \simeq \colim_i \R\Hom(K, j_{i,\ast} j_i^* L)$. If $L \in D^{\geq k}(X_\et)$, then $j_{i,\ast} j_i^* L \in D^{\geq k}(X_\et)$ for all $i$, so the claim follows from Lemma \ref{lem:constructiblealmostcompact}.
\end{proof}

\subsection{Constructible complexes as compact objects}
\label{ss:ConsCompact}

The material of this section is not used in the sequel. However, these results do not seem to be recorded in the literature, so we include them here. We fix a qcqs scheme $X$, and a ring $F$. We assume that all affine $U \in X_\et$ have $F$-cohomological dimension $\leq d$ for some fixed $d \in \N$. The main source of examples is:

\begin{example}
If $X$ is a variety over a separably closed field $k$ and $F$ is torsion, then it satisfies the above assumption. Indeed, Artin proved that $H^i(U_\et,F) = 0 $ for $i > \dim(U)$ if $U$ is an affine $k$-variety.
\end{example}

Recall that $K \in D(X_\et,F)$ is compact if $\R\Hom(K,-)$ commutes with arbitrary direct sums. Let $D_c(X_\et,F) \subset D(X_\et,F)$ be the full subcategory of compact objects. Our goal is to identify $D_c(X_\et,F)$ with the category of constructible complexes.  We start by recording a completeness property of $D(X_\et,F)$:

\begin{lemma}
	For any qcqs $U \in X_\et$, the functor $\R\Gamma(U_\et,-)$ has finite $F$-cohomological dimension.
\end{lemma}
\begin{proof}
Assume first that $U = V_1 \cup V_2$ with $V_i \subset U$ open affines, and $W := V_1 \cap V_2$ affine. Then one has an exact triangle
\[ \R\Gamma(U_\et,-) \to \R\Gamma(V_{1,\et},-) \oplus \R\Gamma(V_{2,\et},-) \to \R\Gamma(W_\et,-) \]
which gives the desired finiteness. The general case is handled by induction using a similar argument, by passing through the separated case first.
\end{proof}

\begin{lemma}
The category $D(X_\et,F)$ is left-complete.
\end{lemma}
\begin{proof}
	This follows from Proposition \ref{prop:leftcompletecrit}.
\end{proof}

\begin{lemma}
	\label{lem:shriekpushforwardcompact}
For any $j:U \to X$ in $X_\et$, the pushforward $j_!:D(U_\et,F) \to D(X_\et,F)$ preserves compact objects.
\end{lemma}
\begin{proof}
Formal by adjunction since $j^*$ preserves all direct sums.
\end{proof}

\begin{lemma}
	\label{lem:pushforwardcompactopen}
	For each qcqs $j:U \to X$ in $X_\et$, we have:
	\begin{enumerate}
		\item The object $j_! \underline{F} \in D(X_\et,F)$ is compact.
		\item The functor $j_*:  D(U_\et,F) \to D(X_\et,F)$ commutes with all direct sums.
	\end{enumerate}
\end{lemma}
\begin{proof}
	For (1), by Lemma \ref{lem:shriekpushforwardcompact}, we may assume $j = \id$, so we want $\R\Gamma(X,-)$ to preserve all direct sums.  We first observe that the finiteness assumption on $X$ and the corresponding left-completeness of $D(X_\et,F)$ give: for any $K \in D(X_\et,F)$, one has $H^i(X,K) \simeq H^i(X,\tau^{\geq -n} K)$ for $n > N_X -i$, where $N_X$ is the $F$-cohomological dimension of $X$. One then immediately reduces to the bounded below case, which is true for any qcqs scheme. For (2), fix some qcqs $V \in X_\et$, and let $W = U \times_X V$.  Then (1) shows that $\R\Gamma(V_\et,-)$ commutes with direct sums. Hence, given any set $\{K_s\}$ of objects in $D(U_\et,F)$, we have
\[ \R\Gamma(V_\et, \oplus_s j_* K_s) \simeq \oplus_s \R\Gamma(V_\et, j_* K_s) \simeq \oplus_s \R\Gamma( W_\et, K_s|_W) \simeq \R\Gamma(W_\et, (\oplus_s K_s)|_W) \simeq \R\Gamma(V_\et, j_* \oplus_s K_s).\]
As this is true for all $V$, the claim follows.
\end{proof}

\begin{lemma}
	\label{lem:pushforwardcompactclosed}
Fix a closed constructible subset $i:Z \hookrightarrow X$ and $K \in D(Z_\et,F)$ that is locally constant with perfect value $L \in D_\perf(F)$. Then $i_* K \in D(X_\et,F)$ is compact.
\end{lemma}
\begin{proof}
	By Lemma \ref{lem:pushforwardcompactopen} (2), it suffices to show the following statement: the functor $\underline{\R\Hom}(i_* K,-):D(X_\et,F) \to D(X_\et,F)$ commutes with direct sums. To check this, we may freely replace $X$ with an \'etale cover. By passing to a suitable cover (see the proof of Lemma \ref{lem:constructiblealmostcompact}), we may assume $K = \underline{L}$ for $L \in D_\perf(F)$. If $j:U \to X$ denotes the qcqs open complement of $i$, then the exact triangle
	\[ j_! \underline{L} \to \underline{L} \to i_* \underline{L} \]
	finishes the proof by Lemma \ref{lem:pushforwardcompactopen} (1)
\end{proof}

\begin{remark}
	The constructibility of $Z$ in Lemma \ref{lem:pushforwardcompactclosed} is necessary. For a counterexample without this hypothesis, choose an infinite profinite set $S$ and a closed point $i:\{s\} \hookrightarrow S$. Then $S - \{s\}$ is not quasi-compact, so $Z$ is not constructible. Using stalks, one checks that $i_* \underline{F} \simeq \colim j_* \underline{F}$, where the colimit is indexed by clopen neighbourhoods $j:U \hookrightarrow S$ of $s \in S$. For such $j$, one has $H^0(S,j_* \underline{F}) = H^0(U,\underline{F}) = \Map_{\mathrm{conts}}(U,F)$. As any continuous map $f:U \to F$ is locally constant, each non-zero section of $H^0(S,j_* \underline{F})$ is supported on some clopen $V \subset U$. As $1 \in H^0(S,i_* \underline{F})$ is supported only at $s$, all maps $i_* \underline{F} \to j_* \underline{F}$ are constant, so $i_* \underline{F}$ is not compact in $D(S,F)$. To get an example with schemes, one simply tensors this example with a geometric point, in the sense of Example \ref{ex:tensorprofiniteset}.
\end{remark}

\begin{proposition}
	\label{lem:constructiblecompact}
$D(X_\et,F)$ is compactly generated, and $D_c(X_\et,F) = D_\cons(X_\et,F)$.
\end{proposition}
\begin{proof}
	We temporarily use the word ``coherent'' to refer to objects of the form $j_! \underline{F}$ for qcqs maps $j:U \to X$ in $X_\et$. Lemma \ref{lem:pushforwardcompactopen} shows that coherent objects are compact. General topos theory shows that all objects in $D(X_\et,F)$ can be represented by complexes whose terms are direct sums of coherent objects, so it follows that $D(X_\et,F)$ is compactly generated. Furthermore, one formally checks that the subcategory $D_c(X_\et,F) \subset D(X_\et,F)$ of compact objects is the smallest idempotent complete triangulated subcategory that contains the coherent objects.  Then Lemma \ref{lem:shriekpushforwardcons} shows $D_c(X_\et,F) \subset D_\cons(X_\et,F)$. For the reverse inclusion $D_\cons(X_\et,F) \subset D_c(X_\et,F)$, it suffices to show: for any $k:W \hookrightarrow X$ locally closed constructible and $L \in D(W_\et,F)$ locally constant with perfect values, the pushforward $K := k_! L$ is compact.  Choose $W \stackrel{f}{\to} U \stackrel{g}{\to} X$ with $f$ a constructible closed immersion, and $g$ a qcqs open immersion. Then $f_* K$ is compact in $D(U_\et,F)$ by Lemma \ref{lem:pushforwardcompactclosed}, so $k_! K \simeq g_! f_* K$ is compact by Lemma \ref{lem:shriekpushforwardcompact}.
\end{proof}

\subsection{Constructible complexes in the pro-\'etale topology}
\label{ss:ConsProet}

Fix a qcqs scheme $X$, and a noetherian ring $R$ complete for the topology defined by an ideal $\fram \subset R$. Set $\widehat{R}_X := \lim R/\fram^n \in \Shv(X_\proet)$; we often simply write $\widehat{R}$ for $\widehat{R}_X$. In fact, in the notation of Lemma \ref{l:ExSheafTopSpace}, $\widehat{R} = \widehat{R}_X$ is the sheaf $\mathcal{F}_R$ on $X_\proet$ associated with the topological ring $R$. We write $\underline{L}$ for the image of $L \in D(R)$ under the pullback $D(R) \to D(X_\proet,R)$, and $\widehat{\underline{L}} \in D(X_\proet,\widehat{R})$ for the $\fram$-adic completion of $\underline{L}$. When $L = R$ or $R/\fram^n$, we drop the underline. The key definition is:

\begin{definition} 
We say that $K \in D(X_\proet,\widehat{R})$ is {\em constructible} if $K$ is $\fram$-adically complete, and $K \otimes_{\widehat{R}}^L R/\fram$ is obtained via pullback of a constructible $R/\fram$-complex under $\nu:X_\proet \to X_\et$. Write
\[D_{\cons}(X_\proet,\widehat{R})  \subset D(X_\proet,\widehat{R}) \]
for the full subcategory spanned by constructible complexes.
\end{definition}

It is immediate that $D_\cons(X_\proet,\widehat{R})$ is a triangulated subcategory of $D(X_\proet,\widehat{R})$.  Applying the same definition to $(R/\fram^n,\fram)$, we get $D_\cons(X_\proet,R/\fram^n) \simeq D_\cons(X_\et,R/\fram^n)$ via $\nu$; note that the two evident definitions of $D_\cons(X_\et,R/\fram^n)$ coincide by Lemma \ref{lem:nilpidealcons}.

\begin{example}
	When $X$ is a geometric point, pullback induces an equivalence $D_\perf(R) \simeq D_\cons(X_\proet,\widehat{R})$. 
\end{example}

\begin{lemma}
Each $K \in D_\cons(X_\proet,\widehat{R})$ is bounded.
\end{lemma}
\begin{proof} Completeness gives $K \simeq \R\lim(K \otimes^L_R R/\fram^n)$. As $\R\lim$ has cohomological dimension $\leq 1$ by repleteness, it suffices to show  $K_n := K \otimes^L_R R/\fram^n$ has amplitude bounded independent of $n$.  This follows from standard sequences as $K_1$ has finite flat dimension.
\end{proof}

\begin{lemma}
If $K \in D_\cons(X_\proet,\widehat{R})$, then $K \otimes_{\widehat{R}} R/\fram^n \in D_\cons(X_\proet,R/\fram^n)$ for each $n$.
\end{lemma}
\begin{proof}
This is immediate from $K \otimes_{\widehat{R}} R/\fram^n \otimes_{R/\fram^n} R/\fram \simeq K \otimes_{\widehat{R}} R/\fram$.
\end{proof}

\begin{lemma}
	\label{lem:constensorproduct}
	$D_\cons(X_\proet,\widehat{R}) \subset D_\comp(X_\proet,\widehat{R})$ is closed under tensor products. In fact, if $K,L \in D_\cons(X_\proet,\widehat{R})$, then $K \otimes_{\widehat{R}} L$ is already complete.
\end{lemma}
\begin{proof}
	The assertion is local on $X_\proet$. By filtering $K$ and $L$, and replacing $X$ by a cover, we may assume: $X$ is w-contractible and henselian along a constructible closed subset $i:Z \hookrightarrow X$, and $K = i_* \widehat{\underline{M}}$ and $L = i_* \widehat{\underline{N}}$ for $M,N \in D_\perf(R)$. By realising $M$ and $N$ as direct summands of finite free $R$-complexes, we reduce to $M = N = R$. Let $j:U \to X$ be the open complement of $i$. We claim the more precise statement that $i_* \widehat{R} \otimes_{\widehat{R}} i_* \widehat{R} \simeq i_* \widehat{R}$. For this, using the sequence
	\[ j_! \widehat{R} \to \widehat{R} \to i_* \widehat{R},\]
	we are reduced to checking that $j_! \widehat{R} \otimes_{\widehat{R}} i_* \widehat{R} = 0$, which is automatic by adjunction: for any $K \in D(U_\proet,\widehat{R})$ and $L \in D(Z_\proet,\widehat{R})$, one has 
	\[ \R\Hom(j_! K \otimes_{\widehat{R}} i_* L, -) = \R\Hom(j_! K, \underline{\R\Hom}(i_* L, -)) = \R\Hom(K, \underline{\R\Hom}(j^* i_* L, j^*(-))) = 0,\]
	where the last equality uses $j^* i_* = 0$.
\end{proof}

\begin{lemma}
	\label{lem:liftconstancy}
	Fix $K \in D_\cons(X_\proet,\widehat{R})$ with $K \otimes_{\widehat{R}} R/\fram$ constant locally on $X_\et$. Then $K \otimes_{\widehat{R}} R/\fram^n$ is also constant locally on $X_\et$.
\end{lemma}
\begin{proof}
	Since the question concerns only complexes pulled back from $X_\et$, we can \'etale localize to assume that $(X,x)$ is a local strictly henselian scheme. Then the assumption implies $K \otimes_{\widehat{R}} R/\fram$ is constant. Moreover, one easily checks that $D(R/\fram^n) \to D(X_\et,R/\fram^n)$ is fully faithful (as $\R\Gamma(X_\et,-) \simeq x^*$). Chasing triangles shows that each $K \otimes_{\widehat{R}} R/\fram^n$ is in the essential image of $D(R/\fram^n) \to D(X_\et,R/\fram^n)$, as wanted.
\end{proof}

\begin{corollary}
	\label{cor:constanthenselian}
	Assume $X$ is a strictly henselian local scheme. Then pullback
	\[D_\perf(R) \to D_\cons(X_\proet,\widehat{R}) \]
	is fully faithful with essential image those $K$ with $K \otimes_{\widehat{R}} R/\fram$ locally constant.
\end{corollary}
\begin{proof}
	The full faithfulness is automatic since $\R\Gamma(X,\widehat{R}) \simeq \R\lim \R\Gamma(X,R/\fram^n) \simeq \R\lim R/\fram^n \simeq R$. The rest follows by Lemma \ref{lem:liftconstancy}.
\end{proof}

\begin{lemma}
	\label{lem:constrestrictopen}
Fix a locally closed constructible subset $k:W \hookrightarrow X$.
\begin{enumerate}
	\item One has $k^*(\widehat{R}_X) = \widehat{R}_W$.
	\item The functor $k^*:D(X_\proet,\widehat{R}_X) \to D(W_\proet,\widehat{R}_W)$ preserves constructible complexes.
	\item The functor $k_!:D(W_\proet,\widehat{R}_W) \to D(X_\proet,\widehat{R}_X)$ preserves constructible complexes.
\end{enumerate}
\end{lemma}
\begin{proof}
	(1)  follows from the fact that $k^*:\Shv(X_\proet) \to \Shv(W_\proet)$ commutes with limits (as this is true for constructible open and closed immersions).  This also implies $k^*(K \otimes_{\widehat{R}_X} R/\fram) \simeq k^* K \otimes_{\widehat{R}_W} R/\fram$ for any $K \in D(X_\proet,\widehat{R}_X)$, which gives (2). The projection formula for $k_!$ shows $k_! K \otimes_{\widehat{R}_X} R/\fram \simeq k_!(K \otimes_{\widehat{R}_W} R/\fram)$, which gives (3).
\end{proof}

\begin{lemma}
	\label{lem:conspullback}
	Let $f:X \to Y$ be a map of qcqs schemes, and let $f_*:D(X_\proet,\widehat{R}) \to D(Y_\proet,\widehat{R})$ be the pushforward. Then we have:
\begin{enumerate}
	\item For $K \in D(X_\proet,\widehat{R})$, we have an identification $\{f_*K \otimes_{\widehat{R}} R/\fram^n\} \simeq \{f_*(K \otimes_{\widehat{R}} R/\fram^n) \}$ of pro-objects. 
	\item For $K \in D(X_\proet,\widehat{R})$, we have $f_* \widehat{K} \simeq \widehat{f_*K}$. In particular, $f_*$ preserves $\fram$-adically complete complexes, and hence induces $f_*:D_\comp(X_\proet,\widehat{R}) \to D_\comp(Y_\proet,\widehat{R})$.
	\item For any perfect complex $L \in D(R)$, we have $f_* K \otimes_{\widehat{R}}  \widehat{\underline{L}} \simeq f_*(K \otimes_{\widehat{R}} \widehat{\underline{L}})$.
	\item Pullback followed by completion gives $f_\comp^*:D_\comp(X_\proet,\widehat{R}) \to D_\comp(Y_\proet,\widehat{R})$ left adjoint to $f_{*}$.
	\item $f_{\comp}^*$ preserves constructible complexes, and hence defines
	\[f_\comp^*:D_\cons(Y_\proet,\widehat{R}) \to D_\cons(X_\proet,\widehat{R})\ .\]
\end{enumerate}
\end{lemma}
\begin{proof}
	(1) would be clear if each $R/\fram^n$ is $R$-perfect. To get around this, choose $P$ and $J$ as in the proof of Proposition \ref{prop:derivedcompnoetherian}. Then $\{R \otimes_P P/J^n\} \simeq \{R/\fram^n\}$ is a strict pro-isomorphism, so $\{K \otimes_R R/\fram^n\} \simeq \{K \otimes_P  P/J^n\}$ as pro-objects as well, and similarly for $f_* K$. The claim now follows as $P/J^n$ is $P$-perfect. (2) immediately follows from (1) (or simply because $T(f_*K,x) \simeq f_* T(K,x) \simeq 0$ for $x \in \fram$ and $K$ is complete as $f_*$ commutes with $\R\lim$). (3) immediately follows from the case $L = R$ by devissage,  while (4) follows from (2) by adjointness of completion. For (5), as $f^*$ commutes with tensor products, we have $f_\comp^*(K) \otimes_{\widehat{R}_Y} R/\fram \simeq f^*(K \otimes_{\widehat{R}_X} R/\fram)$, so the claim follows from preservation of constructibility under pullbacks in the classical sense.
\end{proof}

\begin{remark}
	When $f:X \to Y$ is a finite composition of qcqs weakly \'etale maps and constructible closed immersion, we have $f_\comp^* = f^*$, i.e., that $f^* K$ is complete if $K$ is so; this follows from Lemma \ref{lem:constrestrictopen}. 
\end{remark}

Lemma \ref{lem:conspullback} shows that pushforwards in the pro-\'etale topology commute with taking $\fram$-adic truncations in the sense of pro-objects. To get strict commutation, we need a further assumption:

\begin{lemma}
	\label{lem:abstractpushforward}
	Let $f:X \to Y$ be a map of qcqs schemes. Assume that $f_*:\Mod(X_\et,R/\fram) \to \Mod(Y_\et,R/\fram)$ has cohomological dimension $\leq d$ for some integer $d$. Then:
	\begin{enumerate}
		\item If $P \in D^{\leq k}(R)$ and $K \in D^{\leq m}_\cons(X_\proet,\widehat{R})$, then $f_*(K \widehat{\otimes}_{\widehat{R}} \widehat{\underline{P}}) \in D^{\leq k+m+d+2}(Y_\proet,\widehat{R})$.
		\item If $K \in D_\cons(X_\proet,\widehat{R})$ and $M \in D^{-}(R)$, then $f_*(K \widehat{\otimes}_{\widehat{R}} \widehat{\underline{M}}) \simeq f_* K \widehat{\otimes}_{\widehat{R}} \widehat{\underline{M}}$.
		\item If $K \in D_\cons(X_\proet,\widehat{R})$,  then $f_* K \otimes_{\widehat{R}} R/\fram^n \simeq f_* (K \otimes_{\widehat{R}} R/\fram^n)$ for all $n$.
	\end{enumerate}
\end{lemma}
\begin{proof}
	For (1), observe that 
	\[ f_*(K \widehat{\otimes}_{\widehat{R}} \widehat{\underline{P}}) \simeq f_* \R\lim(K_n \otimes_{R/\fram^n} \underline{P_n}) \simeq \R\lim f_*(K_n \otimes_{R/\fram^n} \underline{P_n}) \in D^{\leq k+m+d+2}(Y_\proet,\widehat{R}),\] 
	where the last inclusion follows from Lemma \ref{lem:cohdimetproet} and repleteness. For (2), we may assume by shifting that $K \in D^{\leq 0}_\cons(X_\proet,\widehat{R})$. First observe that if $M$ is a free $R$-module, then the claim is clear. For general $M$, fix an integer $i$ and choose an $i$-close approximation $P_i \to M$  in $D(R)$ with $P_i$ a finite complex of free $R$-modules, i.e., the homotopy-kernel $L_i$ lies in $D^{\leq -i}(R)$. Then $\widehat{\underline{P_i}} \to \widehat{\underline{M}}$ is an $i$-close approximation in $D(X_\proet,\widehat{R})$. Moreover, $f_*(K \widehat{\otimes}_{\widehat{R}} \widehat{\underline{P_i}}) \simeq f_* K \widehat{\otimes}_{\widehat{R}} \widehat{\underline{P_i}}$ as $\widehat{\underline{P_i}}$ is a finite complex of free $\widehat{R}$-modules. We then get a commutative diagram
	\[ \xymatrix{ f_* K \widehat{\otimes}_{\widehat{R}} \widehat{\underline{P_i}} \ar[r]^-a \ar[d]^-b & f_* K \widehat{\otimes}_{\widehat{R}} \widehat{\underline{M}} \ar[d]^-c \\
	f_* (K \widehat{\otimes}_{\widehat{R}} \widehat{\underline{P_i}}) \ar[r]^-d & f_*(K \widehat{\otimes}_{\widehat{R}} \widehat{\underline{M}} ). } \]
	Then $b$ is an equivalence as explained above. The homotopy-kernel $f_*(K \widehat{\otimes}_{\widehat{R}} \widehat{\underline{L_i}})$ of $d$ is $(-i+d+2)$-connected by (1), and the homotopy-kernel $f_* K \widehat{\otimes}_{\widehat{R}} \widehat{\underline{L_i}}$ of $a$ is $(-i+d+2)$-connected since $f_*K \simeq \R\lim f_* K_n \in D^{\leq d+1}(Y_\proet)$. Thus, the homotopy-kernel of $c$ is also $(-i+d+2)$-connected. Letting $i \to \infty$ shows $c$ is an isomorphism.  (3) follows from (2) by setting $M = R/\fram^n$, observing that $R/\fram^n$ is already derived $\fram$-complete, and using  $- \widehat{\otimes}_{\widehat{R}} R/\fram \simeq - \otimes_{\widehat{R}} R/\fram$ as any $R/\fram$-complex is automatically derived $\fram$-complete.
\end{proof}

\begin{remark}
Unlike pullbacks, the pushforward along a map of qcqs schemes does not preserve constructibility: if it did, then $H^0(X,\Z/2)$ would be finite dimensional for any qcqs scheme $X$ over an algebraically closed field $k$, which is false for $X = \Spec(\prod_{i=1}^\infty k)$. We will see later that there is no finite type counterexample.
\end{remark}

\subsection{Constructible complexes on noetherian schemes}
\label{ss:ConsNoeth}

Fix $X$ and $R$ as in \S \ref{ss:ConsProet}. Our goal in this section is to prove that the notion of a constructible complexes on $X$ coincides with the classical one from topology if $X$ is noetherian: $K \in D(X_\proet,\widehat{R})$ is constructible if and only if it is locally constant along a stratification, see Proposition \ref{prop:conscons}. In fact, it will be enough to assume that the topological space underlying $X$ is noetherian. The proof uses the notion of w-strictly local spaces, though a direct proof can be given for varieties, see Remark \ref{rmk:consconsvariety}.

For any affine scheme $Y$, there is a natural morphism $\pi:Y_\et \to \pi_0(Y)$ of sites. Our first observation is that $\pi$ is relatively contractible when $Y$ is w-strictly local.

\begin{lemma}
	\label{lem:dercatwstrictlylocal}
If $Y$ is a w-strictly local affine scheme, then pullback $D(\pi_0(Y)) \to D(Y_\et)$ is fully faithful.
\end{lemma}
\begin{proof}
	Fix $K \in D(\pi_0(Y),F)$. Choose a point $y \in \pi_0(Y)$, and let $\overline{y} \in Y$ be its unique preimage that is closed. Then the projective system $\{\pi^{-1} U\}$ of open neighbourhoods of $\overline{y}$ obtained via pullback of open neighbourhoods $y \in U$ in $\pi_0(Y)$ is cofinal in the projective system $\{V\}$ of all open neighbourhoods $\overline{y} \in V$ in $Y$. Hence, 
	\[ \colim_{y \in U} \R\Gamma(U,\pi_* \pi^* K) \simeq \colim_{y \in U} \R\Gamma(\pi^{-1} U, \pi^* K) \simeq \colim_{\overline{y} \in V} \R\Gamma(V,\pi^* K) \simeq (\pi^* K)_{\overline{y}} \simeq K_y.\]
	Here the penultimate isomorphism uses that the Zariski and \'etale localizations of $Y$ at $\overline{y}$ coincide. This shows that $K \to \pi_* \pi^* K$ induces an isomorphism on stalks, so must be an isomorphism. The rest follows by adjunction.
\end{proof}

For a profinite set $S$, we define $S_\proet := \underline{S}_\proet$, with $\ast$ some fixed geometric point, and $\underline{S} \in \Shv(\ast_\proet)$ the corresponding scheme. Alternatively, it is the site defined by profinite sets over $S$ with covers determined by finite families of continuous and jointly surjective maps, see Example \ref{ex:proetalegeometricpoint}. Using repleteness of $\Shv(S_\proet)$, we show that a compatible system of constant perfect $R/\fram^n$-complexes $L_n$ on $S$ has a constant perfect limit $L$ in $S_\proet$; the non-trivial point is that we do not {\em a priori} require the transition maps be compatible with trivializations.

\begin{lemma}
	\label{lem:profinitesetconstancycrit}
	Let $S$ be a profinite set. Fix $L \in D_\comp(S_\proet,\widehat{R})$ with $L \otimes_R R/\fram^n$ constant with perfect value $C_n \in D(R/\fram^n)$ for all $n$. Then $L$ is constant with perfect values.
\end{lemma}
\begin{proof}
	Fix a point $s \in S$. Passing to the stalks at $s$ shows that there exists $C \in D_\perf(R)$ with $C \otimes_R R/\fram^n \simeq C_n$. Write $\widehat{\underline{C}} \in D(S_\proet,\widehat{R})$ and $\underline{C_n} \in D(S_\proet,R/\fram^n)$ for the corresponding constant complexes. We will show $\Isom_{\widehat{R}}(L,\widehat{\underline{C}}) \neq \emptyset$. First observe that $\Ext^i_{R/\fram^n}(\underline{C_n},\underline{C_n}) \simeq \Map_{\mathrm{conts}}(S,\Ext^i_{R/\fram^n}(C_n,C_n))$. By Lemma \ref{lem:mlmaps} and Lemma \ref{lem:compatautml}, the system $\{\Ext^i_{R/\fram^n}(\underline{C_n},\underline{C_n})\}$ satisfies ML. As a map $f:\underline{C_n} \to \underline{C_n}$ is an automorphism if and only if it is so modulo $\fram$, it follows that $\{\Aut_{R/\fram^n}(\underline{C_n})\}$ also satisfies ML. Lemma \ref{lem:mltorsor} and the assumption on $L_n$ shows that $\{\Isom_{R/\fram^n}(L_n,\underline{C_n})\}$ satisfies ML. As the evident map $\Isom_{R/\fram^n}(L_n,\underline{C_n}) \times \Ext^i_{R/\fram^n}(\underline{C_n},\underline{C_n}) \to \Ext^i_{R/\fram^n}(L_n,\underline{C_n})$ is surjective, Lemma \ref{lem:mlsurjective} shows that $\{\Ext^i_{R/\fram^n}(L_n,\underline{C_n})\}$ satisfies ML. On the other hand, completeness gives
	\[\R\Hom_{\widehat{R}}(L,\widehat{\underline{C}}) \simeq \R\lim \R\Hom_{R/\fram^n}(L_n,\underline{C_n}),\]
	so 
	\[ \Hom_{\widehat{R}}(L,\widehat{\underline{C}}) \simeq \lim_n \Hom_{R/\fram^n}(L_n,\underline{C_n}).\]
	By completeness, a map $f:L \to \widehat{\underline{C}}$ is an isomorphism if and only $f \otimes_{\widehat{R}} R/\fram$ is one, so $\Isom_{\widehat{R}}(L,\widehat{\underline{C}}) \simeq \lim_n \Isom_{R/\fram^n}(L_n,\underline{C_n})$. As $\{\Isom_{R/\fram^n}(L_n,\underline{C_n})\}$ satisfies ML with non-empty terms, the limit is non-empty. 
\end{proof}

The next few lemmas record elementary facts about projective systems $\{X_n\}$ of sets; for such a system, we write $X_n^\circ := \cap_k \im(X_{n+k} \to X_n) \subset X_n$ for the stable image.

\begin{lemma}
	\label{lem:mlmaps}
	Fix a topological space $S$ and a projective system $\{X_n\}$ of sets satisfying the ML condition. Then $\{\Map_{\mathrm{conts}}(S,X_n)\}$ also satisfies the ML condition.
\end{lemma}
\begin{proof}
	Fix $n$ and $N$ such that $X_n^\circ = \im(X_N \to X_n)$. Fix a continuous map $f:S \to X_n$ that lifts to $X_N$. Then $f$ factors through a continuous map $S \to X_n^\circ$. As $\{X_n^\circ\}$ has surjective transition maps, the claim follows.
\end{proof}

\begin{lemma}
	\label{lem:mltorsor}
	Let $\{G_n\}$ be a projective system of groups, and let $\{X_n\}$ be a compatible projective system of transitive $G$-sets. Assume  $\{G_n\}$ satisfies ML and $X_n \neq \emptyset$ for all $n$. Then $\{X_n\}$ satisfies ML, and $\lim X_n \neq \emptyset$.
\end{lemma}
\begin{proof}
	Note that any $\N^\opp$-indexed system of non-empty sets satisfying the ML condition has a non-empty inverse limit: the associated stable system has non-empty terms and surjective transition maps. Hence, it suffices to show $\{X_n\}$ satisfies ML. Write $h_{ij}:G_i \to G_j$ and $f_{ij}:X_i \to X_j$ for the transition maps. Fix $n$ and $N$ such that $G_n^\circ = \im(G_N \to G_n)$. Fix some $x_n \in X_n$ that lifts to an $x_N \in X_N$. For $m \geq N$, choose some $x_m \in X_m$, and $g_N \in G_N$ with $g_N \cdot f_{mN}(x_m) = x_N$; this is possible by transitivity. Then there exists a $g_m \in G_m$ with $h_{mn}(g_m) = h_{Nn}(g_n)$, so  $x_m := g_m^{-1} \cdot x_m \in X_m$ lifts $x_n \in X_n$, which proves the ML property.
\end{proof}

\begin{lemma}
	\label{lem:mlsurjective}
	Let $f:\{W_n\} \to \{Y_n\}$ be a map of projective systems. Assume that $\{W_n\}$ satisfies ML, and that $f_n:W_n \to Y_n$ is surjective. Then $\{Y_n\}$ satisfies ML.
\end{lemma}
\begin{proof}
	Fix $n$, and choose $N$ such that $W_n^\circ = \im(W_N \to W_n)$.  Then any $y_n \in Y_n$ that lifts to some $y_N \in Y_N$ further lifts to some $w_N \in W_N$ with image $w_n \in W_n$ lifting $y_n$. By choice of $N$, there is a $w_{n+k} \in W_{n+k}$ for all $k$ lifting $w_n \in W_n$. The images $y_{n+k} := f_{n+k}(w_{n+k}) \in Y_{n+k}$ then lift $y_n \in Y_n$ for all $k$, which proves the claim.
\end{proof}

A version of the Artin-Rees lemma shows:

\begin{lemma}
	\label{lem:compatautml}
	For $K \in D_\perf(R)$, the natural map gives pro-isomorphisms $\{H^i(K)/\fram^n\} \simeq \{H^i(K \otimes_R R/\fram^n)\}$. 
\end{lemma}
\begin{proof}
	Let $\calC$ be the category of pro-($R$-modules), and consider the functor $F:\Mod_R^f \to \calC$ given by $M \mapsto \{M/\fram^n M\}$. Then $F$ is exact by the Artin-Rees lemma, so for any finite complex $K$ of finitely generated $R$-modules, one has $F(H^i(K)) \simeq H^i(F(K))$. Applying this to a perfect $K$ then proves the claim.
\end{proof}

\begin{lemma}
	\label{lem:locallyconstantwstrictlylocal}
Let $Y$ be a w-strictly local affine scheme. Then any $M \in D(Y_\et)$  that is locally constant on $Y_\et$ is constant over a finite clopen cover, and hence comes from $D(\pi_0(Y))$ via pullback. 
\end{lemma}
\begin{proof}
	For the first part, we may assume that there exist finitely many qcqs \'etale maps $f_i:U_i \to Y$ with $f:\sqcup_i U_i \to Y$ surjective such that $f_i^* M \simeq \underline{A_i}$ for some $A_i \in D(\Ab)$. By w-strict locality, there is a section $s:Y \to \sqcup_i U_i$ of $f$. Then $\{V_i := s^{-1} U_i\}$ is a finite clopen cover of $Y$ with $M|_{V_i} \simeq \underline{A_i} \in D( V_{i,\et})$. Now any finite clopen cover of $Y$ is the pullback of a finite clopen cover of $\pi_0(Y)$, so the second part follows.
\end{proof}

\begin{lemma}
	\label{lem:conslocallyconstant}
	Let $X = \Spec(A)$ be connected. Fix $K \in D_\cons(X_\proet,\widehat{R})$ with $K \otimes_{\widehat{R}} R/\fram$ locally constant on $X_\et$ with perfect values. Then there exists a pro-\'etale cover $f:Y \to X$ with $f^* K \simeq \underline{C}$ with $C \in D_\perf(R)$.
\end{lemma}
\begin{proof}
	First observe that, by connectedness and examination of stalks in $X_\et$, each $K_n := K \otimes_{\widehat{R}} R/\fram^n$ is locally constant on $X_\et$ with the same perfect value $C_n$.  Now choose a pro-\'etale cover $f:Y \to X$ with $Y$ w-strictly local, and let $\pi:Y \to \pi_0(Y)$ be the natural map. Then Lemma \ref{lem:locallyconstantwstrictlylocal} and Lemma \ref{lem:dercatwstrictlylocal} show $f^* K_n \simeq \pi^* L_n \simeq \pi^* \underline{C_n}$, where $L_n := \pi_* f^* K_n \in D(\pi_0(Y),R/\fram^n)$, and the isomorphism $L_n \simeq \underline{C_n}$ is non-canonical. Lemma \ref{lem:dercatwstrictlylocal} shows that 
	\[ L_{n+1} \otimes_{R/\fram^{n+1}} R/\fram^n  \simeq \pi_* \pi^*(L_{n+1} \otimes_{R/\fram^{n+1}} R/\fram^n) \simeq \pi_* \Big(f^\ast K_{n+1} \otimes_{R/\fram^{n+1}} R/\fram^n\Big) \simeq \pi_* f^\ast K_n = L_n\]
	via the natural map $L_{n+1} \to L_n$. Applying Lemma \ref{lem:truncationofcomplete} to $\{L_n\}$ shows that $L := \pi_* K \simeq \R\lim L_n \in D(\pi_0(Y)_\proet,\widehat{R})$ satisfies $L \otimes_{\widehat{R}} R/\fram^n \simeq L_n$. Lemma \ref{lem:profinitesetconstancycrit} then shows $L \simeq \widehat{\underline{C}}  \in D(\pi_0(Y)_\proet,\widehat{R})$, where $C := \R\lim C_n \in D_\perf(R)$ is a stalk of $K$.
\end{proof}

To state our result, we need the following definition.

\begin{definition} A scheme $X$ is said to be topologically noetherian if its underlying topological space is noetherian, i.e. any descending sequence of closed subsets is eventually constant.
\end{definition}

\begin{lemma}\label{l:topnoeth} Let $T$ be a topological space.
\begin{enumerate}
\item If $T$ is noetherian, then $T$ is qcqs and has only finitely many connected components. Moreover, any locally closed subset of $T$ is constructible, and noetherian itself.
\item If $T$ admits a finite stratification with noetherian strata, then $T$ is noetherian.
\item Assume that $X$ is a topologically noetherian scheme, and $Y\to X$ \'etale. Then $Y$ is topologically noetherian.
\end{enumerate}
\end{lemma}

\begin{proof}\begin{enumerate}
\item Quasicompacity of $T$ is clear. Also, the property of being noetherian passes to closed subsets, as well as to open subsets. Thus, all open subsets are quasicompact; this implies that all locally closed subsets are constructible, and that $T$ is quasiseparated. Every connected component is an intersection of open and closed subsets; this intersection has to be eventually constant, so that every connected component is open and closed. By quasicompacity, there are only finitely many.
\item Under this assumption, any descending sequence of closed subsets becomes eventually constant on any stratum, and thus constant itself.
\item There is a stratification of $X$ over which $Y\to X$ is finite \'etale. By (2), we may assume that $Y\to X$ is finite \'etale. Any closed $Z\subset Y$ gives rise to a function $f_Z: X\to \N$, mapping any $x\in X$ to the cardinality of the fibre of $Z$ above a geometric point above $x$. As $Z\to X$ is finite, the function $f_Z$ is upper semicontinuous, i.e. for all $n$, $\{x\mid f_Z(x)\geq n\}\subset X$ is closed. Moreover, all $f_Z$ are bounded independently of $Z$ (by the degree of $Y\to X$). Given a descending sequence of $Z$'s, one gets a descending sequence of $f_Z$'s. Thus, for any $n$, $\{x\mid f_Z(x)\geq n\}$ forms a descending sequence of closed subsets of $X$, which becomes eventually constant. As there are only finitely many $n$ of interest, all these subsets are eventually constant. This implies that $f_Z$ is eventually constant, which shows that $Z$ is eventually constant, as desired.
\end{enumerate}
\end{proof}

Here is the promised result.

\begin{proposition}
	\label{prop:conscons}
	Let $X$ be a topologically noetherian scheme. A complex $K \in D(X_\proet,\widehat{R})$ is constructible if and only if there exists a finite stratification $\{X_i \hookrightarrow X\}$ with $K|_{X_i}$ locally constant with perfect values on $X_{i,\proet}$.
\end{proposition}

The phrase ``locally constant with perfect values'' means locally isomorphic to $\widehat{\underline{L}} \simeq \underline{L} \otimes_R \widehat{R}$ for some $L \in D_\perf(R)$.

\begin{proof}
	For the forward direction, fix $K \in D_\cons(X_\proet,\widehat{R})$. By noetherian induction, it suffices to find a dense open $U \subset X$ such that $K|_U$ is locally constant with perfect values in $D(U_\proet,\widehat{R})$. By assumption, there exists a $U \subset X$ such that $K|_U \otimes_{\widehat{R}} R/\fram \in D(U_\et,R/\fram)$ is locally constant with perfect values.  Any topologically noetherian scheme has only finitely many (clopen) connected components. Thus, by passing to connected components, we may assume $U$ is connected. Lemma \ref{lem:conslocallyconstant} then proves the claim. For the reverse, fix $K \in D(X_\proet,\widehat{R})$, and assume there exists a finite stratification $\{X_i \hookrightarrow X\}$ such that $K|_{X_i}$ is, locally on $X_{i,\proet}$, the constant $\widehat{R}$-complex associated to a perfect $R$-complex. Then $K$ is complete by Lemmas \ref{lem:constrestrictopen}  and standard sequences (as completeness is a pro-\'etale local property). For the rest, by similar reasoning, we may assume that $X$ is affine and there exists a pro-\'etale cover $f:Y \to X$ such that $K|_Y \simeq \widehat{\underline{L}}$ for a perfect $R$-complex $L$. Then $K_1$ is locally constant with perfect value $L_1$ on $X_\proet$. Lemma \ref{lem:conslocalproet} then shows that $K_1$ is \'etale locally constant with perfect value $L_1$, as wanted.
\end{proof}

The next example shows the necessity of the noetherian hypothesis in Proposition \ref{prop:conscons}:

\begin{example}
	\label{ex:notnaive}
	Fix an algebraically closed field $k$, a prime number $\ell$. Set $X_n = \underline{\Z/\ell^n}$, and $X = \lim X_n = \underline{\Z_\ell} \in \Spec(k)_\proet$ following the notation of Example \ref{ex:tensorprofiniteset}, so $X$ is qcqs. Consider the sheaf of rings $\widehat{R} = \lim \Z/\ell^n \in \Shv(\Spec(k)_\proet)$; $X$ represents $\widehat{R}$, but we ignore this. We will construct a complex $K \in D(X_\proet,\widehat{R})$ satisfying:
	\begin{enumerate}
		\item $K \otimes_{\widehat{R}}^L \Z/\ell$ is constant with perfect values over a finite clopen cover of $X$, so $K \in D_\cons(X_\proet,\widehat{R})$.
		\item $K$ is constant on the connected components of $X$ with perfect values.
		\item There does not exist a finite stratification $\{X_i \hookrightarrow X\}$ with $K|_{X_i}$ locally constant on $X_{i,\proet}$.
	\end{enumerate}
	For each $n$, let $K_n' \in D(X_{n,\proet},\Z/\ell^n)$ be the locally constant complex whose value over the connected component of $X_n$ determined by $\alpha \in \Z/\ell^n$ is $\Big(\Z/\ell^n \stackrel{\alpha}{\to} \Z/\ell^n\Big)$. Set $K_n \in D(X_\proet,\Z/\ell^n)$ to be its pullback to $X$. Then there is a coherent system of quasi-isomorphisms $K_{n+1} \otimes_{\Z/\ell^{n+1}}^L \Z/\ell^n \simeq K_n$. Patching along these isomorphisms gives a complex $K := \R\lim K_n \in D(X_\proet,\widehat{R})$ satisfying: for each map $f_\alpha:\Spec(k) \to X$ determined by an $\alpha \in \Z_\ell$, we have $f_\alpha^* K \simeq \Big(\Z_\ell \stackrel{\alpha}{\to} \Z_\ell\Big)$. As $X$ is totally disconnected, (2) is clear. Since $K \otimes_{\widehat{R}} \Z/\ell \simeq K_1$, one easily checks (1). Finally, as the stalks $f_\alpha^* K$ over $\alpha \in X(k)$ take on infinitely many disinct values, (3) follows.
\end{example}

\begin{remark}
	\label{rmk:consconsvariety}
	When $X$ is a variety over an algebraically closed field $k$, it is easy to give a direct proof that any $K \in D_\cons(X_\proet,\widehat{R})$ is locally constant along a stratification, together with an explicit description of the trivializing cover over each stratum. Indeed, as in Proposition \ref{prop:conscons}, it suffices to find a dense open $U \subset X$ such that $K|_U$ is locally constant in $D(U_\proet,\widehat{R})$. Replacing $X$ by a suitable open, we may assume (by Artin's theorem \cite[\S XI.3]{SGA4Tome3}) that:
\begin{enumerate}
	\item $X$ is smooth, affine, connected, and a $K(\pi,1)$, i.e., pullback along the canonical map $\Shv(X_\et) \to \Shv(X_{f\et})$ induces a fully faithful functor $D^+(X_{f\et},R/\fram^n) \to D^+(X_\et,R/\fram^n)$\footnote{By the Leray spectral sequence for $\Phi:(\Shv(X_\et),R/\fram^n) \to (\Shv(X_{f\et}),R/\fram^n)$ and devissage to reduce $n$, it suffices to check that  $H^i(Y_\et,R/\fram) \simeq H^i(Y_{f\et},R/\fram)$ for all $i$ and all $Y \in X_{f\et}$. By passage to suitable filtered colimits, we may assume $R/\fram = \F_\ell$ or $R/\fram = \Q$. If $R/\fram = \F_\ell$ with $\ell \in k^*$, then the equality is due to Artin. If $R/\fram = \F_p$ with $p$ zero in $k$, then the Artin-Schreier sequence and the affineness of $Y$ show that $\R \Phi_* \F_p \simeq \F_p$, which clearly suffices.  If $R/\fram = \Q$, then $H^i(Y_{f\et},\Q) = 0$ by a trace argument; the normality of $Y$ combined with examination at stalks shows that $\Q \simeq \R\eta_* \Q$, where $\eta:\Spec(K) \to Y$ is the finite disjoint union of generic points of $Y$, which proves the claim by reduction to Galois cohomology.}.
	\item	$\nu_* K_1$ is locally constant on $X_\et$, i.e., pulled back from $X_{f\et}$.
\end{enumerate}
The normalization of $X$ in the maximal unramified extension of its fraction field within a fixed separable closure gives a pro-(finite \'etale) cover $f:Y \to X$. We will show $f^*K$ is constant.   Note that $Y$ is affine, connected, normal, and all finitely presented locally constant sheaves of $R/\fram^n$-modules on  $Y_\et$ are constant by construction. In particular, each $\calH^i(K_n)$ is constant over $Y$. Moreover, since $X$ was a $K(\pi,1)$, we have $\R\Gamma(Y_\et,M) \simeq M$ for any $M \in \Mod_{R/\fram^n}$.  Then the left-completeness of $D(Y_\proet)$ formally shows $D(R/\fram^n) \to D(Y_\proet,R/\fram^n)$ is fully faithful. Induction on the amplitude of $K_n$ then shows $f^*K_n \simeq \underline{C_n}$ for $C_n := \R\Gamma(Y_\proet,K_n) \in D(R/\fram^n)$. As $K$ is constructible, each $C_n$ is perfect (since $C_n = x^* f^* K_n$ for any geometric point $x$ of $Y$), and $C_{n+1} \otimes_{R/\fram^{n+1}} R/\fram^n \simeq C_n$ via the natural map. Then $C := \R\lim C_n \in D(R)$ is perfect, and $f^* K \simeq \R\lim f^* K_n \simeq \R\lim \underline{C_n} =: \widehat{\underline{C}} \in D(Y_\proet,\widehat{R})$, which proves the claim.
\end{remark}

\subsection{The 6 functors}
\label{ss:6Fun}

We fix a complete noetherian local ring $(R,\fram)$ with finite residue field of characteristic $\ell$. We say that a scheme $X$ is $\ell$-coprime if $\ell$ is invertible on $X$.

\begin{theorem}[Grothendieck, Gabber]
	\label{thm:gabberfiniteness}
	Let $f:X \to Y$ be a finitely presented map of qcqs schemes. Assume either that $f$ is proper, or that $Y$ is quasi-excellent and $\ell$-coprime. Then $f_*:D(X_\et,R/\fram) \to D(Y_\et,R/\fram)$ has finite cohomological dimension and preserves constructibility.
\end{theorem}

\begin{lemma}[Pushforward]
	\label{lem:6funpushforward}
	Let $f:X \to Y$ be a finitely presented map of qcqs schemes. Assume either that $f$ is proper, or that $Y$ is quasi-excellent and $\ell$-coprime. Then $f_*:D_\comp(X_\proet,\widehat{R}) \to D_\comp(Y_\proet,\widehat{R})$ preserves constructibility. The induced functor $f_{\ast}:D_\cons(X_\proet,\widehat{R}) \to D_\cons(Y_\proet,\widehat{R})$ is right adjoint to $f_{\comp}^*$. 
\end{lemma}
\begin{proof}
	Fix $K \in D_\cons(X_\proet,\widehat{R})$. Then $f_*K$ is complete by Lemma \ref{lem:conspullback}. Lemma \ref{lem:abstractpushforward} shows $f_* K \otimes_{\widehat{R}} R/\fram \simeq f_*(K \otimes_{\widehat{R}} R/\fram)$, so constructibility follows Lemma \ref{lem:funcpushforward} and Theorem \ref{thm:gabberfiniteness}; the adjunction is automatic. 
\end{proof}

\begin{remark}
	The $\ell$-coprimality assumption in Lemma \ref{lem:6funpushforward} is {\em necessary}: $H^1(\A^1_{\overline{\F}_p},\F_p)$ is infinite dimensional.
\end{remark}

\begin{lemma}[Smooth base change]
	\label{lem:sbc}
Fix a cartesian square of $\ell$-coprime qcqs schemes
\[ \xymatrix{ X' \ar[r]^g \ar[d]^f & X \ar[d]^-f \\
	      Y' \ar[r]^g & Y } \]	
	      with $f$ qcqs and $g$ smooth. Then for any $K \in D_\cons(X_\proet,\widehat{R})$, the natural map induces an isomorphism
	      \[ g_\comp^* \circ f_{*} K \simeq f_{*} \circ g_{\comp}^* K \in D_\comp(Y'_\proet,\widehat{R}).\]
	      If $Y$ is quasi-excellent and $f$ finitely presented, the preceding equality takes place in $D_\cons(Y'_\proet,\widehat{R})$.
\end{lemma}
\begin{proof}
	Lemma \ref{lem:conspullback} shows that $\{f_* K \otimes_{\widehat{R}} R/\fram^n\} \simeq \{f_* (K \otimes_{\widehat{R}} R/\fram^n)\}$ as pro-objects. By the constructibility assumption on $K$, each $K \otimes_{\widehat{R}} R/\fram^n$ is the pullback under $\nu$ of a constructible complex in $D^b(X_\et,R/\fram^n)$, so $f_*(K \otimes_{\widehat{R}} R/\fram^n)$ is a pullback from $D^+(X_\et,R/\fram^n)$ by Lemma \ref{lem:funcpushforward}. The claim now follows by definition of $g_{\comp}^*$ and classical smooth base change (which applies to $D^+(X_\et,R/\fram^n)$).
\end{proof}

\begin{lemma}[Proper base change I]
	\label{lem:pbc1}
Fix a cartesian square of qcqs schemes
\[ \xymatrix{ X' \ar[r]^g \ar[d]^f & X \ar[d]^-f \\
	      Y' \ar[r]^g & Y } \]	
	      with $f$ proper. Then for any $K \in D_\cons(X_\proet,\widehat{R})$, the natural map induces an isomorphism
	      \[ g_\comp^* \circ f_{*} K \simeq f_{*} \circ g_{\comp}^* K \in D_\cons(Y'_\proet,\widehat{R}).\]
\end{lemma}
\begin{proof}
	This reduces to the corresponding assertion in \'etale cohomology as all functors in sight commute with application of $- \otimes_{\widehat{R}} R/\fram$ by Lemma \ref{lem:conspullback} and Lemma \ref{lem:abstractpushforward}.
\end{proof}

\begin{definition}
	\label{def:compactlysupportedcoh}
	Let $f:X \to Y$ be a separated finitely presented map of qcqs schemes. Then we define $f_{!}:D_\cons(X_\proet,\widehat{R}) \to D_\cons(Y_\proet,\widehat{R})$ as $\overline{f}_* \circ j_!$ where $X \stackrel{j}{\hookrightarrow} \overline{X} \stackrel{\overline{f}}{\to} Y$ be a factorization with $j$ an open immersion, and $\overline{f}$ proper. If $Y$ is a geometric point, we write $\R\Gamma_c(X_\proet,K) := \R\Gamma(Y_\proet, f_! K)$.
\end{definition}

\begin{lemma}
	Definition \ref{def:compactlysupportedcoh} is well-defined, i.e., $f_{!}$ is independent of choice of $j$ and preserves constructibility.
\end{lemma}
\begin{proof}
This follows by the same argument used in the classical case thanks to  Lemma \ref{lem:opencloseddevissagederived}.
\end{proof}

\begin{remark}
	\label{rmk:compactlysupportedcohnotderived}
	Both $j_!$ and $f_*$ are right adjoints at the level of abelian categories. However, the functor $f_!$ from Definition \ref{def:compactlysupportedcoh} is {\em not} the derived functor of the composition $f_* \circ j_!: \Ab(X_\proet) \to \Ab(Y_\proet)$, i.e., of $\calH^0(f_!)$. To see this, take $X \to Y$ to be $\A^1 \to \Spec(k)$ with $k$ algebraically closed. Then we choose $j:X \hookrightarrow \overline{X}$ to be $\A^1 \subset \P^1$. It suffices to check that the derived functors of $F \mapsto \Gamma(\overline{X},j_! F)$ fail to compute $\R\Gamma(Y,f_! F)$. Lemma \ref{lem:closeddevissage} shows $\Gamma(\overline{X},j_! F) = \ker(F(X) \to F(\widetilde{\eta}))$ where $\eta \to X$ is the generic point, and $\widetilde{\eta} \to \eta \to X$ is the restriction of the henselization at $\infty$ on $\P^1$ to $\A^1$. The map $\widetilde{\eta} \to \eta$ is a pro-\'etale cover, so we can write $\Gamma(\overline{X},j_! F) = \ker(F(X) \to F(\eta))$ for any $F \in \Ab(X_\proet)$. As $\eta \to X$ is a subobject in $X_\proet$, the map $F(X) \to F(\eta)$ is surjective for $F$ injective. The derived functors of $F \mapsto \Gamma(\overline{X},j_! F)$ are thus computed by the homotopy-kernel of the map 
	\[ \R\Gamma(X,F) \to \R\Gamma(\eta,F).\]
	Taking $F = \Z/n$ for $n \in k^*$ shows $H^0(Y_\proet,R^2 \calH^0(f_!) F) \simeq H^1(\eta,\Z/n) \neq H^2_c(\A^1,\Z/n)$.
\end{remark}

\begin{remark}
	The phenomenon of Remark \ref{rmk:compactlysupportedcohnotderived} also occurs in classical \'etale cohomology, i.e., $f_!$ does not compute the derived functors of $\calH^0(f_!)$. However, the reason is different. In the example considered in Remark \ref{rmk:compactlysupportedcohnotderived}, if $X^0 \subset X$ is the set of closed points, then 
	\[ \Gamma(\overline{X},j_! F) = \oplus_{x \in X^0} \Gamma_{x}(X,F),\]
	for $F \in \Ab(X_\et)$ torsion; one checks this directly for constructible sheaves, and then observes that the constructible ones generate all torsion sheaves on $X_\et$ under filtered colimits.  The derived functors of $F \mapsto \Gamma(\overline{X},j_! F)$ are thus calculated by the homotopy-kernel of 
	\[ \oplus_{x \in X^0} \R\Gamma(X,F) \to \oplus_{x \in X^0} \R\Gamma(X-\{x\},F).\]
	Taking $F = \Z/n$ for $n \in k^*$ shows $H^0(Y_\et,R^2 \calH^0(f_!) F) \simeq \oplus_{x \in X^0} H^1(X-\{x\},\Z/n) \neq H^2_c(\A^1,\Z/n)$.
\end{remark}

\begin{lemma}[Proper base change II]
	\label{lem:pbc2}
	The conclusion of Lemma \ref{lem:pbc1} is valid for any separated finitely presented map $f$ provided $f_{*}$ is replaced by $f_{!}$.
\end{lemma}
\begin{proof}
	This follows from Lemma \ref{lem:pbc1} and Lemma \ref{lem:shriekpushforwardlocallyclosed}.
\end{proof}

\begin{lemma}
	\label{lem:shriekpullbackimmersionconst}
	Let $i:Z \hookrightarrow X$ be a constructible locally closed immersion with $X$ quasi-excellent and $\ell$-coprime. Then $i^!:D(X_\proet,\widehat{R}) \to D(Z_\proet,\widehat{R})$ preserves constructible complexes, and the resulting functor $i^!:D_{\cons}(X_\proet,\widehat{R}) \to D_{\cons}(Z_\proet,\widehat{R})$ is a right adjoint to $i_!:D_\cons(Z_\proet,\widehat{R}) \to D_\cons(X_\proet,\widehat{R})$.
\end{lemma}
\begin{proof}
	If $i$ is an open immersion, then $i^! = i^*$, so Lemma \ref{lem:constrestrictopen} settles the claim. Thus, we may assume $i$ is a closed immersion with open comelement $j:U \hookrightarrow X$. Fix $K \in D_\cons(X_\proet,\widehat{R})$. There is an exact triangle
	\[ i_* i^! K \to K \to j_* j^* K.\]
	Lemma \ref{lem:constrestrictopen} and Lemma \ref{lem:6funpushforward} imply that $j_* j^* K$ is constructible, and hence $i_* i^! K$ is also constructible. Another application of Lemma \ref{lem:constrestrictopen} shows that $i^! K = i^* i_* i^! K$ is also constructible.
\end{proof}

\begin{lemma}[$\otimes$-products]
	Let $X$ be a qcqs scheme. Then $D_\cons(X_\proet,\widehat{R}) \subset D(X_\proet,\widehat{R})$ is closed under $\otimes$-products. 
\end{lemma}
\begin{proof}
	This is Lemma \ref{lem:constensorproduct}.
\end{proof}

\begin{lemma}[Internal $\Hom$]
	\label{lem:inthom}
	Let $X$ be a quasi-excellent $\ell$-coprime scheme. If $K,L \in D_\cons(X_\proet,\widehat{R})$, then $\underline{\R\Hom}_R(K,L) \in D_\cons(X_\proet,\widehat{R})$. Moreover, for any $n \geq 0$, one has $\underline{\R\Hom}_R(K,L) \otimes_{\widehat{R}} R/\fram^n \simeq \underline{\R\Hom}_{R/\fram^n}(K \otimes_{\widehat{R}} R/\fram^n, L \otimes_{\widehat{R}} R/\fram^n)$.
\end{lemma}
\begin{proof}
	The assertion is local on $X$. By filtering $K$, we may assume $K = i_! \widehat{R}$ for $i:Z \hookrightarrow X$ a constructible closed immersions. By adjointness, we have $\underline{\R\Hom}(K,L) = i_* \underline{\R\Hom}(\widehat{R},  i^! L) \simeq i_*i^! L$, which is constructible by Lemma \ref{lem:shriekpullbackimmersionconst} and Lemma \ref{lem:6funpushforward}. The second assertion is proved similarly.
\end{proof}

\begin{lemma}[Projection Formula]
	\label{lem:projformula}	
Let $f:X \to Y$ be a separated finitely presented map of qcqs schemes. For any $L \in D_{\cons}(Y_\proet,\widehat{R})$ and $K \in D_\cons(X_\proet,\widehat{R})$, we have $f_! K \widehat{\otimes}_{\widehat{R}} L \simeq f_!(K \widehat{\otimes}_{\widehat{R}} f_\comp^* L)$ via the natural map.
\end{lemma}
\begin{proof}
The assertion is local on $Y$. By filtering $L$, we may assume $L = i_* \widehat{R}$ for $i:Z \hookrightarrow Y$ a constructible closed immersion. Let $j:U \hookrightarrow X$ be the open complement of $Z$. For any $\widehat{R}$-complex $L$, we have $L \otimes_{\widehat{R}} j_! j^* \widehat{R} \simeq j_! j^* L$, and hence $L \otimes_{\widehat{R}} i_* \widehat{R} \simeq i_* i^* L$.  Using this formula, the assertion now follows from Lemma \ref{lem:pbc2} as $i^* = i_\comp^*$.
\end{proof}

\begin{remark}
	The analogue of Lemma \ref{lem:projformula} for $f_*$ is false, even for quasiexcellent $\ell$-coprime schemes. Indeed, the projection formula for the special case $L = i_* \widehat{R}$ for $i:Z \hookrightarrow X$ is equivalent to the base change theorem as in Lemma \ref{lem:pbc2}, which fails for $f_*$.
\end{remark}

\begin{lemma}
	\label{lem:shriekprojformula}
	Let $f:X \to Y$ be a separated finitely presented map of qcqs schemes. For any $K \in D_\cons(X_\et,R/\fram^n)$ and $M \in D^b(R)$, we have $f_! K \otimes_{R/\fram^n} M \simeq f_!(K \otimes_{R/\fram^n} M) \in D^b(Y_\et,R/\fram^n)$.
\end{lemma}
\begin{proof}
	Lemma \ref{lem:abstractpushforward} (applied with $\widehat{R} = R/\fram^n$) proves the corresponding statement in the pro-\'etale world, i.e., after applying $\nu^*$. It remains to observe that both sides of the desired equality lie in $D^b(Y_\et, R/\fram^{n-1})$ by Lemma \ref{lem:6funpushforward} and the finite flat dimensionality of constructible complexes, so we can apply $\nu_*$ to get the claim.
\end{proof}

\begin{lemma}
	\label{lem:projformulaconstant}
	Let $f:X \to Y$ be a finitely presented map of quasi-excellent $\ell$-coprime schemes. For any $K \in D_\cons(X_\et,R/\fram^n)$ and $M \in D^{b}(R/\fram^n)$, we have $f_* K \otimes_{R/\fram^n} M \simeq f_*(K \otimes_{R/\fram^n} M) \in D^b(Y_\et,R/\fram^n)$.
\end{lemma}
\begin{proof}
	This is proven exactly like Lemma \ref{lem:shriekprojformula}.
\end{proof}

\begin{lemma}
	\label{lem:shriekpullbackclassical}
	Let $f:X \to Y$ be a separated finitely presented map of quasiexcellent $\ell$-coprime schemes. Then $f_!:D^+(X_\et,R/\fram^n) \to D^+(Y_\et,R/\fram^n)$ has a right adjoint $f_n^!$. This adjoint preserves constructibility, and the following two diagrams commute for $n \leq m$:
	\[ \xymatrix{ D^+(Y_\et,R/\fram^n) \ar[r] \ar[d]^{f_n^!} & D^+(Y_\et,R/\fram^m) \ar[d]^{f_m^!} & &  D_\cons(Y_\et,R/\fram^m) \ar[r] \ar[d]^{f_m^!} & D_\cons(Y_\et,R/\fram^n) \ar[d]^{f_n^!} \\
		D^+(X_\et,R/\fram^n) \ar[r] & D^+(X_\et,R/\fram^m) & & D_\cons(X_\et,R/\fram^m) \ar[r] & D_\cons(X_\et,R/\fram^n). }\]
		Here the horizontal maps are induced by restriction and extension of scalars along $R/\fram^m \to R/\fram^n$ respectively.
\end{lemma}
\begin{proof}
	The existence of $f_n^!$ and preservation of constructibility is classical. For the rest, we write $R_n = R/\fram^n$.  The commutativity of the square on the left is equivalent to the commutativity of the corresponding square of left adjoints, which follows from the projection formula in \'etale cohomology.  For the square on the right, fix $K_m \in D_\cons(Y_\et,R_m)$, and write $K_n = K_m \otimes_{R_m} R_n \in D_\cons(Y_\et,R_n)$.  We must show that $f_m^! K_m \otimes_{R_m} R_n \simeq f_n^! K_n$ via the natural map $f_m^! K_m \to f_m^! K_m \simeq f_n^! K_n$. This assertion is local on $X$, so we can factor $f$ as $X \stackrel{i}{\hookrightarrow} P \stackrel{g}{\to} S$ with $i$ a constructible closed immersion, and $g$ smooth of relative dimension $d$. Since $f_m^! = i_m^! \circ g_m^!$, it suffices to prove the analogous claim for $i$ and $g$ separately. Since $g_m^! = g_m^*(d)[2d]$, the assertion is immediate. For $i$, let $j:U \hookrightarrow P$ be the open complement of $i$. Using the triangle $i_* i_m^! \to \id \to j_* j^*$, it suffices to show that $j_* j^* K_m \otimes_{R_m} R_n \simeq j_* j^* K_n$, which follows from Lemma \ref{lem:projformulaconstant}.
\end{proof}

\begin{lemma}[$!$-pullback]
	\label{lem:shriekpullback}
	Let $f:X \to Y$ be a separated finitely presented map of quasiexcellent $\ell$-coprime schemes. Then $f_!:D_\cons(X_\proet,\widehat{R}) \to D_\cons(Y_\proet,\widehat{R})$ has a right adjoint $f^!$ with $f^! K \otimes_{\widehat{R}} R/\fram^n \simeq f_n^!(K \otimes_{\widehat{R}} R/\fram^n)$.
\end{lemma}
\begin{proof}
	Fix $K \in D_\cons(Y_\proet,\widehat{R})$, and let $K_n = K \otimes_R R/\fram^n \in D_\cons(Y_\et,R/\fram^n)$ be its truncation. Lemma \ref{lem:shriekpullbackclassical} gives a projective system $\{f_n^! K_n\}$ in $D_\comp(X_\proet,\widehat{R})$, and we write $f^! K := \R\lim f_n^! K_n \in D_\comp(X_\proet,\widehat{R})$. By completeness and Lemma \ref{lem:shriekpullbackclassical}, one immediately checks that $f^! K$ has the right adjointness properties. It remains to show $f^! K \otimes_{\widehat{R}} R/\fram \simeq f_1^! K_1$, which also implies $f^! K$ is constructible. This follows from the second half of Lemma \ref{lem:shriekpullbackclassical} and Lemma \ref{lem:truncationofcomplete}.
\end{proof}

\begin{lemma}[Duality]
	\label{lem:duality}
	Let $X$ be an excellent $\ell$-coprime scheme equipped with a dimension function $\delta$. Then there exists a dualizing complex $\Omega_X \in D_\cons(X_\proet,\widehat{R})$, i.e., if $D_X := \underline{\R\Hom}_X(-,\Omega_X)$, then $\id \simeq D_X^2$ on $D_\cons(X_\proet,\widehat{R})$.
\end{lemma}
\begin{proof}
	First consider the case $R = \Z_\ell$, and set $R_n = \Z/\ell^n$. Then for each $n$, there exists a unique (up to unique isomorphism)  potential dualising complex $\omega_n \in D_\cons(X_\proet,\Z/\ell^n)$, see \cite[XVII.2.1.2, XVII.5.1.1, XVII 6.1.1]{GabberEtaleCoh}. By \cite[XVII.7.1.3]{GabberEtaleCoh} and uniqueness, one may choose isomorphisms $\omega_{n+1} \otimes_{\Z/\ell^{n+1}} \Z/\ell^n$ for each $n$. Set $\omega_X = \lim \Omega_n \in D(X_\proet,\widehat{\Z_\ell})$. Then $\omega_X$ is $\ell$-adically complete, and $\omega_X \otimes_{\Z_\ell} \Z/\ell^n \simeq \omega_n$ (by a slight modification of Lemma \ref{lem:truncationofcomplete}). Lemma \ref{lem:inthom} then gives the duality isomorphism $\id \simeq D_X^2$ in this case by reduction modulo $\ell$. For general rings $R$, set $R_n := R/\fram^n$, so each $R_n$ is a $\Z/\ell^n$-algebra. Then \cite[XVII.7.1.3]{GabberEtaleCoh} shows that $\Omega_n := \omega_n \otimes_{\Z/\ell^n} R_n \in D_\cons(X_\proet,R_n)$ is dualizing. A repeat of the argument for the previous case then shows that $\Omega_X := \lim \Omega_n \in D_\cons(X_\proet,\widehat{R})$ has the required properties.
\end{proof}

\begin{remark}
	The dualizing complex constructed in Lemma \ref{lem:duality} is {\em not} the traditional dualizing complexes (as in \cite[\S XVII.7]{GabberEtaleCoh}) unless $R$ is Gorenstein. For example, when $X$ is a geometric point, the dualizing complex above is simply the ring $R$ itself, rather than the dualizing complex $\omega^\bullet_R$ coming from local duality theory. This is a reflection of our choice of working with a more restrictive class of complexes in $D_{\cons}(X_\proet,\widehat{R})$: when $X$ is a point, $D_{\cons}(X,\widehat{R}) \simeq D_\perf(R)$. 
\end{remark}

\subsection{$\Z_\ell$-,$\Q_\ell$-,$\bar{\Z}_\ell$- and $\bar{\Q}_\ell$-sheaves}
\label{ss:LadicSheaves}

Let us start by defining the relevant categories. For the moment, let $X$ be any scheme.

\begin{definition} Let $E$ be an algebraic extension of $\Q_\ell$ with ring of integers $\calO_E$. Let $E_X = \mathcal{F}_E$ and $\calO_{E,X} = \mathcal{F}_{\calO_E}$ be the sheaves associated with the topological rings $E$ and $\calO_E$ on $X_\proet$ as in Lemma \ref{l:ExSheafTopSpace}.
\end{definition}

We first identify these sheaves in terms of the familiar algebraic definitions directly on $X_\proet$:

\begin{lemma}\ 
	\label{l:IdentConstants}\begin{enumerate}
\item If $E$ is a finite extension of $\Q_\ell$ with uniformizer $\varpi$, then $\calO_{E,X} = \widehat{\calO}_E = \lim_n \calO_E/\varpi^n\calO_E$, with notation as in Subsection \ref{ss:ConsProet}.
\item In general, $\calO_{E,X} = \colim_{F\subset E} \calO_{F,X}$, where $F$ runs through finite extensions of $\Q_\ell$ contained in $E$. Moreover, $E_X = \calO_{E,X}[\ell^{-1}]$.
\end{enumerate}
\end{lemma}

\begin{proof}\begin{enumerate}
\item This follows from Lemma \ref{l:ExSheafTopSpace} and the identity
\[
\Map_\cont(S,\calO_E) = \lim_n \Map_\cont(S,\calO_E/\varpi^n\calO_E)
\]
for any profinite set $S$.
\item This follows from Lemma \ref{l:ExSheafTopSpace} and the identities
\[
\Map_\cont(S,\calO_E) = \colim_{F\subset E} \Map_\cont(S,\calO_F)\ ,
\]
\[
\Map_\cont(S,E) = \Map_\cont(S,\calO_E)[\ell^{-1}]
\]
for any profinite set $S$, which result from the compactness of $S$ and Lemma \ref{lem:compactcountabletower}. \qedhere
\end{enumerate}
\end{proof}

In this section, we abbreviate $E=E_X$ and $\calO_E = \calO_{E,X}$ if no confusion is likely to arise. First, we define lisse $E$-sheaves.

\begin{definition} A lisse $E$-sheaf (or $E$-local system) is a sheaf $L$ of $E$-modules on $X_\proet$ such that $L$ is locally free of finite rank. Similarly, a lisse $\calO_E$-sheaf, or $\calO_E$-local system, is a sheaf $M$ of $\calO_E$-modules on $X_\proet$ such that $M$ is locally free of finite rank over $\calO_E$. Let $\Loc_X(E)$, resp. $\Loc_X(\calO_E)$, be the corresponding categories. 
\end{definition}

For any discrete ring $R$, we also have the category $\Loc_X(R)$ consisting of sheaves of $R$-modules on $X_\proet$ which are locally free of finite rank over $R$. In fact, this category is just the classical one defined using $X_\et$, cf. Corollary \ref{c:LocSysDiscrete}. Our first aim is to show that our definitions coincide with the usual definitions of lisse sheaves. This amounts to the following proposition.

\begin{proposition}\label{p:LisseSheaves} \ 
\begin{enumerate}
\item If $E$ is a finite extension of $\Q_\ell$, with uniformizer $\varpi$, then the functor
\[
M\mapsto (M/\varpi^nM)_n: \Loc_X(\calO_E)\to \lim_n \Loc_X(\calO_E/\varpi^n\calO_E)
\]
is an equivalence of categories.
\item For general $E$, lisse $\calO_E$-sheaves satisfy descent for pro-\'etale covers.
\item If $X$ is qcqs, the functor
\[
\colim_{F\subset E} \Loc_X(\calO_F)\to \Loc_X(\calO_E)
\]
is an equivalence of categories, where $F$ runs through finite extensions of $\Q_\ell$ contained in $E$.
\item If $X$ is qcqs, the functor
\[
M\mapsto L=M[\ell^{-1}]: \Loc_X(\calO_E)[\ell^{-1}]\to \Loc_X(E)
\]
is fully faithful.
\item Lisse $E$-sheaves satisfy descent for pro-\'etale covers.
\item Let $L$ be a lisse $E$-sheaf on $X$. Then there is an \'etale cover $Y\to X$ such that $L|_Y$ lies in the essential image of the functor from (4).
\end{enumerate}
\end{proposition}

\begin{proof}
\begin{enumerate}
\item Easy and left to the reader.
\item This is clear.
\item For fully faithfulness, observe that one has obvious internal Hom's, which are compatible with extension of scalars. Thus, fully faithfulness follows from the observation that for an $\calO_F$-local system $M_F$ with base extensions $M_E$, $M_{F^\prime}$ for $F^\prime\subset E$ finite over $F$, $M_E = \colim M_{F^\prime}$ and $M_E(X) = \colim M_{F^\prime}(X)$ as $X$ is qcqs.

Now fix a qcqs w-contractible cover $Y\in X_\proet$, and describe $\Loc_X(\calO_E)$ in terms of descent data for $Y\to X$. Any lisse $\calO_E$-sheaf over $Y$ is necessarily trivial (and hence already defined over $\Z_\ell$), so that the categories of descent data are equivalent by fully faithfulness, using that $Y$ is still qcqs.
\item Both categories admit obvious internal Hom's, which are compatible with the functor $M\mapsto M[\ell^{-1}]$. Thus the result follows from $M[\ell^{-1}](X) = M(X)[\ell^{-1}]$, which is true as $X$ is qcqs.
\item This is clear.
\item Consider the sheaf $F$ on $X_\proet$ taking any $U\in X_\proet$ to the set of $M\in \Loc_U(\calO_E)$ with $M\otimes_{\calO_E} E = L$. We claim that $F$ is locally constant on $X_\proet$. To prove this, we can assume that $L = E^n$ is trivial. We show more precisely that in this case, $F$ is represented by (the constant sheaf associated with) the discrete set $S = \GL_n(E) / \GL_n(\calO_E)$, via mapping $g\in S$ to $M_g = g \calO_E^n$. Clearly, the map $S\to F$ is injective. Let $x\in X$ be any point. For any $M\in \Loc_X(\calO_E)$ with $M\otimes_{\calO_E} E = L$, the fibre $M_x$ is a $\calO_E$-lattice in $L_x = E^n$. Thus, by applying an element of $\GL_n(E)$, we may assume that $M_x = \calO_E^n$. This gives $n$ sections $m_{1,x},\ldots,m_{n,x}\in M_x$, which are defined over an open neighborhood of $x$; upon replacing $X$ by a neighborhood of $x$, we may assume that they are (the images of) global sections $m_1,\ldots,m_n\in M$. Similarly, one can assume that there are $n$ sections $m_1^\ast,\ldots,m_n^\ast\in M^\ast = \Hom_{\calO_E} (M,\calO_E)$ whose images in $M_x^\ast$ are the dual basis to $m_{1,x},\ldots,m_{n,x}$. This extends to an open neighborhood, so that $M = \calO_E^n$ in a neighborhood of $x$, proving surjectivity of $S\to F$.

Thus, $F$ is locally constant on $X_\proet$. In particular, it is locally classical, and therefore classical itself by Lemma \ref{l:Classical}. As there is a pro-\'etale cover $Y\to X$ with $F(Y)\neq \emptyset$, it follows that there is also an \'etale such cover, finishing the proof. \qedhere
\end{enumerate}
\end{proof}

\begin{corollary}\label{c:LisseSheavesAbelian} If $X$ is topologically noetherian, then for any morphism $f: L\to L^\prime$ in $\Loc_X(E)$, the kernel and cokernel of $f$ are again in $\Loc_X(E)$. In particular, $\Loc_X(E)$ is abelian.
\end{corollary}

\begin{proof} After passage to an \'etale cover, we may assume that there are lisse $\calO_E$-sheaves $M$, $M^\prime$ and a map $g: M\to M^\prime$ giving rise to $f: L\to L^\prime$ by inverting $\ell$. Moreover, we may assume that $X$ is connected; fix a geometric base point $\bar{x}\in X$. Then $\Loc_X(\calO_E)$ is equivalent to the category of representations of $\pi_1(X,\bar{x})$ on finite free $\calO_E$-modules. It follows that $f: L\to L^\prime$ is classified by a morphism of representations of $\pi_1(X,\bar{x})$ on finite-dimensional $E$-vector spaces. The latter category obviously admits kernels and cokernels, from which one easily deduces the claim.
\end{proof}

Next, we consider constructible sheaves. For this, we restrict to the case of topologically noetherian $X$. Note that the construction of $E_X$ is compatible with pullback under locally closed immersions, i.e. $E_Y = E_X|_Y$ for $Y\subset X$ locally closed. In the topologically noetherian case, any locally closed immersion is constructible.

\begin{definition} A sheaf $F$ of $E$-modules on $X_\proet$ is called constructible if there exists a finite stratification $\{X_i\to X\}$ such that $F|_{X_i}$ is lisse.
\end{definition}

\begin{lemma}\label{c:ConstrSheavesAbelian} For any morphism $f: F\to F^\prime$ of constructible $E$-sheaves, the kernel and cokernel of $f$ are again constructible. In particular, the category of constructible $E$-sheaves is abelian.
\end{lemma}

\begin{proof} After passing to a suitable stratification, this follows from Corollary \ref{c:LisseSheavesAbelian}.
\end{proof}

In particular, the following definition is sensible.

\begin{definition} A complex $K\in D(X_\proet,E)$ is called constructible if it is bounded and all cohomology sheaves are constructible. Let $D_\cons(X_\proet,E)$ denote the corresponding full subcategory of $D(X_\proet,E)$.
\end{definition}

\begin{corollary}\label{c:ConsTriangulated} The category $D_\cons(X_\proet,E)$ is triangulated.
\end{corollary}

\begin{proof} This follows from Lemma \ref{c:ConstrSheavesAbelian}, also observing stability of constructibility under extensions.
\end{proof}

Also recall the full triangulated subcategories $D_\cons(X_\proet,\calO_E)\subset D(X_\proet,\calO_E)$ for $E/\Q_\ell$ finite defined in Subsection \ref{ss:ConsProet}. Under our assumption that $X$ is topologically noetherian, these can be defined similarly to $D_\cons(X_\proet,E)$, cf. Proposition \ref{prop:conscons}. More precisely, we have the following proposition.

\begin{definition} For general $E$, a constructible $\calO_E$-sheaf on the topologically noetherian scheme $X$ is a sheaf $C$ of $\calO_E$-modules such that there exists a finite stratification $\{X_i\to X\}$ such that $C|_{X_i}$ is locally isomorphic to $\Lambda\otimes_{\calO_E} \calO_{E,X}$ for a finitely presented $\calO_E$-module $\Lambda$. Let $\Cons_X(\calO_E)$ be the corresponding category.
\end{definition}

\begin{proposition}\label{p:ConstrSheaves}\ 
\begin{enumerate}
\item The category of constructible $\calO_E$-sheaves is closed under kernels, cokernels, and extensions.
\item The functor
\[
\colim_{F\subset E} \Cons_X(\calO_F)\to \Cons_X(\calO_E)
\]
is an equivalence of categories, where $F$ runs through finite extensions of $\Q_\ell$.
\item If $E$ is a finite extension of $\Q_\ell$, then an object $K\in D(X_\proet,\calO_E)$ is constructible if and only if it is bounded and all cohomology sheaves are constructible.
\end{enumerate}
\end{proposition}

\begin{proof}\begin{enumerate}
\item The proof is similar to the proof of Lemma \ref{c:ConstrSheavesAbelian}.
\item The proof is similar to the proof of Proposition \ref{p:LisseSheaves} (3).
\item By (1), the set $D_\cons^\prime(X_\proet,\calO_E)$ of $K\in D(X_\proet,\calO_E)$ which are bounded with all cohomology sheaves constructible forms a full triangulated subcategory. To show $D_\cons^\prime(X_\proet,\calO_E)\subset D_\cons(X_\proet,\calO_E)$, using that $D_\cons(X_\proet,\calO_E)\subset D(X_\proet,\calO_E)$ is a full triangulated subcategory, it suffices to prove that a constructible $\calO_E$-sheaf $C$ concentrated in degree $0$ belongs to $D_\cons(X_\proet,\calO_E)$. Passing to a stratification, we can assume that $C$ is locally isomorphic to $\Lambda\otimes_{\calO_E} \calO_{E,X}$ for a finitely presented $\calO_E$-module $\Lambda$. In this case, $\Lambda$ has a finite projective resolution, giving the result.

For the converse, we argue by induction on $q-p$ that $D_\cons^{[p,q]}(X_\proet,\calO_E)\subset D_\cons^\prime(X_\proet,\calO_E)$. Thus, if $K\in D_\cons^{[p,q]}(X_\proet,\calO_E)$, it is enough to show that $\calH^q(X)$ is a constructible $\calO_E$-sheaf. This follows easily from Proposition \ref{prop:conscons}.
\end{enumerate}
\end{proof}

In particular, for general $E$, we can define $D_\cons(X_\proet,\calO_E)\subset D(X_\proet,\calO_E)$ as the full triangulated subcategory of bounded objects whose cohomology sheaves are constructible $\calO_E$-sheaves.

\begin{lemma}\label{l:constructiblealmostcompactproetale} For any $K\in D_\cons(X_\proet,\calO_E)$, the functor $\R\Hom(K,-)$ commutes with arbitrary direct sums in $D^{\geq 0}(X_\proet,\calO_E)$.
\end{lemma}

\begin{proof} The proof is the same as for Lemma \ref{lem:constructiblealmostcompact}.
\end{proof}

Although a lisse $E$-sheaf does not always admit an integral structure as a lisse $\calO_E$-sheaf, it does always admit an integral structure as a constructible $\calO_E$-sheaf.

\begin{lemma}\label{l:lisseconsintegralstructure} Let $L$ be a lisse $E$-sheaf on the topologically noetherian scheme $X$. Then there exists a constructible $\calO_E$-sheaf $C$ such that $C\otimes_{\calO_E} E = L$.
\end{lemma}

\begin{proof} First, we prove that there exists a finite stratification $\{X_i\to X\}$ such that $L|_{X_i}$  admits an $\calO_E$-lattice. By Proposition \ref{p:LisseSheaves} (6), there exists some \'etale cover $Y\to X$ such that $L|_Y$ admits an $\calO_E$-lattice. After passing to a stratification on $X$, we may assume that $Y\to X$ is finite \'etale, and that $X$ is connected; fix a geometric base point $\bar{x}\in X$ with a lift to $Y$. In that case, $L|_Y$ corresponds to a continuous representation of the profinite fundamental group $\pi_1(Y,\bar{x})$ on a finite-dimensional $E$-vector space. As $Y\to X$ is finite \'etale, this extends to a continuous representation of the profinite fundamental group $\pi_1(X,\bar{x})$ on the same finite-dimensional $E$-vector space. Any such representation admits an invariant $\calO_E$-lattice (as $\pi_1(X,\bar{x})$ is compact), giving the claim.

In particular, $L$ can be filtered as a constructible $E$-sheaf by constructible $E$-sheaves which admit $\calO_E$-structures. By Lemma \ref{l:constructiblealmostcompactproetale}, for two constructible $E$-sheaves $C$, $C^\prime$, one has
\[
\Ext^1(C[\ell^{-1}],C^\prime[\ell^{-1}]) = \Ext^1(C,C^\prime)[\ell^{-1}]\ .
\]
This implies that $L$ itself admits a $\calO_E$-structure, as desired.
\end{proof}

The following proposition shows that the triangulated category $D_\cons(X_\proet,E)$ is equivalent to the triangulated category traditionally called $D^b_c(X,E)$.

\begin{proposition}\ 
\begin{enumerate}
\item For general $E$,
\[
\colim_{F\subset E} D_\cons(X_\proet,\calO_F)\to D_\cons(X_\proet,\calO_E)
\]
is an equivalence of triangulated categories, where $F$ runs through finite extensions of $\Q_\ell$ contained in $E$.
\item The functor $D_\cons(X_\proet,\calO_E)[\ell^{-1}]\to D_\cons(X_\proet,E)$ is an equivalence of triangulated categories.
\end{enumerate}
\end{proposition}

Note that in part (2), one has an equivalence of categories without having to pass to \'etale covers of $X$.

\begin{proof}
\begin{enumerate}
\item Lemma \ref{l:constructiblealmostcompactproetale} gives full faithfulness. For essential surjectivity, one can thus reduce to the case of a constructible $\calO_E$-sheaf. In that case, the result follows from Proposition \ref{p:ConstrSheaves} (2).
\item Again, full faithfulness follows from Lemma \ref{l:constructiblealmostcompactproetale}. For essential surjectivity, one can reduce to the case of an $E$-local system $L$. In that case, the result is given by Lemma \ref{l:lisseconsintegralstructure}. \qedhere
\end{enumerate}
\end{proof}

\begin{remark} Let $\Lambda\in \{\calO_E,E\}$. Under the same assumptions as in \S \ref{ss:6Fun}, the 6 functors are defined on $D_\cons(X_\proet,\Lambda)$. Note that one can also define most of the 6 functors on $D(X_\proet,\Lambda)$. All schemes are assumed to be noetherian in the following. There are obvious $\otimes$, $\underline{\R\Hom}$ and $f_\ast$ functors for a morphism $f: Y\to X$. The functor $f_\ast$ admits a left adjoint $f^\ast: D(X_\proet,\Lambda)\to D(Y_\proet,\Lambda)$ given explicitly by $f^\ast K = f_\naive^\ast K\otimes_{f_\naive^\ast \Lambda_X} \Lambda_Y$, where $f_\naive^\ast$ denotes the naive pullback. If $f$ is \'etale or a closed immersion (or a composition of such), then $f_\naive^\ast \Lambda_X = \Lambda_Y$, so $f^\ast K = f_\naive^\ast K$ is the naive pullback. Moreover, one has the functor $j_!: D(U_\proet,\Lambda)\to D(X_\proet,\Lambda)$ for an open immersion $j: U\to X$; by composition, one gets a functor $f_!$ for a separated morphism $f: Y\to X$. If $f$ is a closed immersion, $f_! = f_\ast$ admits a right adjoint $f^!: D(X_\proet,\Lambda)\to D(Y_\proet,\Lambda)$, given as the derived functor of sections with support in $Y$.

It follows from the results of \S \ref{ss:6Fun} and the previous discussion that under the corresponding finiteness assumptions, these functors preserve constructible complexes, and restrict to the 6 functors on $D_\cons(X_\proet,\Lambda)$. In particular, one can compute these functors by choosing injective replacements in $D(X_\proet,\Lambda)$.
\end{remark}
\newpage

\section{The pro-\'etale fundamental group}
\label{sec:Pi1}

We study the fundamental group resulting from the category of locally constant sheaves on the pro-\'etale topology, and explain how it overcomes some shortcomings of the classical \'etale fundamental group for non-normal schemes. The relevant category of sheaves, together with some other geometric incarnations, is studied in \S \ref{ss:LocConsSheaves}, while the fundamental group is constructed in \S \ref{ss:pi1}. However, we first isolate a class of topological groups \S \ref{ss:noohigroups}; this class is large enough to contain the fundamental group we construct, yet tame enough to be amenable to a formalism of ``infinite'' Galois theory introduced in \S \ref{ss:infgaloistheory}.

\subsection{Noohi groups} 
\label{ss:noohigroups}

All topological groups in this section are assumed Hausdorff, unless otherwise specified. We study the following class of groups, with a view towards constructing the pro-\'etale fundamental group:

\begin{definition}
	Fix a topological group $G$. Let $G\textrm{-}\Set$ be the category of discrete sets with a continuous $G$-action, and let $F_G:G\textrm{-}\Set \to \Set$ be the forgetful functor. We say that $G$ is a {\em Noohi group}\footnote{These groups are called prodiscrete groups in \cite{Noohi}. However, they are not pro-(discrete groups), which seems to be the common interpretation of this term, so we adapt a different terminology.} if the natural map induces an isomorphism $G \simeq \Aut(F_G)$ of topological groups, where $\Aut(F_G)$ is topologized using the compact-open topology on $\Aut(S)$ for each $S \in \Set$.
\end{definition}

The basic examples of Noohi groups are:

\begin{example}
If $S$ is a set, then $G := \Aut(S)$ is a Noohi group under the compact-open topology; recall that a basis of open neighbourhoods of $1 \in \Aut(S)$ in the compact-open topology is given by the stabilizers $U_F \subset G$ of finite subsets $F \subset S$. The natural map $G \to \Aut(F_G)$ is trivially injective. For surjectivity, any $\phi \in \Aut(F_G)$ induces a $\phi_S \in G$ as $S$ is naturally a $G$-set.  It is therefore enough to show that any transitive $G$-set is a $G$-equivariant subset of $S^n$ for some $n$. Any transitive $G$-set is of the form $G/U_F$ for some finite subset $F \subset S$ finite. For such $F$, the $G$-action on the given embedding $F \hookrightarrow S$ defines a $G$-equivariant inclusion $G/U_F \to \Map(F,S)$, so the claim follows.
\end{example}

It is often non-trivial to check that a topological group with some ``intrinsic'' property, such as the property of being profinite or locally compact, is a Noohi group. To systematically deal with such issues, we relate Noohi groups to more classical objects in topological group theory: complete groups.

\begin{definition} 
	For a  topological group $G$, we define the {\em completion} $G^*$ of $G$ as the completion of $G$ for its two-sided uniformity, and write $i:G \hookrightarrow G^*$ for the natural embedding. We say $G$ is {\em complete} if $i$ is an isomorphism.
\end{definition}

We refer the reader to \cite{TopGroups} for more on topological groups, especially \cite[\S 3.6]{TopGroups} for the existence and uniqueness of completions. We will show that a topological group is Noohi if and only if it admits enough open subgroups and is complete. In preparation, we have:

\begin{lemma}
	\label{lem:autsetcomplete}
	For any set $S$, the group $\Aut(S)$ is complete for the compact-open topology.
\end{lemma}
\begin{proof}
	Let $G := \Aut(S)$, and $\eta$ be a Cauchy filter on $G$ for its two-sided uniformity. For each $F \subset S$ finite, the stabilizer $U_F \subset G$ is open, so, by the Cauchy property, we may (and do) fix some $H_F \in \eta$ such that 
	\[ H_F \times H_F \subset \{ (x,y) \in G^2 \mid xy^{-1} \in U_F \quad \textrm{and} \quad x^{-1}y \in U_F\}.\]
	Fix also some $h_{F} \in H_F$ for each such $F$. Then the above containment means: $h(f) = h_{F}(f)$ and $h^{-1}(f) = h_{F}^{-1}(f)$ for all $h \in H_F$ and $f \in F$. If $F \subset F'$, then the filter property $H_F \cap H_{F'} \neq \emptyset$ implies that $h_{F'}(f) = h_{F}(f)$, and $h_{F'}^{-1}(f) = h^{-1}_{F}(f)$ for all $f \in F$. Hence, there exist unique maps $\phi:S \to S$ and $\psi:S \to S$ such that $\phi|_F = h_F|_F$ and $\psi|_F = h_F^{-1}|_F$ for all finite subsets $F \subset S$. It is then immediate that $\phi$ and $\psi$ are mutually inverse automorphisms, and that the filter $B_\phi$ of open neighbourhoods of $\phi$ is equivalent to $\eta$, so $\eta$ converges to $\phi$, as wanted.
\end{proof}

The promised characterisation is:

\begin{proposition}\label{prop:CharNoohiGroups}
Let $G$ be a topological group with a basis of open neighbourhoods of $1 \in G$ given by open subgroups. Then there is a natural isomorphism $\Aut(F_G) \simeq G^*$. In particular, $G$ is Noohi if and only if it is complete.
\end{proposition}
\begin{proof}
	Let $\calU$ be the collection of open subgroups $U \subset G$. For $U \in \calU$ and $g \in G$, we write $T_g:G/(g U g^{-1}) \to G/U$ for the $G$-equivariant isomorphism $\alpha \cdot gUg^{-1} \mapsto \alpha g \cdot U$, i.e., right multiplication by $g$.
	
	We first construct a natural injective map $\psi:\Aut(F_G) \to G^\ast$. Given $\phi \in \Aut(F_G)$, one obtains induced automorphisms $\phi_U:G/U \to G/U$ for $U \in \calU$. Let $g_U \cdot U := \phi_U(1\cdot U)$ and $h_U \cdot U := \phi_U^{-1}(1 \cdot U)$ denote the images of the identity coset $1 \cdot U \subset G/U$ under $\phi_U$ and $\phi_U^{-1}$; here we view a coset of $U$ as a subset of $G$. We claim that $\{g_U \cdot U\}$ (indexed by $U \in \calU$) is a filter base that defines a Cauchy and shrinking filter. The finite intersection property follows immediately from $\phi$ commuting with the projection maps $G/W \to G/U$ for $W \subset U$ a smaller open subgroup. For the Cauchy property, we must check: given $U \in \calU$,  there exists $W \in \calU$ and $b \in G$ such that $g_W \cdot W \subset U \cdot b$. Fix an element $h \in G$ defining the coset $h_U \cdot U$, and let $W = hUh^{-1}$ be the displayed conjugate of $U$. Then one has a $G$-equivariant isomorphism $T_h:G/W \to G/U$ defined in symbols by $\alpha \cdot W \mapsto \alpha \cdot W h = \alpha h \cdot U$, where the last equality is an equality of subsets of $G$. The compatibility of $\phi$ with $T_h$ then shows $g_W \cdot W \cdot h = \phi_U(h \cdot U) = U$, where the last equality uses $\phi_U \circ \phi_U^{-1} = \id$;  setting $b = h^{-1}$ then gives the Cauchy property. For the shrinking property, we must show: for each $U \in \calU$, there exist $V,W,Y \in \calU$ such that $V \cdot g_W \cdot W \cdot Y \subset g_U \cdot U$; we may simply take $W = Y = U$, and $V = gUg^{-1}$ for some $g \in G$ lifting the coset $g_U \cdot U$. Let $\psi(\phi)$ be the Cauchy and shrinking filter associated to $\{g_U \cdot U\}$, i.e., $\psi(\phi)$ is the collection of open subsets $Y \subset G$ such that $g_U \cdot U \subset Y$ for some $U \in \calU$. Then $\psi(\phi) \in G^*$, which defines a map $\psi:\Aut(F_G) \to G^*$.
	
	Next, we show that $\psi$ is injective. If $\phi \in \ker(\psi)$, then $g_U \cdot U = U$ in the notation above. Now pick some $U \in \calU$ and fix some $g \in G$. The naturality of $\phi$ with respect to $T_g:G/(gUg^{-1}) \to G/U$ shows that $\phi_U(g \cdot U) = g \cdot U$, which proves that $\phi_U = \id$ for all $U \in \calU$. Any $S \in G\textrm{-}\Set$ may be written as $S = \sqcup_i G/U_i$ for suitable $U_i$, so $\phi_S = \id$ for all such $S$, and hence $\phi = \id$.

	It now suffices to show that $\Aut(F_G)$ is complete. Recall that the class of complete groups is closed inside that of all topological groups under products and passage to closed subgroups. We may realize $\Aut(F_G)$ as the equalizer of 
	\[ \equalizer{\prod_{U \in \calU} \Aut(U)}{\prod_{U,V \in \calU} \prod_{\Map_G(G/U,G/V)} \Map(G/U,G/V)}, \]
	with the maps given by precomposition and postcomposition by automorphisms. Hence, $\Aut(F_G)$ is a closed subgroup of $\prod_{U \in \calU} \Aut(S)$; as the latter is complete by Lemma \ref{lem:autsetcomplete}, the claim follows.
\end{proof}

Proposition \ref{prop:CharNoohiGroups} leads to an abundance of Noohi groups:

\begin{example}
	\label{ex:loccompactnoohi}
Any locally compact group with a fundamental system of neighbourhoods of $1$ given by open subgroups is a Noohi group. Indeed, any locally compact group is complete. Some important classes of examples are: (a) profinite groups, and (b) the groups $G(E)$ where $E$ is a local field, and $G$ is a finite type $E$-group scheme, and (c) discrete groups.
\end{example}

Perhaps surprisingly, the algebraic closure $\overline{\Q}_\ell$ of $\Q_\ell$ is also a Noohi group for the colimit topology, in contrast with the situation for the $\ell$-adic topology. In fact, one has:

\begin{example}
	\label{ex:QellbarNoohi}
	Fix a prime number $\ell$. For any algebraic extension $E$ of a $E_0 = \Q_\ell$, the group $\GL_n(E)$ is a Noohi group under the colimit topology (induced by expressing $E$ as a union of finite extensions) for all $n$. To see this, we first show that $E$ is itself Noohi. Choose a tower $E_0 \subset E_1 \subset E_2 \subset \dots \subset E$ such that $E = \colim E_i$. Let $\calU$ be the collection of all open subgroups of $\calO_E$ in the colimit topology. By Lemma \ref{lem:openinnoohi}, we must check that $\calO_E \simeq \calO_E^* := \lim_U \calO_E/U$; here we use that $\calO_E$ is abelian to identify the completion $\calO_E^*$. A cofinal collection of open subgroups is of the form $U_f$, where $f:\N \to \N$ is a function, and $U_f = \langle \ell^{f(i)} \calO_{E_i} \rangle$ is the group generated in $\calO_E$ by the displayed collection of open subgroups of each $\calO_{E_i}$. Choose $\calO_{E_i}$-linear sections $\calO_{E_{i+1}} \to \calO_{E_i}$; in the limit, this gives $\calO_{E_i}$-linear retractions $\psi_i:\calO_E \to \calO_{E_i}$ for each $i$. An element $x \in \calO_E^* = \lim_U \calO_E/U$ determines $\psi_i(x) \in \calO_{E_i}^* = \calO_{E_i}$. If the sequence $\{\psi_i(x)\}$ is eventually constant (in $\calO_E$), then there is nothing to show. Otherwise, at the expense of passing to a cofinal set of the $E_i$'s, we may assume $\psi_i(x) \in \calO_{E_i} - \calO_{E_{i-1}}$. Then one can choose a strictly increasing sequence $\{k_i\}$ of integers such that $\psi_i(x) \in \calO_{E_i}$ but $\psi_i(x) \notin \calO_{E_{i-1}} + \ell^{k_i} \calO_{E_i}$. The association $i \mapsto k_i$ gives a function $f:\N \to \N$. Choose some $x_f \in \calO_{E_j}$ for some $j$ representing the image of $x$ in $\calO_E/U_f$. Now $\psi_i(x) - \psi_i(x_f) \in \psi_i(U_f)$ for each $i$. As $\psi_i$ is $\calO_{E_i}$-linear and $f$ is strictly increasing,  it follows that $\psi_i(x) \in \calO_{E_j} + \ell^{k_i} \calO_{E_i}$ for each $i > j$; this directly contradicts the assumption on $\psi_i(x)$, proving that $\calO_E$ is Noohi. To pass from $\calO_E$ to $\GL_n(\calO_E)$, we use that the exponential $\exp:\ell^c \cdot M_n(\calO_E) \to \GL_n(\calO_E)$ (for some $c > 0$ to avoid convergence issues) is an isomorphism of uniform spaces onto an open subgroup of the target, where both sides are equipped with the two-sided uniformity associated to open subgroups of the colimit topology; see, for example, \cite[\S 18]{SchneiderpadicLieGroups} for more on the $p$-adic exponential for Lie groups.
\end{example}

The following lemma was used above:

\begin{lemma}
	\label{lem:openinnoohi}
	If a topological group $G$ admits an open Noohi subgroup $U$, then $G$ is itself Noohi.
\end{lemma}
\begin{proof}
	We must show that the natural map $G \to \Aut(F_G)$ is an isomorphism. By considering the action of both groups on the $G$-set $G/U$, it is enough to check that $U \simeq \Stab_{\Aut(F_G)}(x) =: H$, where $x \in G/U$ is the identity coset. For any $U$-set $S$, one has an associated $G$-set $\Ind_U^G(S) = (S \times G)/\sim$, where the equivalence relation is $(us,g) \sim (s,gu)$ for any $u \in U$, $s \in S$, $g \in G$, and the $G$-action is defined by $h \cdot (s,g) = (s,hg)$ for $h \in G$. This construction gives a functor $\Ind_U^G:U\textrm{-}\Set \to G\textrm{-}\Set$, left adjoint to the forgetful functor. For any $U$-set $S$, there is an obvious map $\Ind_U^G(S) \to G/U$ of $G$-sets defined by the projection $S \times G \to G$. The fibre of this map over $x \in G/U$ is the $U$-set $S$. In particular, there is an induced $H$-action on $S$. One checks that this gives a continuous map $H \to \Aut(F_U)$ extending the obvious map $U \to \Aut(F_U)$. Now the essential image of $\Ind_U^G$ generates $G\textrm{-}\Set$ under filtered colimits: for any open subgroup $V \subset U$, one has $\Ind_U^G(U/V) = G/V$. Thus, $H \to \Aut(F_U)$ is injective. On the other hand, as $U$ is Noohi, the composite $U \to H \to \Aut(F_U)$ is an isomorphism, and hence so is $U \to H$.
\end{proof}

\subsection{Infinite Galois theory}
\label{ss:infgaloistheory}

Infinite Galois theory gives conditions on a pair $(\calC,F:\calC \to \Set)$, consisting of a category $\calC$ and a functor $F$, to be equivalent to a pair $(G\textrm{-}\Set,F_G:G\textrm{-}\Set \to \Set)$ for $G$ a topological group. Here, an object $X\in \calC$ is called connected if it is not empty (i.e., initial), and for every subobject $Y\subset X$ (i.e., $Y\buildrel\sim\over\to Y\times_X Y$), either $Y$ is empty or $Y=X$.

\begin{definition}
	An {\em infinite Galois category}\footnote{A similar definition is made in \cite{Noohi}. However, the conditions imposed there are too weak: The category of locally profinite sets with open continuous maps as morphisms satisfies all axioms imposed in \cite{Noohi}, but does not arise as $G\textrm{-}\Set$ for any Noohi group $G$. There are even more serious issues, see Example \ref{ex:higmanstone}.} is a pair $(\calC, F:\calC \to \Set)$ satisfying:
	\begin{enumerate}
	\item $\calC$ is a category admitting colimits and finite limits.
	\item Each $X\in \calC$ is a disjoint union of connected objects.
	\item $\calC$ is generated under colimits by a set of connected objects.
	\item $F$ is faithful, conservative, and commutes with colimits and finite limits.
	\end{enumerate}
	The {\em fundamental group} of $(\calC,F)$ is the topological group  $\pi_1(\calC,F) : = \Aut(F)$, topologized by the compact-open topology on $\Aut(S)$ for any $S \in \Set$.
\end{definition}

\begin{example}
	If $G$ is a Noohi group, then $(G\textrm{-}\Set,F_G)$ is a Noohi category, and $\pi_1(\calC,F) = G$.
\end{example}

However, not all infinite Galois categories arise in this way:

\begin{example} 
	\label{ex:higmanstone}
	There are cofiltered inverse systems $G_i$, $i\in I$, of free abelian groups with surjective transition maps such that the inverse limit $G=\lim G_i$ has only one element, cf. \cite{HigmanStone}.  One can define an infinite Galois category $(\calC,F)$ as the $2$-categorical direct limit of $G_i\textrm{-}\Set$. It is not hard to see that $\pi_1(\calC,F) = \lim G_i$, which has only one element, yet $F: \calC\to \Set$ is not an equivalence.
\end{example}

This suggests the following definition.

\begin{definition} An infinite Galois category $(\calC,F)$ is {\em tame} if for any connected $X\in \calC$, $\pi_1(\calC,F)$ acts transitively on $F(X)$.
\end{definition}

The main result is:

\begin{theorem} 
	\label{thm:infGaloisthy}
	Fix an infinite Galois category $(\calC,F)$ and a Noohi group $G$. Then
\begin{enumerate}
	\item $\pi_1(\calC,F)$ is a Noohi group.
	\item There is a natural identification of $\Hom_{\cont}(G, \pi_1(\calC,F))$ with the groupoid of functors $\calC \to G\textrm{-}\Set$ that commute with the fibre functors.
	\item If $(\calC,F)$ is tame, then $F$ induces an equivalence $\calC \simeq \pi_1(\calC,F)\textrm{-}\Set$.
\end{enumerate}
\end{theorem}

\begin{proof} Fix a set $X_i\in \calC$, $i\in I$, of connected generators. As in the proof of Proposition \ref{prop:CharNoohiGroups}, one gets that $\pi_1(\calC,F)$ is the closed subgroup of $\prod_i \Aut(F(X_i))$ of those elements compatible with all maps between all $X_i$. It follows that $\pi_1(\calC,F)$ is closed in a Noohi group, and thus a Noohi group itself, proving (1). Also, part (2) is completely formal.

It remains to prove part (3). As $(\calC,F)$ is tame, we know that for any connected $X\in \calC$, $\pi_1(\calC,F)$ acts transitively on $F(X)$. It follows that the functor $\calC\to \pi_1(\calC,F)\textrm{-}\Set$ preserves connected components. By interpreting maps $f: Y\to X$ in terms of their graph $\Gamma_f\subset Y\times X$, one sees that the functor is fully faithful. For essential surjectivity, take any open subgroup $U\subset \pi_1(\calC,F)\textrm{-}\Set$. As $\pi_1(\calC,F)$ is closed in $\prod_i \Aut(F(X_i))$, there are finitely many $X_{i_j}$, with points $x_j\in F(X_{i_j})$, $j\in J$, such that $U$ contains the subgroup $U^\prime$ of $\pi_1(\calC,F)$ fixing all $x_j$. The element $\pi_1(\calC,F)/U^\prime\in \pi_1(\calC,F)\textrm{-}\Set$ is the image of some $X_{U^\prime}\in \calC$, as it can be realized as the connected component of $\prod_j X_{i_j}$ containing $(x_j)_j$. As colimits exist in $\calC$, the quotient $X_U = X_{U^\prime} / U$ exists in $\calC$. As colimits are preserved by $F$, it follows that $F(X_U) = \pi_1(\calC,F)/U$, as desired.
\end{proof}

Proposition \ref{prop:CharNoohiGroups} is useful to study Noohi groups under limits. Similarly,  Theorem \ref{thm:infGaloisthy} is useful for studying the behaviour under colimits. For example, one has coproducts:

\begin{example}
	\label{ex:noohicoprod}
The category of Noohi groups admits coproducts. Indeed, if $G$ and $H$ are Noohi groups, then we can define an infinite Galois category $(\calC,F)$ as follows: $\calC$ is the category of triples $(S,\rho_G,\rho_H)$ where $S \in \Set$, while $\rho_G:G \to \Aut(S)$ and $\rho_H:H \to \Aut(S)$ are continuous actions on $S$ of $G$ and $H$ respectively, and $F:\calC \to \Set$ is given by $(S,\rho_G,\rho_H) \mapsto S$. One has an obvious map from the coproduct of abstract groups $G\ast H$ to $\pi_1(\calC,F)$, from which one can see that $(\calC,F)$ is tame. Then $G \ast^N H := \pi_1(\calC,F)$ is a coproduct of $G$ and $H$ in the category of Noohi groups.
\end{example}

\begin{remark} It may be true that general infinite Galois categories are classified by certain group objects $G$ in the pro-category of sets. One has to assume that the underlying pro-set of this group can be chosen to be strict, i.e. with surjective transition maps. In that case, one can define $G\textrm{-}\Set$ as the category of sets $S$ equipped with an action of $G$ (i.e., equipped with a map $G\times S\to S$ in the pro-category of sets that satisfies the usual axioms). It is easy to verify that $G\textrm{-}\Set$ forms an infinite Galois category under the strictness hypothesis. To achieve uniqueness of $G$, one will again have to impose the condition that there are enough open subgroups. Fortunately, the infinite Galois categories coming from schemes will be tame, so we do not worry about such esoteric objects!
\end{remark}

\subsection{Locally constant sheaves}
\label{ss:LocConsSheaves}

Fix a scheme $X$ which is locally topologically noetherian. We will consider the following classes of sheaves on $X_\proet$:

\begin{definition} Fix $F \in \Shv(X_\proet)$. We say that $F$ is 
\begin{enumerate}
	\item {\em locally constant} if there exists a cover $\{X_i \to X\}$ in $X_\proet$ with $F|_{X_i}$ constant. 
	\item {\em locally weakly constant} if there exists a cover $\{Y_i \to X\}$ in $X_\proet$ with $Y_i$ qcqs such that $F|_{Y_i}$ is the pullback of a classical sheaf on the profinite set $\pi_0(Y_i)$.
	\item {\em a geometric covering} if $F$ is an \'etale $X$-scheme satisfying the valuative criterion of properness.
\end{enumerate}
We write $\Loc_X$, $w\Loc_X$ and $\Cov_X$  for the corresponding full subcategories of $\Shv(X_\proet)$.
\end{definition}

\begin{remark}
	The objects of $\Loc_X, w\Loc_X$ and $\Cov_X$ are classical. This is evident for $\Cov_X$, and follows from Lemma \ref{l:Classical} for $\Loc_X$ and $w\Loc_X$.
\end{remark}

\begin{remark}
Any $Y \in \Cov_X$ is quasiseparated: $Y$ is locally topologically noetherian by Lemma \ref{l:topnoeth}. Hence, we can write $Y$ as a filtered colimit of its qcqs open subschemes. This remark will be used without comment in the sequel.
\end{remark}

\begin{remark}
	\label{rmk:wlocclosure}
	Fix an $F \in \Shv(X_\proet)$. One checks that $F \in w\Loc_X$ if and only if for any qcqs w-contractible $Y \in X_\proet$, the restriction $F|_Y$ is classical, and the pullback of its pushforward to $\pi_0(Y)$. For such $Y$, pushforward and pullback along $\Shv(Y_\et) \to \Shv(\pi_0(Y)_\et)$, as well as the inclusion $\Shv(Y_\et) \subset \Shv(Y_\proet)$,  commute with all colimits and finite limits; thus, the subcategory $w\Loc_X \subset \Shv(X_\proet)$ is closed under all colimits and finite limits.
\end{remark}

\begin{example}
	\label{ex:coveringsfield}
If $X = \Spec(k)$ is the spectrum of a field, then $\Loc_X = w\Loc_X = \Cov_X = \Shv(X_\et)$. Indeed, this is immediate from the observation that any separable closure of $k$ provides a connected w-contractible cover of $X$. More generally, the same argument applies to any finite scheme of Krull dimension $0$: the underlying reduced scheme is a finite product of fields.
\end{example}

\begin{lemma}
	\label{lem:localgebraicspace}
If $Y$ is a qcqs scheme, and $F \in \Shv(Y_\proet)$ is the pullback of a classical sheaf on $\pi_0(Y)$, then
\begin{enumerate}
	\item $F$ is representable by an algebraic space \'etale over $Y$.
	\item $F$ satisfies the valuative criterion for properness.
	\item The diagonal $\Delta:F \to F \times_Y F$ is a filtered colimit of clopen immersions.
\end{enumerate}
\end{lemma}
\begin{proof}
	As any classical sheaf on a profinite set is a filtered colimit of finite locally constant sheaves, so $F = \colim_i U_i$ is a filtered colimit of finite \'etale $Y$-schemes $U_i$ indexed by a filtered poset $I$. In particular, (2) and (3) are clear. (1) follows by expressing $F$ as the quotient of the \'etale equivalence relation on $\sqcup_i U_i$ given by the two evident maps $\sqcup_{i \leq j} U_i \to \sqcup_i U_i$: the identity $U_i \to U_i$ and the transition map $U_i \to U_j$.
\end{proof}

\begin{remark}
	The algebraic space $F$ in Lemma \ref{lem:localgebraicspace} need {\em not} be quasiseparated over $Y$. For example, we could take $F$ to be the pullback of two copies of $\pi_0(Y)$ glued along a non-quasicompact open subset. This phenomenon does not occur for the geometric coverings we consider as $X$ is topologically noetherian.
\end{remark}

\begin{lemma}
	\label{lem:coveringhenselian}
If $Y$ is a henselian local scheme, then any $F \in \Cov_X$ is a disjoint union of finite \'etale $Y$-schemes.
\end{lemma}
\begin{proof}
	If $Z \subset Y$ is the closed point, then $F|_Z = \sqcup_i Z_i$ with $Z_i \to Z$ connected finite \'etale schemes by Example \ref{ex:coveringsfield}. Let $\tilde{Z_i} \to Y$ be the (unique) connected finite \'etale $Y$-scheme lifting $Z_i \to Z$. Then the henselian property ensures that $F(\tilde{Z_i}) = F|_Z(Z_i)$, so one finds a canonical \'etale map $\phi:\sqcup_i \tilde{Z_i} \to F$ inducing an isomorphism after restriction to $Z$. As the image of $\phi$ is closed under generalization, and because each point of $F$ specializes to a point of the special fibre (by half of the valuative criterion), one checks that $\phi$ is surjective. To check $\phi$ is an isomorphism, one may assume $Y$ is strictly henselian, so $\tilde{Z_i} = Y$ for each $i$. Then each $\tilde{Z_i} \to F$ is an \'etale monomorphism, and hence an open immersion. Moreover, these open immersions are pairwise disjoint (by the other half of the valuative criterion), i.e., that $\tilde{Z_i} \cap \tilde{Z_j} = \emptyset$ as subschemes of $F$ for $i \neq j$. Then $\sqcup_i \tilde{Z_i}$ gives a clopen decomposition for $F$, as wanted.
\end{proof}

\begin{lemma} 
	\label{lem:locallagree}
	One has $\Loc_X = w\Loc_X = \Cov_X$  as subcategories of $\Shv(X_\proet)$. 
\end{lemma}
\begin{proof}
	The property that a sheaf $F \in \Shv(X_\proet)$ lies in $\Loc_X$, $w\Loc_X$, or $\Cov_X$ is Zariski local on $X$. Hence, we may assume $X$ is topologically noetherian. It is clear that $\Loc_X \subset w\Loc_X$. For $w\Loc_X \subset \Cov_X$,  fix some $F \in w\Loc_X$. Lemma \ref{lem:localgebraicspace} and descent show that $F$ satisfies the conclusion of Lemma \ref{lem:localgebraicspace}. To get $F$ to be a scheme, note that $F$ is quasiseparated as $X$ is topologically noetherian, and thus the diagonal of $F$ is a clopen immersion by quasicompactness. In particular, $F$ is separated, and thus a scheme: any locally quasifinite and separated algebraic space over $X$ is a scheme, see \cite[Tag 0417]{StacksProject}.
	
	We next show $\Cov_X \subset w\Loc_X$, i.e., any geometric covering $F \to X$ is locally weakly constant. In fact, it suffices to show the following: for any qcqs $U \in X_\et$ and  map $\phi:U \to F$, one may, locally on $X_\et$, factor $\phi$ as $U \to L \to F$ with $L$ finite locally constant. Indeed, this property implies that $F|_Y$ is a filtered colimit of finite locally constant sheaves for any w-contractible $Y \in X_\proet$, which is enough for local weak constancy. As $F$ is a filtered colimit of qcqs open subschemes, this property follows from Lemma \ref{lem:coveringhenselian} and spreading out.

	It remains to check $w\Loc_X = \Loc_X$. Choose $F \in w\Loc_X$ and a qcqs w-contractible cover $Y \to X$ such that $F|_Y = \pi^* G$ for some $G \in \Shv(\pi_0(Y)_\et)$, where $\pi:Y \to \pi_0(Y)$ is the natural map. We must show that $G$ is locally constant. Let $X_\eta \subset X$ be the union of the finite collection of generic points of $X$, and write $Y_\eta \subset Y$ for the corresponding fibre. Let $\overline{Y_\eta}$ be a qcqs w-contractible cover of $Y_\eta$. Then we obtain a diagram
	\[ \xymatrix{ \overline{Y_\eta} \ar[r]^-{\psi} \ar[d]^-a & \pi_0(\overline{Y_\eta}) \ar[d]^-{\pi_0(a)} \\
	Y_\eta \ar[r]^-{\phi} \ar[d]^-b & \pi_0(Y_\eta) \ar[d]^-{\pi_0(b)}  \\
	Y \ar[r]^-{\pi} \ar[d]^-c & \pi_0(Y) \\
			X } \]
			Each connected component of $Y$ is a strict henselisation of $X$, and thus contains a point lying over a point of $X_\eta$, i.e., a point of $Y_\eta$. This shows that $\pi_0(b)$ is surjective. The map $\pi_0(a)$ is clearly surjective. Write $f:\pi_0(\overline{Y_\eta}) \to \pi_0(Y)$ for the composite surjection. As $Y$ is w-contractible, the space $\pi_0(Y)$ is extremally disconnected. Thus, it is enough to show that $f^* G$ is locally constant. As $\psi_* \psi^* \simeq \id$ as endofunctors of $\Shv(\pi_0(\overline{Y_\eta}))$, it is enough to show $\psi^* f^* G$ is locally constant. By the commutativity of the diagram, the latter sheaf coincides with the restriction of $F$ to $\overline{Y_\eta}$. But $\overline{Y_\eta}$ is a w-contractible cover of $X_\eta$, so the claim follows from the equality $w\Loc_{X_\eta} = \Loc_{X_\eta}$ of Example \ref{ex:coveringsfield}.
\end{proof}

\begin{remark}
	If $X$ is Nagata, one may prove a more precise form of Lemma \ref{lem:locallagree}: there exists a pro-\'etale cover $\{U_i \to X\}$ with $U_i$ connected such that $F|_{U_i}$ is constant for any $F \in w\Loc_X$. To see this, choose a stratification $\{X_i \to X\}$ with $X_i$ affine, normal and connected; this is possible as $X$ is Nagata. Let $V_i$ be the henselisation of $X$ along $X_i$, and $U_i \to V_i$ be a connected pro-(finite \'etale) cover that splits all finite \'etale $V_i$-schemes. Then one checks that $\{U_i \to X\}$ satisfies the required properties using the Gabber-Elkik theorem (which identifies $V_{i,f\et} \simeq X_{i,f\et}$), and the observation that each $F \in w\Loc_{X_i}$ is a disjoint union of finite \'etale $X_i$-schemes by normality.
\end{remark}

\begin{remark}
	\label{rmk:locforlfi}
For an arbitrary scheme $Y$, define $\Loc_Y$, $w\Loc_Y$ and $\Cov_Y$ as above, except that objects of $w\Loc_Y$ and $\Cov_Y$ are required to be quasiseparated. Then the proof of Lemma \ref{lem:locallagree} shows that one always has $\Loc_Y \subset w\Loc_Y = \Cov_Y$, and the inclusion is an equivalence if $Y$ has locally a finite number of irreducible components.
\end{remark}

\begin{example} Some topological condition on the scheme $X$ (besides being connected) is necessary to make coverings well-behaved. Indeed, consider the following example. Let $T$ be topological space underlying the adic space corresponding to the closed unit disc over $\Q_p$. This is a spectral space, so there is some ring $A$ for which $X=\Spec A$ is homeomorphic to $T$. All arguments in the following are purely topological, so we will argue on the side of $T$. The origin $0\in T$ is a closed point which admits no generalizations, yet $T$ is connected. One has open subsets $T_1,T_{1/2},\ldots\subset T$, where $T_{1/i}$ is the annulus with outer radius $1/i$ and inner radius $1/(i+1)$.

The open subset $U = T\setminus \{0\}\subset T$ defines an object of $\Cov_X$. Indeed, it is clearly \'etale, and it satisfies the valuative criterion of properness, as $0$ does not admit nontrivial generalizations. One can show that $U$ also defines an object of $w\Loc_X$, however it is not hard to see that $U$ does not define an object of $\Loc_X$. We claim that the disjoint union of $U$ with an infinite disjoint union of copies of $X$ belongs to $\Loc_X$. This will prove that $\Loc_X$ is not closed under taking connected components, so that it cannot define an infinite Galois category.

Consider the pro-\'etale cover $Y\to X$ which has connected components $\pi_0(Y) = \{0,1,1/2,1/3,\ldots\}$, with connected components $Y_0 = \{0\}$, $Y_{1/i} = U_{1/i}$; it is easy to see how to build $Y$ as an inverse limit. The pullback of $U$ to $Y$ is the pullback of the sheaf $F_U$ on $\pi_0(Y)$ concentrated on $\{1,1/2,1/3,\ldots\}$. To show that the disjoint union of $U$ with an infinite disjoint union of copies of $X$ belongs to $\Loc_X$, it is enough to show that the disjoint union of $F_U$ with an infinite constant sheaf on $\pi_0(Y)$ is again an infinite constant sheaf. This boils down to some easy combinatorics on the profinite set $\pi_0(Y)$, which we leave to the reader.
\end{example}

\subsection{Fundamental groups}
\label{ss:pi1}

In this section, we assume that $X$ is locally topologically noetherian and connected, and we fix a geometric point $x$ of $X$ with $\ev_x:\Loc_X \to \Set$ being the associated functor $F \mapsto F_x$.

\begin{lemma}\label{lem:LocXTame}
The pair $(\Loc_X,\ev_x)$ is an infinite Galois category. Moreover, it is tame.
\end{lemma}
\begin{proof}
	For the first axiom, Remark \ref{rmk:wlocclosure} shows that $w\Loc_X \subset \Shv(X_\proet)$ is closed under colimits and finite limits. For the second axiom, we use $\Cov_X$. Indeed, any $Y\in \Cov_X$ is locally topologically noetherian, so that its connected components are clopen. Any clopen subset of $Y$ defines another object of $\Cov_X$. It is a connected object. Indeed, assume $Y\in \Cov_X$ is connected as a scheme, and $Z\to Y$ is some map in $\Cov_X$. The image of $Z$ is open and closed under specializations (by the valuative criterion of properness). As $Y$ is locally topologically noetherian, open implies locally constructible, and in general, locally constructible and closed under specializations implies closed. Thus, the image of $Z$ is open and closed, and thus either empty or all of $Y$. The third axiom regarding things being a set (as opposed to a proper class) is left to the reader. For the last axiom, we use $\Loc_X$. As any pair of points of $X$ is linked by a chain of specializations, one checks that $\ev_x$ is conservative and faithful on $\Loc_X$. As $\ev_x$ is given by evaluation on a connected w-contractible object, it commutes with all colimits and all limits in $\Shv(X_\proet)$, and hence with all colimits and finite limits in $\Loc_X$.

Finally, we have to prove tameness. This comes down to showing that $\pi_1$ is large enough, i.e. we have to construct enough paths in $X$. Thus, choose some connected cover $Y\to X$, and any two geometric points $y_1,y_2$ above $x$. They give rise to topological points $\bar{y}_1,\bar{y}_2\in Y$. As $Y$ is locally topologically noetherian, we can find a paths $\bar{y}_1 = \bar{z}_0, \bar{z}_1, \ldots, \bar{z}_n = \bar{y}_2$ of points in $Y$ such that for each $i=0,\ldots,n-1$, $\bar{z}_{i+1}$ is either a specialization or a generalization of $\bar{z}_i$. Fix geometric points $z_i$ above $\bar{z}_i$. By projection, we get geometric points $x_i$ of $X$, lying above topological points $\bar{x}_i\in X$.

For each $i$, choose a valuation ring $R_i$ with algebraically closed fraction field, together with a map $\Spec R_i\to Y$ such that the special and generic point are (isomorphic to) $z_i$ and $z_{i+1}$ (or the other way around); we fix the isomorphisms. The valuation rings $R_i$ induce maps $\Spec R_i\to X$, which induce isomorphisms of fibre functors $\ev_{x_i}\simeq \ev_{x_{i+1}}$. By composition, we get an isomorphism of fibre functors
\[
\ev_x = \ev_{x_0}\simeq \ev_{x_1}\simeq \ldots \simeq \ev_{x_n} = \ev_x\ ,
\]
i.e. an automorphism $\gamma\in \pi_1(\Loc_X,\ev_x)$ of the fibre functor $\ev_x$. By construction, we have $\gamma(y_1) = y_2$, showing that $(\Loc_X,\ev_x)$ is tame.
\end{proof}

Tameness implies that the following definition is robust:

\begin{definition}
The {\em pro-\'etale fundamental group} is defined as $\pi_1^\proet(X,x) := \Aut(\ev_x)$; this group is topologized using the compact-open topology on $\Aut(S)$ for any $S \in \Set$.
\end{definition}

We now relate $\pi^\proet_1(X_,x)$ to other fundamental groups. First, the profinite completion of $\pi_1^\proet(X,x)$ recovers the \'etale fundamental group $\pi_1^\et(X,x)$, as defined in \cite{SGA1}:

\begin{lemma}
	Let $G$ be a profinite group. There is an equivalence
	\[\underline{\Hom}_{\cont}(\pi_1^\proet(X,x), G) \simeq (B\calF_G)(X_\proet)\ .\]
\end{lemma}
Here, $\underline{\Hom}(H,G)$ for groups $G$ and $H$ denotes the groupoid of maps $H\to G$, where maps between $f_1,f_2: H\to G$ are given by elements $g\in G$ conjugating $f_1$ into $f_2$.
\begin{proof}
Both sides are compatible with cofiltered limits in $G$, so we reduce to $G$ finite. In this case, one easily checks that both sides classify $G$-torsors on $X_\proet$.
\end{proof}

To understand representations of $\pi_1^\proet(X,x)$, we first construct ``enough'' objects in $\Loc_X$. 

\begin{construction}
	The equivalence $\Cov_X \simeq \Loc_X \simeq \pi_1^\proet(X,x)\textrm{-}\Set$ implies that for each open subgroup $U \subset \pi_1^\proet(X,x)$, there exists a canonically defined $X_U \in \Cov_X$ with a lift of the base point $x \in X_{U,\proet}$ corresponding to $\pi_1^\proet(X,x)/U \in \pi_1^\proet(X,x)\textrm{-}\Set$ in a base point preserving manner. Moreover, as $X_U$ is itself a locally topologically noetherian scheme, one has $\pi_1^\proet(X_U,x) = U$ as subgroups of $\pi_1^\proet(X,x)$.
\end{construction}

Write $\Loc_{X_\et}$ for the category of locally constant sheaves on $X_\et$, viewed as a full subcategory of $\Loc_X$. The difference between $\Loc_{X_\et}$ and $\Loc_X$ can be explained group theoretically:

\begin{lemma}
	Under $\Loc_X \simeq \pi_1^\proet(X,x)\textrm{-}\Set$, the full subcategory $\Loc_{X_\et} \subset \Loc_X$ corresponds to the full subcategory of those $S \in \pi_1^\proet(X,x)\textrm{-}\Set$ where an open subgroup acts trivially.
\end{lemma}
\begin{proof}
	Fix $S \in \pi_1^\proet(X,x)\textrm{-}\Set$, and assume an open subgroup $U \subset \pi_1^\proet(X,x)$ acts trivially on $S$. Then the corresponding locally constant sheaf is trivialized by passage to $X_U$, which is an \'etale cover of $X$. Conversely, fix some $F \in \Loc_{X_\et}$ with fibre $S$, and consider the sheaf $G = \Isom(F,\underline{S})$ on $X_\proet$. The \'etale local trivializability of $F$ shows that $G$ is an $\underline{\Aut(S)}$-torsor on $X_\et$; here we use that $\underline{\Aut(S)} = \Aut(\underline{S}) = \nu_* \calF_{\Aut(S)}$ on $X_\et$ as each $U \in X_\et$ has a discrete $\pi_0$. Then $G \in \Cov_X$, so there exists an open subgroup $U \subset \pi_1^\proet(X,x)$ and a factorization $X_U \to G \to X$. By construction, $F|_G$ is constant, so $U = \pi_1^\proet(X_U,x)$ acts trivially on the fibre $F_x$.
\end{proof}

The pro-(discrete group) completion of $\pi_1^\proet(X,x)$ covers the fundamental pro-group defined in \cite[\S 6]{SGA3ExpX}:

\begin{lemma}
	\label{lem:pi1proetsga3}
	Let $G$ be a discrete group. There is an equivalence
	\[\underline{\Hom}_{\cont}(\pi_1^\proet(X,x),G) \simeq (BG)(X_\et)\ .\]
\end{lemma}
\begin{proof}
	Any continuous map $\rho:\pi_1^\proet(X,x) \to G$ gives a $G$-torsor in $\pi_1^\proet(X,x)\textrm{-}\Set$, and hence an object of $(BG)(X_\proet)$; one then simply observes that $(BG)(X_\proet) = (BG)(X_\et)$ as $G$ is discrete. Conversely, any $G$-torsor $F$ on $X_\et$ defines a fibre preserving functor $G\textrm{-}\Set \to \Loc_X$ simply by pushout, and hence comes from a continuous map $\pi_1^\proet(X,x) \to G$.
\end{proof}

Lemma \ref{lem:pi1proetsga3} shows that the inverse limit of the pro-group defined in \cite[\S 6]{SGA3ExpX} is large enough, i.e., the limit topological group has the same discrete group representations as the defining pro-group.

We now explain why the group $\pi_1^\proet(X,x)$ is richer than its pro-(discrete group) completion: the latter does not know the entirety of $\Loc_X(\Q_\ell)$ (see Example \ref{ex:delignenodalcurve}), while the former does. The main issue is that $\Loc_X(\Q_\ell)$ is not $\Loc_X(\Z_\ell)[\frac{1}{\ell}]$, but rather the global sections of the stack associated to the prestack $U \mapsto \Loc_U(\Z_\ell)[\frac{1}{\ell}]$ on $X_\proet$.

\begin{lemma}
	\label{lem:pi1ladicrep}
	For a local field $E$, there is an equivalence of categories $\Rep_{E,\cont}(\pi^\proet_1(X,x)) \simeq \Loc_X(E)$.
\end{lemma}
\begin{proof}
The claim is clear if $E$ is replaced by $\calO_E$ as $\GL_n(\calO_E)$ is profinite. Now given a continuous representation $\rho:\pi_1^\proet(X,x) \to \GL_n(E)$, the group $U = \rho^{-1} \GL_n(\calO_E)$ is open in $\pi_1^\proet(X,x)$, and hence defines a pointed covering $X_U \to X$ with $\pi_1^\proet(X_U,x) = U$. The induced representation $\pi_1^\proet(X_U,x) \to \GL_n(\calO_E)$ defines some $M \in \Loc_{X_U}(\calO_E)$, and hence an $M' \in \Loc_{X_U}(E)$; one checks that $M'$ comes equipped with descent data for $X_U \to X$, and hence comes from a unique $N(\rho) \in \Loc_X(E)$. Conversely, fix some $N \in \Loc_X(E)$, viewed as an $\calF_{\GL_n(E)}$-torsor for suitable $n$. For each $S \in \GL_n(E)\textrm{-}\Set$, one has an induced representation $\rho_S:\calF_{\GL_n(E)} \to \calF_{\Aut(S)}$. The pushout of $N$ along $\rho_S$ defines an $N_S \in \Loc_X$ with stalk $S$. This construction gives a functor $\GL_n(E)\textrm{-}\Set \to \Loc_X$ which is visibly compatible with the fibre functor. As $\GL_n(E)$ is Noohi, one obtains by Galois theory the desired continuous homomorphism $\rho_N:\pi_1^\proet(X,x) \to \GL_n(E)$.
\end{proof}

\begin{remark}
	By Example \ref{ex:QellbarNoohi}, the conclusion of Lemma \ref{lem:pi1ladicrep} also applies to any algebraic extension $E/\Q_\ell$ with the same proof.
\end{remark}

The following example is due to Deligne:

\begin{example} 
	\label{ex:delignenodalcurve}
	Let $Y$ be a smooth projective curve of genus $\geq 1$ over an algebraically closed field. Fix three distinct points $a,b,x \in Y$, and paths $\ev_a \simeq \ev_x \simeq \ev_b$ between the corresponding fibre functors on $\Loc_Y$. Let $X = Y/\{a,b\}$ be the nodal curve obtained by identifying $a$ and $b$ on $Y$; set $\pi:Y \to X$ for the natural map, and $c = \pi(a) = \pi(b)$. Then one has two resulting paths $\ev_x \simeq \ev_c$ as fibre functors on $\Loc_X$, and hence an element $\lambda \in \pi_1^\proet(X,x)$ corresponding to the loop.  Fix a local field $E$, a rank $n$ local system $M \in \Loc_Y(E)$ with monodromy group $\GL_n(\calO_E)$ with $n \geq 2$,  and a generic non-integral matrix $T \in \GL_n(E)$. Then identifying the fibres $M_a$ and $M_b$ using $T$ (using the chosen paths) gives a local system $M \in \Loc_X(E)$ where $\lambda$ acts by $T$; a similar glueing construction applies to local systems of sets, and shows $\pi_1^\proet(X,x) \simeq \pi_1^\proet(Y,y) \ast^N \lambda^\Z$ in the notation of Example \ref{ex:noohicoprod}. In particular, the monodromy group of $L$ is $\GL_n(E)$. Assume that the corresponding continuous surjective representation $\rho:\pi_1^\proet(X,x) \to \GL_n(E)$ factors through the pro-(discrete group) completion of $\pi_1^\proet(X,x)$, i.e., the preimage of each open subgroup $W \subset \GL_n(E)$ contains an open normal subgroup of $\pi_1^\proet(X,x)$. Then $U := \rho^{-1}(\GL_n(\calO_E))$ is open, so it contains an open normal $V \subset U$.  By surjectivity, the image $\rho(V)$ is a closed normal subgroup of $\GL_n(E)$ lying in $\GL_n(\calO_E)$. One then checks that $\rho(V) \subset \G_m(\calO_E)$, where $\G_m \subset \GL_n$ is the center. In particular, the induced representation $\pi_1^\proet(X,x) \to \PGL_n(E)$ factors through a discrete quotient of the source. It follows that $L$ has abelian monodromy over an \'etale cover of $X$, which is clearly false: the corresponding statement fails for $M$ over $Y$ by assumption. 
\end{example}

Example \ref{ex:delignenodalcurve} is non-normal. This is necessary:

\begin{lemma}
If $X$ is geometrically unibranch, then $\pi_1^\proet(X,x) \simeq \pi_1^\et(X,x)$.
\end{lemma}

\begin{proof}
	One first checks that irreducible components are clopen in any locally topologically noetherian geometrically unibranch scheme: closedness is clear, while the openness is local, and may be deduced by a specialization argument using the finiteness of generic points on a topologically noetherian scheme. It follows by connectedness that $X$ is irreducible. Moreover, by the same reasoning, any connected $Y \in \Cov_X$ is also irreducible. Let $\eta \in X$ be the generic point, and let $Y_\eta \to \eta$ be the generic fibre. Then $Y_\eta$ is connected by irreducibility of $Y$, and hence a finite scheme as $\Loc_\eta$ is the category of disjoint unions of finite \'etale covers of $\eta$. In particular, $\pi:Y \to X$ has finite fibres. We claim that $\pi$ is finite \'etale;  this is enough for the lemma as $\pi_1^\et(X,x)$ classifies finite \'etale covers of $X$. For the proof, we may assume $X$ quasicompact. Now any quasicompact open $U \subset Y$ containing $Y_\eta$ is finite \'etale over a quasicompact open $V \subset X$, and hence includes all points over $V$. Expanding $U$ to include the fibre over some point in the complement of $V$ and proceeding inductively (using that $X$ is topologically noetherian) then shows that $Y$ is itself quasicompact. Then $\pi$ is proper and \'etale, whence finite \'etale.
\end{proof}

\begin{remark}
	The fundamental group $\pi_1^{dJ}(X,x)$ for rigid-analytic spaces over a non-archimedean valued field constructed by de Jong \cite{deJongPi1} has some similarities with the group $\pi_1^\proet(X,x)$ introduced above. In fact, in the language of our paper, the category $\Cov_X^{dJ}$ of disjoint unions of ``coverings'' in the sense of \cite[Definition 2.1]{deJongPi1} is a tame infinite Galois category by \cite[Theorem 2.10]{deJongPi1}. Thus, the corresponding group $\pi_1^{dJ}(X,x)$ is a Noohi group; by \cite[Theorem 4.2]{deJongPi1}, the category of continuous finite dimensional $\Q_\ell$-representations of $\pi_1^{dJ}(X,x)$ recovers the category of lisse $\Q_\ell$-sheaves (and the same argument also applies to $\overline{\Q}_\ell$-sheaves by Example \ref{ex:QellbarNoohi}). However, it is not true that a naive analogue of $\Cov_X^{dJ}$ for schemes reproduces the category $\Cov_X$ used above: the latter is larger. Note, moreover, that \cite[Lemma 2.7]{deJongPi1} is incorrect: the right hand side is a monoid, but need not be a group. As far as we can tell, this does not affect the rest of \cite{deJongPi1}.
\end{remark}

The following definition is due to Gabber:

\begin{remark}\label{rem:GabberPi1}
	Assume $Y$ is a connected scheme with locally a finite number of irreducible components. Then one may define the {\em weak fundamental groupoid} $w\pi(Y)$ as the groupoid-completion of the category of points of $Y_\et$ (which is equivalent to the category of connected w-contractible objects in $Y_\proet$). For each such point $y \in w\pi(Y)$, one has a corresponding automorphism group $w\pi(Y,y)$; as $Y$ is connected, the resulting functor $B(w\pi(Y,y)) \to w\pi(Y)$ is an equivalence. One can think of elements of $w\pi(Y,y)$ as paths (of geometric points) in $Y$, modulo homotopy.

Note that the definition of $\pi_1^\proet(Y,y)$ works in this generality, cf. Remark \ref{rmk:locforlfi}. Moreover, each $F \in \Loc_Y$ restricts to functor $w\pi(Y) \to \Set$, so the fibre $\ev_y(F)$ has a canonical $w\pi(Y,y)$-action. This construction gives a map $w\pi(Y,y) \to \pi_1^\proet(Y,y)$; by the proof of Lemma \ref{lem:LocXTame}, this map has dense image. If we equip $w\pi(Y,y)$ with the induced topology, then continuous maps from $\pi_1^\proet(Y,y)$ to Noohi groups $G$ are the same as continuous maps from $w\pi(Y,y)$ to $G$. In particular, one can describe lisse $\Q_\ell$- (resp. $\overline{\Q}_\ell$-) sheaves in terms of continuous representations of $w\pi(Y,y)$ on finite-dimensional $\Q_\ell$- (resp. $\overline{\Q}_\ell$-) vector spaces.
\end{remark}

\bibliography{proetale}

\begin{thebibliography}{SGA72b}

\bibitem[AT08]{TopGroups}
Alexander Arhangelskii and Mikhail Tkachenko.
\newblock {\em Topological groups and related structures}, volume~1 of {\em
  Atlantis Studies in Mathematics}.
\newblock Atlantis Press, Paris, 2008.

\bibitem[Del80]{DeligneWeil2}
Pierre Deligne.
\newblock La conjecture de {W}eil. {II}.
\newblock {\em Inst. Hautes \'Etudes Sci. Publ. Math.}, (52):137--252, 1980.

\bibitem[dJ95]{deJongPi1}
A.~J. de~Jong.
\newblock \'{E}tale fundamental groups of non-{A}rchimedean analytic spaces.
\newblock {\em Compositio Math.}, 97(1-2):89--118, 1995.
\newblock Special issue in honour of Frans Oort.

\bibitem[Eke90]{Ekedahl}
Torsten Ekedahl.
\newblock On the adic formalism.
\newblock In {\em The {G}rothendieck {F}estschrift, {V}ol.\ {II}}, volume~87 of
  {\em Progr. Math.}, pages 197--218. Birkh\"auser Boston, Boston, MA, 1990.

\bibitem[Gab94]{GabberAffineBC}
Ofer Gabber.
\newblock Affine analog of the proper base change theorem.
\newblock {\em Israel J. Math.}, 87(1-3):325--335, 1994.

\bibitem[Gle58]{GleasonExtremallyDisconnected}
Andrew~M. Gleason.
\newblock Projective topological spaces.
\newblock {\em Illinois J. Math.}, 2:482--489, 1958.

\bibitem[GR03]{GabberRamero}
Ofer Gabber and Lorenzo Ramero.
\newblock {\em Almost ring theory}, volume 1800 of {\em Lecture Notes in
  Mathematics}.
\newblock Springer-Verlag, Berlin, 2003.

\bibitem[Gro64]{SGA3ExpX}
A.~Grothendieck.
\newblock Caract\'erisation et classification des groupes de type
  multiplicatif.
\newblock In {\em Sch\'emas en {G}roupes ({S}\'em. {G}\'eom\'etrie
  {A}lg\'ebrique, {I}nst. {H}autes \'{E}tudes {S}ci., 1963/64)}, pages Fasc. 3,
  Expos\'e 10, 50. Inst. Hautes \'Etudes Sci., Paris, 1964.

\bibitem[Hoc69]{HochsterSpectral}
M.~Hochster.
\newblock Prime ideal structure in commutative rings.
\newblock {\em Trans. Amer. Math. Soc.}, 142:43--60, 1969.

\bibitem[HS54]{HigmanStone}
Graham Higman and A.~H. Stone.
\newblock On inverse systems with trivial limits.
\newblock {\em J. London Math. Soc.}, 29:233--236, 1954.

\bibitem[Hub93]{HuberAffineBC}
Roland Huber.
\newblock \'{E}tale cohomology of {H}enselian rings and cohomology of abstract
  {R}iemann surfaces of fields.
\newblock {\em Math. Ann.}, 295(4):703--708, 1993.

\bibitem[ILO]{GabberEtaleCoh}
Luc Illusie, Yves Laszlo, and Fabrice Orgogozo.
\newblock {\em Travaux de {G}abber sur l'uniformisation locale et la
  cohomologie etale des schemas quasi-excellents}.
\newblock To appear in {\em Asterisque}.

\bibitem[Jan88]{JannsenContsCoh}
Uwe Jannsen.
\newblock Continuous \'etale cohomology.
\newblock {\em Math. Ann.}, 280(2):207--245, 1988.

\bibitem[LO08]{LaszloOlsson}
Yves Laszlo and Martin Olsson.
\newblock The six operations for sheaves on {A}rtin stacks. {I}. {F}inite
  coefficients.
\newblock {\em Publ. Math. Inst. Hautes \'Etudes Sci.}, (107):109--168, 2008.

\bibitem[Lur09]{LurieHTT}
Jacob Lurie.
\newblock {\em Higher topos theory}, volume 170 of {\em Annals of Mathematics
  Studies}.
\newblock Princeton University Press, Princeton, NJ, 2009.

\bibitem[Lur11]{LurieDAGXII}
Jacob Lurie.
\newblock Proper morphisms, completions, and the {G}rothendieck existence
  theorem.
\newblock Available at \url{http://math.harvard.edu/~lurie/}, 2011.

\bibitem[Mil78]{MillerIntegraldeRham}
Edward~Y. Miller.
\newblock de {R}ham cohomology with arbitrary coefficients.
\newblock {\em Topology}, 17(2):193--203, 1978.

\bibitem[Noo08]{Noohi}
Behrang Noohi.
\newblock Fundamental groups of topological stacks with the slice property.
\newblock {\em Algebr. Geom. Topol.}, 8(3):1333--1370, 2008.

\bibitem[Oli72]{Olivier}
Jean~Pierre Olivier.
\newblock Fermeture int\'egrale et changements de base absolument plats.
\newblock In {\em Colloque d'{A}lg\`ebre {C}ommutative ({R}ennes, 1972), {E}xp.
  {N}o. 9}, pages 13 pp. Publ. S\'em. Math. Univ. Rennes, Ann\'ee 1972. Univ.
  Rennes, Rennes, 1972.

\bibitem[Ols11]{OlssonnonabelianpHT}
Martin~C. Olsson.
\newblock Towards non-abelian {$p$}-adic {H}odge theory in the good reduction
  case.
\newblock {\em Mem. Amer. Math. Soc.}, 210(990):vi+157, 2011.

\bibitem[Sch11]{SchneiderpadicLieGroups}
Peter Schneider.
\newblock {\em {$p$}-adic {L}ie groups}, volume 344 of {\em Grundlehren der
  Mathematischen Wissenschaften [Fundamental Principles of Mathematical
  Sciences]}.
\newblock Springer, Heidelberg, 2011.

\bibitem[Sch13]{ScholzepHT}
Peter Scholze.
\newblock p-adic {H}odge theory for rigid-analytic varieties.
\newblock {\em Forum of Mathematics, Pi}, 1, 2013.

\bibitem[SGA71]{SGA1}
{\em Rev\^etements \'etales et groupe fondamental}.
\newblock Springer-Verlag, Berlin, 1971.
\newblock S{\'e}minaire de G{\'e}om{\'e}trie Alg{\'e}brique du Bois Marie
  1960--1961 (SGA 1), Dirig{\'e} par Alexandre Grothendieck. Augment{\'e} de
  deux expos{\'e}s de M. Raynaud, Lecture Notes in Mathematics, Vol. 224.

\bibitem[SGA72a]{SGA4Tome1}
{\em Th\'eorie des topos et cohomologie \'etale des sch\'emas. {T}ome 1:
  {T}h\'eorie des topos}.
\newblock Lecture Notes in Mathematics, Vol. 269. Springer-Verlag, Berlin,
  1972.
\newblock S{\'e}minaire de G{\'e}om{\'e}trie Alg{\'e}brique du Bois-Marie
  1963--1964 (SGA 4), Dirig{\'e} par M. Artin, A. Grothendieck, et J. L.
  Verdier. Avec la collaboration de N. Bourbaki, P. Deligne et B. Saint-Donat.

\bibitem[SGA72b]{SGA4Tome2}
{\em Th\'eorie des topos et cohomologie \'etale des sch\'emas. {T}ome 2}.
\newblock Lecture Notes in Mathematics, Vol. 270. Springer-Verlag, Berlin,
  1972.
\newblock S{\'e}minaire de G{\'e}om{\'e}trie Alg{\'e}brique du Bois-Marie
  1963--1964 (SGA 4), Dirig{\'e} par M. Artin, A. Grothendieck et J. L.
  Verdier. Avec la collaboration de N. Bourbaki, P. Deligne et B. Saint-Donat.

\bibitem[SGA73]{SGA4Tome3}
{\em Th\'eorie des topos et cohomologie \'etale des sch\'emas. {T}ome 3}.
\newblock Lecture Notes in Mathematics, Vol. 305. Springer-Verlag, Berlin,
  1973.
\newblock S{\'e}minaire de G{\'e}om{\'e}trie Alg{\'e}brique du Bois-Marie
  1963--1964 (SGA 4), Dirig{\'e} par M. Artin, A. Grothendieck et J. L.
  Verdier. Avec la collaboration de P. Deligne et B. Saint-Donat.

\bibitem[Sta]{StacksProject}
The {S}tacks {P}roject.
\newblock Available at \url{http://stacks.math.columbia.edu}.

\end{thebibliography}
\end{document}